\documentclass[10pt]{amsart}
\usepackage{amsmath}
\usepackage[utf8]{inputenc}
\usepackage{color}
\usepackage{verbatim}
\usepackage{mathtools}
\usepackage{amstext}
\usepackage{amsthm}
\usepackage{amssymb}
\usepackage[unicode=true,
 bookmarks=false,
 breaklinks=false,pdfborder={0 0 1},backref=section,colorlinks=false]
 {hyperref}

\hypersetup{
 colorlinks,linkcolor={red!60!black},citecolor={green!60!black},urlcolor={blue!60!black}}

\makeatletter
\numberwithin{equation}{section}
\numberwithin{figure}{section}
\theoremstyle{plain}

\numberwithin{equation}{section}
\usepackage{lineno}
\usepackage{footnote}
\usepackage{enumitem}
\usepackage{setspace}
\usepackage{tikz}
\usepackage{cite}\usepackage{thm-restate}
\makesavenoteenv{tblr}

\newtheorem{theorem}{Theorem}[section]
\newtheorem{proposition}[theorem]{Proposition}
\newtheorem{lemma}[theorem]{Lemma}
\newtheorem{remark}[theorem]{Remark}
\theoremstyle{remark}

\global\long\def\R{\mathbf{\mathbb{R}}}%
\global\long\def\norm#1{\|#1\|}%
\global\long\def\wt#1{\widehat{#1}}%
\global\long\def\tilde#1{\widetilde{#1}}%
\global\long\def\bbR{\mathbf{\mathbb{R}}}%
\global\long\def\bbN{\mathbf{\mathbb{N}}}%

\makeatother

\title[Interior Stefan problem]{Melting and freezing rates of the radial interior Stefan problem in two dimension}
\author{Jeongheon Park}
\email{jse05002@kaist.ac.kr}
\address{Department of Mathematical Sciences, Korea Advanced Institute of Science
and Technology, 291 Daehak-ro, Yuseong-gu, Daejeon 34141, Korea}

\keywords{ blow-up construction,
asymptotic profile}
\subjclass[2020]{35B44, 35Q55, 37K40}

\begin{document}

\begin{abstract}
We consider the interior Stefan problem under radial symmetry in two dimension. A water ball surrounded by ice undergoes melting or freezing. We construct a discrete family of global-in-time solutions, both melting and freezing scenarios. The evolution of the free boundary, represented by the radius of the water ball,  $\lambda(t)$ exhibits exponential convergence to a limiting radius value $\lambda_\infty > 0$, characterized by the asymptotic expression
\[
\lambda(t) = \lambda_\infty + (1 - \lambda_\infty)\, e^{-\frac{\lambda_k}{\lambda_\infty^2} t + o_{t \to \infty}(1)},
\]
where $\lambda_k$ stands for the $k$-th Dirichlet eigenvalue of the Laplacian on the unit disk for any $k\in \mathbb{N}$. Our approach draws inspiration from the research conducted by Hadžić and Raphaël~\cite{Had-Raph} concerning the exterior radial Stefan problem, which involves an ice ball is surrounded by water. In contrast, the bounded geometry in our setting leads to scenario results in a non-degenerate spectrum, leading to distinctly different long-term behavior. These solutions for each $k$ remain stable under perturbations of co-dimension~$k - 1$.
\end{abstract}

\maketitle

\section{Introduction}
We consider the interior Stefan problem in two dimension 
\begin{equation*}
u:\Omega(t)\subset\R^2\rightarrow\mathbb{R},
\end{equation*}
\begin{equation}\label{stefan}
\begin{cases}
\partial_{t}u-\Delta u=0 \text{ in } \Omega(t),\\
\partial_{n}u=V_{\partial\Omega(t)} \text{ on } \partial\Omega(t),\\
u=0\text{ on } \partial\Omega(t).
\end{cases}
\end{equation}
The Stefan problem addresses the phase transition that takes place at the boundary that separates two phases. The problem involves determining the unknowns: a time-variant domain $\Omega(t) \subset \bbR^2$ with a dynamic boundary and a temperature function $u(t,x)$ defined in $\Omega(t)$. We consider that at time $t$, the region $\Omega(t)$ denotes the water phase, whereas its complement $\Omega(t)^{c}$ consists of ice. We examine the situation in which the water region $\Omega(t)$ is enclosed, and this scenario is named an \textit{interior} Stefan problem. The free boundary condition, which is the normal velocity $V_{\partial\Omega(t)}$ of the moving boundary $\partial\Omega(t)$, is dictated by the normal derivative of $u$, represented as $\partial_n u$.

This paper aims to explore melting and freezing rates in the context of the one-phase classical interior Stefan problem. Our study is inspired by a similar study of the exterior Stefan problem by Had\v{z}\'ic and Rapha\"el \cite{Had-Raph}. In the exterior Stefan problem, an ice ball is surrounded by water, as so the domain $\Omega(t)$ is the exterior of the region outside a ball. The authors in \cite{Had-Raph}  construct a discrete sequence of melting and freezing rates for the radial exterior Stefan problem. This construction provides a rigorous justification for the formal asymptotic rates previously established by Herrero and Vel\'azquez~\cite{Herrero-Velaz}.

A key geometric difference between the interior and exterior problems lies in the evolution of the bounded region occupied by one of the phases (water or ice), rather than in the domain of the solution itself. In the interior problem, the region $\Omega(t)$ represents the water phase, which shrinks over time in the freezing regime. In contrast, in the exterior problem analyzed by Had\v{z}\'ic and Rapha\"el \cite{Had-Raph}, the complement $\Omega(t)^c$ corresponds to the ice region, which shrinks during the melting process. From the perspective of the shrinking behavior of the bounded region, we construct a sequence of solutions so that the interior problem exhibits global-in-time dynamics, in contrast to the exterior problem as in \cite{Had-Raph}.

\subsection{ Well-posedness results}
The Stefan problem is known to be well-posed only for data satisfying some geometric condition, such as the Taylor sign condition. In general, the dynamics may develop a singularity in a short time. In one dimension, the free boundary $\partial \Omega(t)$ may develop a jump discontinuity of the free boundary in finite time, provided that the initial data are large near the initial boundary $\partial \Omega(0)$; see \cite{Sherman, L.Chayes, M.Herrero, L.Chayes.I.Kim1, L.Chayes.I.Kim2}. Formally, this phenomenon is attributed to the presence of highly active water particles near the boundary, which induce arbitrarily high rates of freezing. These jump discontinuities in the free boundary are a primary source of non-existence and non-uniqueness of weak solutions (see also the discussion in \cite{Delaure}).

However, under suitable constraints on the initial data, a substantial body of works has addressed the local well-posedness and regularity of solutions to the Stefan problem. Kamenomostskaja \cite{KamenomostskayaS.L.}, Friedman \cite{Friedman1}, and Ladyzhenskaya, Solonnikov, and Ural’tseva \cite{Solonnikov} formulated the notion of weak solutions to \eqref{stefan} and constructed weak solutions. For results concerning the existence and regularity of free boundaries in this setting, we refer to \cite{Caffarelli1, Caffarelli2, Caff-Evans, Caffa-Friedman, Friedman3, Kinderlehrer, Kinderlehrer2}, and references therein.

An important framework for studying \eqref{stefan} is the theory of viscosity solutions. Kim \cite{Kim.I.C} established the existence of viscosity solutions for using the maximum principle. The regularity theory of viscosity solutions was studied in the works \cite{Caffalleli3, Athanasopoulos1, Athanasopoulos2, Friedman1} and \cite{Choikim2}.
Recently, alternative solution notions have been proposed to address different aspects of the Stefan problem. For the supercooled Stefan problem, Choi, Kim, and Kim~\cite{Kim-Kim2024,ChoiKim2024} established the global existence of maximal weak solutiodns in general dimensions, based on a stochastic optimal transport formulation with free-target constraints. See also ~\cite{ChuKim2025} for a \emph{non-local} variant model. They formulated a weak solution framework using martingale optimal transport and established global existence and stability results. These works demonstrate the ongoing efforts to develop generalized solution concepts for various physical scenarios beyond the classical formulation.
The well-posedness theory for \textit{classical solutions} is of interest. In the classical works of Meirmanov \cite{Meirmanov} and Hanzawa \cite{Hanzawa}, local-in-time existence of smooth solutions was established using analytic tools such as von Mises-type coordinate transforms and Nash–Moser iteration. These approaches involve a loss of derivatives and require high regularity of the initial data.
More recently, Had\v{z}i\'c and Shkoller \cite{Had-Shk} developed a robust local well-posedness theory in Sobolev spaces $\mathcal{C}([0,T]; H^2)$ without derivative loss. Their analysis is based on energy estimates and elliptic regularity and crucially exploits a geometric condition known as the \textit{Taylor sign condition}, which ensures that the normal derivative of the temperature at the free boundary satisfies $\partial_n u(x)\ge \lambda > 0$ for all $x\in\partial\Omega$.  As in the construction of Had\v{z}\'ic and Rapha\"el \cite{Had-Raph} of melting and freezing dynamics,  for the radial Stefan problem, the well-posedness is simple and one can understand solutions are \textit{classical solutions} as the Taylor sign condition propagates as long as the radius $\lambda(t)>0 $.

In this work, we are also in the radial symmetric case. In our setting, the domain $\Omega(t)$ for water to \eqref{stefan} is a ball of radius $\lambda(t)$ centered at the origin, that is,
\[
\Omega(t) = B_{\lambda(t)} = \{ x \in \mathbb{R}^2 : |x| < \lambda(t) \}.
\]
Under this radial symmetry, the Stefan problem \eqref{stefan} reduces to the following one-dimensional system:
\begin{equation}
\begin{cases}
\partial_{t}u-\partial_r^2u-\frac{1}{r}\partial_ru=0,\text{ in } \Omega(t)\\
\partial_ru(t,\lambda(t))=-\dot{\lambda}(t),\\
u(t,\lambda(t))=0,\\
u(0,\cdot)=u_{0},\quad \lambda(0)=\lambda_{0}.
\end{cases}\label{stefanradial}
\end{equation}
The Taylor sign condition holds automatically in the radial setting. Hence, the system \eqref{stefanradial} is locally well posed in $H^2$. That is, for $(u_0, \lambda_0) \in H^2 \times \mathbb{R}_+^*$ with radially symmetric initial data $u_0$, there exists a unique solution 
\[
(u(t), \lambda(t)) \in \mathcal{C}([0,T); H^2) \times \mathcal{C}^1([0,T); \mathbb{R}_+^*)
\]
for some $T > 0$. According to the local theory developed in \cite{Had-Shk}, the solution $u(t, r)$ can be continued as long as the Taylor sign condition is satisfied. In this setting, the free boundary condition $u_r(t, \lambda(t)) = -\dot{\lambda}(t) \ne 0$ forces the Taylor sign condition to propagate. In this work, we construct global-in-time solutions by tracking the modulation law for $\lambda(t)$, ensuring that $\dot{\lambda}(t)$ remains nonzero and sign-definite throughout the evolution. Using the scaling symmetry of the system, we normalize the initial configuration by choosing $\lambda_0 = 1$, so that $\Omega(0) = B_1(0)$ without loss of generality.

\subsection{Main result}
To state our main theorem, we first recall the eigenvalue problem for the radial Laplacian on the two-dimensional unit ball \( B_1(0) \) with Dirichlet boundary condition:
\begin{align}
\begin{cases}
-\Delta_{\text{rad}} u = -\partial_{rr} u - \frac{1}{r} \partial_r u = \lambda u, & \text{in } B_1(0), \\
u(1) = 0, & \text{on } \partial B_1(0).
\end{cases} \label{eq:1}
\end{align}
It is well known from standard ODE theory that problem \eqref{eq:1} admits a sequence of eigenvalue–eigenfunction pairs \( (\lambda_j, \eta_j) \) indexed by \( j \in \mathbb{N} \). Due to radial symmetry, each eigenvalue \( \lambda_j \) is simple. Moreover, the collection \( \{ \eta_j \}_{j=1}^{\infty} \) forms an orthonormal basis of the Hilbert space \( L^2_{\text{rad}}(B_1(0)) \), and the eigenvalues satisfy \( 0 < \lambda_1 < \lambda_2 < \cdots \). For further spectral properties, we refer to \eqref{renormalization of unperturbed eigenfunctions}–\eqref{increment} below.

We are now ready to state our main theorem.
\begin{theorem}
\label{main theorem} 
    For each $k \in \mathbb{N}$, let $\lambda_k$ be the $k$-th eigenvalue of the operator in \eqref{eq:1}. Then there exist global-in-time solutions to the interior Stefan problem \eqref{stefanradial}, with melting or freezing rates associated with $\lambda_k$. More precisely, there exist a small constant $\delta_k > 0$ and smooth initial data $u_0$, with $\|u_0\|_{L^1(B_1(0))} \leq \delta_k$, such that the corresponding solution $(u, \lambda) \in \mathcal{C}([0, \infty), H^2) \times \mathcal{C}^1([0, \infty), \mathbb{R}_+^*)$ to \eqref{stefanradial} exists globally in time.

    The terminal radius $\lambda_\infty$ is given by
    \begin{equation} \label{terminal value}
    \lim_{t \to \infty} \lambda(t) = \lambda_\infty > 0, \qquad \lambda_\infty = \sqrt{1 + \frac{1}{\pi} \int_{\Omega_0} u_0 \, dx},
    \end{equation}
    and the melting or freezing rate satisfies
    \begin{equation} \label{asymptoticoflambda}
    \lambda(t) = \lambda_\infty + (1 - \lambda_\infty) e^{-\frac{\lambda_k}{\lambda_\infty^2} t + o_{t \to \infty}(1)}.
    \end{equation}

    Moreover, for each $k \in \mathbb{N}$, the solution $(u, \lambda)$ is stable under perturbations of codimension $k - 1$. In addition, if $\int_{\Omega_0} u_0 \, dx < 0$, the solution corresponds to the \emph{freezing} regime, while if $\int_{\Omega_0} u_0 \, dx > 0$, it corresponds to the \emph{melting} regime.

\end{theorem}

\begin{remark}[Comparison with radial exterior Stefan problem]
In \cite{Had-Raph}, the authors considered the exterior Stefan problem
where $\Omega(t)=\{|x|\ge\lambda(t)\}$ and constructed a sequence
of melting and freezing rates. Our analysis is closely inspired by
their work, but with a crucial difference in spectral structure that leads to fundamentally different dynamics. In the unbounded domain $\{|x|\ge\lambda(t)\}$, the renormalized linearized
operator deviates from the fundamental solution of the Laplacian in two
dimensions, which is given by $\log r$. In contrast, in the bounded domain $\{|x|\leq\lambda(t)\}$ used in our analysis, the Laplacian
admits a discrete spectrum, and the renormalized operator deviates
from these discrete modes. In the melting regime studied in \cite{Had-Raph}, the authors considered the initial data with a small $H^1$ norm and showed that $\lambda(t) \to 0$ as $t \to T$ for some finite $T > 0$. As the shrinking bounded region $\Omega(t)^c$ collapses to a point, the domain asymptotically fills the entire space $\mathbb{R}^2$. Consequently, all eigenvalues collapse to zero and the corresponding eigenfunctions tend to the fundamental solution $\log r$ of the Laplacian on $\mathbb{R}^2$, in an appropriate sense. This degeneracy of the eigenvalue induces a logarithmic correction in the evolution law of $\lambda(t)$; see \cite[Section 1.5]{Had-Raph} for details. However, in our case, the precompactness of the domain of equation \eqref{eq:1} rules out such spectral degeneracy. This key difference gives rise to a leading-order linear term in the modulation equations, which plays a central role in generating the global-in-time dynamics; see Proposition~\ref{Modulation equation proposition}.
\end{remark}

\begin{remark}
[Tracing the rate of $\lambda(t)$] Let $\tilde{u}$ be a
radial solution to the fixed boundary problem of the heat equation on the ball $B_{R}(0)$: 
\[
\begin{cases}
\partial_{t}\tilde{u}-\Delta\tilde{u}=0\quad\text{in}\quad[0,\infty)\times B_{R}(0)\\
\tilde{u}(t,R)=0\quad\text{for}\quad t\in[0,\infty)\\
\tilde{u}(0,r)=\tilde{u}_{0}\quad\text{in}\quad B_{R}(0)
\end{cases}
\]
If the initial data $\tilde{u}_{0}$ is chosen to be the $k$-th eigenfunction
of the spectral problem the equation \eqref{eq:1} with eigenvalue $\lambda_{k}>0$,
then standard parabolic theory yields the decay estimate \[\tilde{u}(t)\sim e^{-t\frac{\lambda_{k}}{R^{2}}}\tilde{u}(0).\]
By an analogy, we expect that a solution $u$ to \eqref{stefanradial} with initial data $\tilde{u}_{0}$ satisfies 
\[
u(t) \sim e^{-t \frac{\lambda_k}{\lambda(t)^2}},
\quad \text{so that} \quad
\partial_r u(t) = -\dot{\lambda}(t) \sim e^{-t \frac{\lambda_k}{\lambda(t)^2}}.
\]
In general, tracking the evolution of the free boundary in such problems is a highly nontrivial task. To this end, we employ modulation analysis, a technique developed for the construction of type~\text{II} blow-up solutions in both parabolic and dispersive settings. This method has proven powerful in capturing refined asymptotics near singularities and has been used extensively in a variety of contexts, including the critical and supercritical wave map equations (see, e.g., \cite{Raphaelwave, GhoulWaveMap, CollotWave, KriegerBeyond, GaoKrieger}) and the parabolic Keller–Segel system and its variants \cite{CollotKS, HouNguyenKS, CollotAniso, WaldronHMF, BiernatHMF}. Inspired by these developments, we adapt the modulation framework to the interior Stefan problem, allowing us to construct global-in-time solutions with precise asymptotics for the interface evolution.
\end{remark}

\begin{remark}
[Higher-dimensional case] Theorem~\ref{main theorem}
states the result on dimension two. However, the argument extends to higher dimensions in the same manner. In particular, the radial eigenfunctions
$\{u_{j}\}_{j\in\mathbb{N}}$ in $L_{rad}^{2}(\mathbb{R}^{n})$ of
the Laplacian with Dirichlet boundary conditions on the unit ball $B_1(0) \subset \bbR^n$ are given by 
\[
u_{j}(r)=r^{1-\frac{n}{2}}J_{-\frac{2-n}{2}}(rr_{n,j}),
\]
where $J_{-\frac{2-n}{2}}$ is the Bessel function of the first kind of order
$-\frac{2-n}{2}$, and $\{r_{n,j}\}_{j\in\mathbb{N}}$ denotes the increasing sequence
of positive zeros of $J_{-\frac{2-n}{2}}$.
Consequently, the universal constant $\lambda_k$ appearing in the asymptotic law for $\lambda(t)$ in \eqref{asymptoticoflambda} is replaced by $r_{n,k}^{2}$ for each $k \in \bbN$. A detailed derivation in higher dimensions follows from the same spectral and modulation framework and is omitted here for brevity.
\end{remark}

\subsection{Notations}
For quantities $A\in\mathbb{R}$ and $B\geq0$, we denote $A \lesssim B$ or $A=O(B)$ if $|A|\leq CB$ holds for some implicit constant $C$. For $A,B\geq0$, we say $A \sim B$ when $A \lesssim B$ and $B \lesssim A$. Similarly, for $A\geq 0$ and $B\in \mathbb{R}$, we write $A\gtrsim B$ if $B\lesssim A$.

We use $(\cdot, \cdot)$ for the usual dot product in $\mathbb{R}^n$:
For $\vec{a}=(a_1, \cdots ,a_n)$ and $\vec{b}=(b_1, \cdots, b_n)$, we define 
\begin{equation*}
    (\vec{a},\vec{b}) \coloneqq \sum_{i=1}^n a_ib_i.
\end{equation*}

We approximate the exact law by $b\approx -\frac{\lambda_{s}}{\lambda}$. Hence, we introduce the linearized operator $\mathcal{H}_{b}$,
\begin{equation*}
    \mathcal{H}_b \coloneqq -\Delta + b\Lambda,
\end{equation*}
on $B_1(0)$ with the Dirichlet boundary condition. Here, we denote $\Lambda$ as a scaling operator
\begin{equation*}
    \Lambda \coloneqq r\partial_r.
\end{equation*}
Then, as in \cite{Had-Raph}, $\mathcal{H}_{b}$ is self-adjoint with respect to the inner product $\langle \cdot,\cdot \rangle_b$ on $B_{1}(0)$:
\begin{equation}
\rho_{b}(z)\coloneqq e^{-\frac{1}{2}by^{2}},\qquad\langle f,g\rangle_{b}\coloneqq \int_{0}^{1}f(y)g(y)\rho_{b}(y)y\ dy.\label{weightedL2product}
\end{equation}
Also, we denote the associated norms by 
\begin{align*}
 & \norm{f}_{L_{b}^{2}}\coloneqq \left(\int_{0}^{1}f(y)^{2}\rho_{b}(y)ydy\right)^{\frac{1}{2}},\quad\norm{f}_{\dot{H}_{b}^{1}}\coloneqq \norm{\nabla f}_{L_{b}^{2}},\\
 & \text{ and }\norm{f}_{H_{b}^{1}}\coloneqq \left(\norm{\nabla f}_{L_{b}^{2}}^{2}+\norm{f}_{L_{b}^{2}}^{2}\right)^{\frac{1}{2}}.
\end{align*}
We denote Hilbert spaces $L_{b}^{2}$ and $H_{b}^{1}$ by 
\begin{equation*}
L_{b}^{2}\coloneqq \{f:B_{1}(0)\rightarrow\mathbb{R}:f\ \text{is radial with}\ \norm{f}_{L_{b}^{2}}<\infty\},
\end{equation*}
and
\begin{equation*}
H_{b}^{1}\coloneqq \{f:B_{1}(0)\rightarrow\mathbb{R}\ |\ f\ \text{is radial with}\ \norm{f}_{H_{b}^{1}}<\infty\ \text{and}\ f(1)=0\}.
\end{equation*}
The pointwise boundary condition $f(1)=0$ in the definition of
$H_{b}^{1}$ is defined in the trace sense. For $b=0,$ $L_{0}^{2}=L_{rad}^{2}(B_{1}(0))$. Similarly, we define $H_b^2$.

\subsection{Strategy of the proof}
We begin by renormalizing the solution $u$ into a rescaled $v$, transforming the free boundary problem \eqref{stefanradial} into one with a fixed boundary:
\begin{equation}\label{renormalised solution}
v(s,y)\coloneqq u(t,r) |_{t=t(s)},\quad \frac{ds}{dt} \coloneqq \frac{1}{\lambda^{2}},\quad y\coloneqq \frac{r}{\lambda}.
\end{equation}
Under this change of variables, the equation becomes
\begin{equation}\label{renormalised stefan}
\begin{cases}
\partial_{s}v-\frac{\lambda_{s}}{\lambda}\Lambda v-\Delta v=0,\quad  y \leq 1,\\
v(s,1)=0,\\
v_{y}(s,1)=-\frac{\lambda_{s}}{\lambda}.
\end{cases}
\end{equation}
We introduce the dynamical parameter $a$ to be
\begin{equation*}
a\coloneqq -\frac{\lambda_{s}}{\lambda},
\end{equation*}
and our aim is to track the evolution of this parameter.

\textbf{1. Linear analysis:} 
We first study the spectral properties of the linearized operator
\begin{equation}
\mathcal{H}_{b}=-\Delta+b\Lambda,\quad v(1)=0.\label{linearized operator}
\end{equation}
where the auxiliary parameter $b$ approximates $a$ and is introduced to facilitate the bootstrap argument later used to control the error term $\epsilon$. See the choice of hyperparameter $b$ in \eqref{hyperparameter} and the modulation estimate in Proposition~\ref{Modulation equation proposition} and
Section~\ref{main theorem proof} for $k>1$.

Using the Lyapunov-Schmidt perturbation argument, we partially diagonalise $\mathcal{H}_b$ in $H_{b}^{1}$, as in \cite[Proposition 2.3]{Had-Raph}. However, unlike the unbounded case, the compactness of the domain introduces essential differences in spectral behavior. The eigenfunction $\psi_{b,j}$ is constructed as a perturbation of Bessel functions. See Section~\ref{preliminary} for more details. For sufficiently small $|b|<b^{*}(k)\ll1$, we obtain the eigenvalue expansion:
\begin{equation*}
    \lambda_{b,j}=\lambda_{j}-b+O(b^{2}), \quad \text{for } 1\leq j\leq k,
\end{equation*}
with corresponding eigenfunction satisfying
\begin{equation*}
\mathcal{H}_{b}\psi_{b,j}=\lambda_{b,j}\psi_{b,j},\quad \text{for }1\leq j\leq k.
\end{equation*}
Although the unperturbed problem admits a logarithmic singular eigenfunction associated with the zero eigenvalue, the spectral gap in $H^1_b$ due to the maximum principle precludes the appearance of such singular behavior. This is a key distinction from the exterior Stefan problem considered in \cite{Had-Raph}.

\textbf{2. Leading-order dynamics and modulation equations:} We consider approximate solutions of the form
\begin{itemize}
    \item For $k = 1$, we set
    \begin{equation*}
        v(s, y) = b(s) \psi_{b(s),1}(y).
    \end{equation*}
    \item For $k > 1$, we define
    \begin{equation*}
    v(s, y) = \sum_{j=1}^k b_j(s) \psi_{b(s),j}(y).
    \end{equation*}
\end{itemize}
In the case for $k>1$, we separate the parameter $b$ from the coefficients $b_{j}$ to construct a regime where the lower modes $b_{j}$, where
$1\leq j\leq k-1$ are trapped and $b_k$ dominates. This is due to the distribution of the eigenvalues of the Dirichlet Laplacian on $B_1(0)$. See Section~\ref{profile set up and bootstrap assumptions} and Remark~\ref{motivation for introduce b(s)} for details, along with the choice of parameter $b(s)$ in~\eqref{hyperparameter}.

Substituting the ansatz into equation \eqref{renormalised stefan} and projecting onto each eigenmode, we derive the modulation equations.\\
For $k = 1$, we obtain:
\[
b_s + b \lambda_{b,1} - (a - b) b \frac{\langle \Lambda \psi_{b,1}, \psi_{b,1} \rangle_b}{\langle \psi_{b,1}, \psi_{b,1} \rangle_b} = 0.
\]
For $k > 1$:
\[
(b_j)_s + b_j \lambda_{b,j} + a b_k \frac{\langle \Lambda \psi_{b,k}, \psi_{b,j} \rangle_b}{\langle \psi_{b,j}, \psi_{b,j} \rangle_b} = 0, \quad 1 \le j \le k.
\]
Additionally, the boundary condition $\partial_y v(s,1) = a$ implies:
\[
a = b\, \partial_y \psi_{b,1}(1) \quad \text{for } k = 1, \qquad a = \sum_{j=1}^k b_j\, \partial_y \psi_{b,j}(1) \quad \text{for } k > 1.
\]

Combining the modulation equation and the boundary condition, we derive the resulting dynamical system:
\begin{itemize}
    \item For $k = 1$,
\[ \begin{cases}
    b_s + \lambda_1 b + \sqrt{2\lambda_1} b^2 = 0, \\
    a = -\sqrt{2\lambda_1} b, \\
    \frac{ds}{dt} = \frac{1}{\lambda^2}, \quad -\frac{\lambda_s}{\lambda} = a.
    \end{cases}
    \]
    \item For $k > 1$, assuming $b_k \gg b_j$ for $j < k$,
    \[
    \begin{cases}
    (b_k)_s + \lambda_k b_k + (-1)^{k+1} \sqrt{2\lambda_k} b_k^2 = 0, \\
    a = (-1)^k \sqrt{2\lambda_k} b_k, \\
    \frac{ds}{dt} = \frac{1}{\lambda^2}, \quad -\frac{\lambda_s}{\lambda} = a.
    \end{cases}
    \]
\end{itemize}
Thus, the sign of $b_k(0)$ and the parity of $k$ determine whether the solution is in the freezing or melting regime.

\begin{remark}
Whether the solution lies in a melting or freezing regime is determined by the sign of the initial modulation parameter
$b_k(0)$ together with the parity of $k$. This phenomenon originates from the spectral behavior of the associated eigenfunctions near the boundary. The perturbed eigenfunctions $\psi_{b,k}$ are close to the unperturbed eigenfunctions $\eta_k$, which are rescaled Bessel functions of the first kind. As shown in Proposition~\ref{prop:1}, the sign of $\partial_y \psi_{b,k}(1)$ alternates with $k$, due to the oscillatory structure of Bessel functions near the boundary. This sign alternation is what determines whether a given $b_k(0)$ leads to freezing or melting, depending on the parity of $k$.
\end{remark}

\textbf{3. Energy estimates:} To control the error, we consider the exact solution $v$ of the form:
\[
v(s,y)=\sum_{j=1}^{k}b_{j}(s)\psi_{b(s),j}+\epsilon(s,y).
\]
For notational convenience, we set $b_1 = b$ as $k=1$. We impose orthogonality conditions. 
\begin{equation*}
\langle\epsilon,\psi_{b,j}\rangle_{b}=0,
\end{equation*}
to fix the evolution of $\epsilon$ and the modulation parameters $b_j$.
Then $\epsilon$ satisfies:
\[
\partial_s \epsilon + \mathcal{H}_b \epsilon + (a - b) \Lambda \epsilon = 0.
\]
To ensure that the error term $\epsilon$ does not interfere with the leading-order modulation dynamics, we set up a bootstrap argument controlling its size. In particular, we assume a smallness condition on $\norm{\epsilon(s)}_{H^1_b}$ in an appropriate norm over a time interval, and aim to improve this bound through energy estimates.

A key ingredient in closing this bootstrap argument is the coercivity of the operator $\mathcal{H}_b$ under appropriate orthogonality conditions. This coercivity ensures that the leading-order contribution of $\epsilon$ is dissipative, allowing us to gain decay in the weighted energy norm.

Differentiating the weighted $L^2$ norm yields: 
\begin{align*}
\frac{1}{2}\frac{d}{ds}\int_{0}^{1} & \epsilon^{2}(s,y)\rho_{b}(y)y\ dy=-\langle\mathcal{H}_{b}\epsilon,\epsilon\rangle_{b}\ \\
 & -\frac{1}{4}(b_{s}+2b(a-b))\int_{0}^{1}\epsilon^{2}\rho_{b}(y)y^{3}\ dy+(a-b)\norm{\epsilon}_{L_{b}^{2}}^2.
\end{align*}
Thanks to the orthogonality condition, we obtain coercivity of the first term from the spectral gap of $\mathcal{H}_b$ in $H^1_b$, which yields:
\[
\frac{1}{2} \frac{d}{ds} \|\epsilon(s)\|_{L^2_b}^2 \le -\lambda_{k+1} \|\epsilon(s)\|_{L^2_b}^2.
\]
However, to close the bootstrap argument, we must control pointwise boundary derivatives such as $\partial_y \epsilon(s,1)$. This cannot be achieved with $H^1$-regularity alone. Therefore, we establish the bootstrap in $H^2$ to control the boundary terms pointwise. 

Finally, a standard topological argument based on the Brouwer Fixed Point Theorem is used to select suitable initial data for the excited mode and complete the construction.

\subsection*{Acknowledgements}
We are partially supported by the National Research Foundation of Korea (NRF) under grants NRF-2019R1A5A1028324 and NRF-2022R1A2C1091499. We are grateful to Soonsik Kwon for suggesting this problem and for many helpful discussions.

\section{Spectral analysis of linearized operator \texorpdfstring{$\mathcal{H}_{b}$}{Lg}} 
To analyze the renormalized flow \eqref{renormalised stefan}, we begin by studying the spectral structure of the linearized operator
\[
\mathcal{H}_b = -\Delta + b\Lambda \quad \text{on } B_1(0), \quad \text{with } v(1) = 0.
\]
This operator captures the leading-order dynamics of the modulation parameters near the approximate profile and plays a central role in the modulation analysis. 

This approach was first introduced in the pioneering work of Herrero and Vel\'azquez~\cite{Herrero-Velaz}, where formal asymptotic profiles for the Stefan problem were derived. A rigorous justification for the two-dimensional exterior setting was later provided by Had\v{z}i\'c and Rapha\"el~\cite{Had-Raph}, who developed a spectral framework to construct the melting and freezing profiles. This strategy was further extended to three dimensions in \cite{3d-stefan}, building on the techniques introduced in $2$D case.

In this section, we carry out a detailed
spectral analysis of $\mathcal{H}_b$ in a
bounded domain $B_1(0)$. The boundedness of the domain induces a discrete spectrum for the radial Laplacian. More importantly, the spectral gap property, which follows from the maximum principle, rules out eigenvalues close to zero that would correspond to perturbations of the singular profile $\log r$, unlike in unbounded domains where such modes naturally arise. The spectral gap prevents the appearance of near-zero eigenvalues, allowing the modulation framework to generate global-in-time dynamics in the original time variable $t$.

We first recall several classical facts on the Bessel functions, 
which describe the eigenfunctions of the radial Laplacian in the unit ball. These serve as the unperturbed basis in our analysis. We then employ a Lyapunov–Schmidt reduction to treat the $b \Lambda$ term as a small perturbation. This allows us to obtain an explicit expansion of the exact eigenfunctions and eigenvalues of $\mathcal{H}_b$ for small $b$.
These results will be used in Section~\ref{profile set up and bootstrap assumptions} 
to derive the modulation estimates and to track the evolution of
the boundary dynamics.

\subsection{Preliminary}\label{preliminary}
We begin by recalling several classical results on the
eigenvalues and eigenfunctions of the radial Laplacian on the unit ball with Dirichlet boundary conditions \eqref{eq:1}. We also introduce the definition and key properties of the linearized operator $\mathcal{H}_{b}$, together with the notational conventions that will be used throughout Section~\ref{subsec:Partial-diagonalisation of }.

The eigenvalue problem \eqref{eq:1} is equivalent to 
\begin{equation}
y\partial_{y}^{2}u+\partial_{y}u+y\lambda u=0 \text{ in } y\in[0,1], \text{ where } u(1)=0,\quad u'(0)=0.\label{eq:2}
\end{equation}
for some $\lambda\in\mathbb{R}$. By rescaling the spatial variable
$y \coloneqq r/\sqrt{\lambda}$ for $\lambda>0$, equation \eqref{eq:2} reduces
to the Bessel equation of order zero: 
\begin{equation}
r^{2}\partial_{r}^{2}v+r\partial_{r}v+r^{2}v=0 \text{ in } r\in[0,\sqrt{\lambda}], \text{ where } v(\sqrt{\lambda})=0,\quad v'(0)=0.\label{Bessel-equation}
\end{equation}
Solutions to the Sturm--Liouville problem \eqref{Bessel-equation}
are given by the Bessel function of the first kind of order zero, denoted $J_{0}$.
This function is an even regular solution to the Bessel equation on $\bbR$, and possesses infinitely many simple zeros $\{ \pm r_j \}_{j \in \mathbb{N}}$, where $0<r_{j}<r_{j+1}$,
and $J_{0}'(0)=0$.
We construct solutions to \eqref{Bessel-equation}
by truncating $J_0$ at $r_{j}=\sqrt{\lambda_{j}}$. Rescaling back to the variable $y$, we obtain a family of solutions to \eqref{eq:2}, denoted $\{u_{j}:[0,1]\rightarrow\mathbb{R}\}_{j\in\mathbb{N}}$, where
\[
u_{j}(y)=J_{0}(y\sqrt{\lambda_{j}}),\qquad\lambda_{j}=r_{j}^{2}.
\]
From the well-known identity
\begin{equation*}
\int_{0}^{1}[J_{0}(y\sqrt{\lambda_{j}})]^{2}y\ dy=\frac{1}{2}[J_{0}'(\sqrt{\lambda_{j}})]^{2},
\end{equation*}
we define a sequence of normalized eigenfunctions $\{\eta_{j}\}_{j\in\mathbb{N}}\subset C^{\infty}(\overline{B_{1}(0)})$
in $L^{2}(B_{1}(0))$ by
\begin{equation}\label{renormalization of unperturbed eigenfunctions}
\eta_{j}(y)\coloneqq \frac{\sqrt{2}J_{0}(y\sqrt{\lambda_{j}})}{|J_{0}'(\sqrt{\lambda_{j}})|}\text{ on } y\in[0,1].
\end{equation}
The family $\{ \eta_j \}_{j \in \mathbb{N}}$ forms an
orthonormal basis in $L^2_{\mathrm{rad}}(B_1(0))$ with respect to the standard $L^2$ inner product:
\begin{equation*}
\langle\eta_{i},\eta_{j}\rangle_{0}=\delta_{ij}.
\end{equation*}
Moreover, the functions $\eta_j$ belong to $H^1$, and are orthogonal with respect to the $H^1$-inner product as well.

We now summarize several standard spectral properties of the eigenvalues $\lambda_j$ of the equation \eqref{eq:1}:\\
\textbf{1. Gap between consecutive eigenvalues:} 
\begin{equation}
\lambda_{k+1}-\lambda_{k}>1.\label{increment}
\end{equation}
\textbf{2. Spectral gap estimate:}
If $u\in H_{0}^{1}$ satisfies the orthogonality conditions 
\[
\langle u,\eta_{j}\rangle_{0}=0 \text{ for } 1\leq j\leq k,
\]
then we have the coercivity estimate:
\begin{equation}
\norm{\partial_{y}u}_{L_{0}^{2}}^{2}\geq\lambda_{k+1}\norm{u}_{L_{0}^{2}}^{2}.\label{spectralgapofeta}
\end{equation}

Based on the spectral properties of the limiting equation \eqref{eq:1}, we now investigate the spectrum of the linearized operator $\mathcal{H}_{b}=-\Delta+b\Lambda$
on $B_{1}(0)$, subject to the Dirichlet boundary condition $v(1)=0$, for small values of the parameter $b$.
Using the boundary condition and observing that the linearized operator $\mathcal{H}_b$ admits a divergence form, as noted in \cite{Had-Raph},
\begin{equation}\label{self adjointness of H_b}
-\Delta + b\Lambda = -\frac{1}{y\rho_b} \partial_y \left( y \rho_b \partial_y \right),
\end{equation}
where $\rho_b(y) = e^{-\frac{1}{2}by^2}$, one can easily verify that $\mathcal{H}_b$ is self-adjoint with respect to the weighted inner product $\langle \cdot, \cdot \rangle_b$ on $H_b^1$, as defined in \eqref{weightedL2product}. This divergence structure, which plays a key role in our energy estimates, was first highlighted in \cite{Had-Raph}.

\subsection{Partial diagonalisation of $\mathcal{H}_{b}$} \label{subsec:Partial-diagonalisation of }
In this subsection, we diagonalise the linearized operator $\mathcal{H}_{b}$,
for small values of the parameter $b$ by treating it as a perturbation of the \eqref{eq:1}. To obtain spectral information, we employ the Lyapunov--Schmidt reduction
argument following the approach in \cite{Had-Raph}. We fix an $K\in\mathbb{N}$ and aim to construct perturbed first $K$ eigenpairs $(\psi_{b,k},\lambda_{b,k})$ for $1\leq k\leq K$ that remain close to $(\eta_k,\lambda_k)$ for each $k$, provided $b \in (0,b^{*})$ and 
$b^{*}=b^{*}(K)$ sufficiently small. 

In view of the Lyapunov--Schmidt reduction method, we decompose the eigenfunction $\psi_{b,k}$ of $\mathcal{H}_b$.
\[
\begin{cases}
\psi_{b,k}=T_{b,k}+\tilde{\phi}_{b,k},\\
T_{b,k}(y)=\eta_{k}(y)+\sum_{j=1}^{k-1}\mu_{b,jk}\eta_{j}(y),
\end{cases}
\]
as in \eqref{decompose of psi_{b,k}}, where $\tilde{\phi}_{b,k}$ is a correction term. We also rewrite the eigenvalue as in \eqref{eq:11}
\[\lambda_{b,k} = \lambda_k + \mu_{b,k}.\]
Substituting this decomposition into the eigenvalue problem $\mathcal{H}_b\psi_{b,k} = \lambda_{b,k}\psi_{b,k}$, we reinterpret it as an equation for $\tilde{\phi}_{b,k}$:
\begin{align*}
    (\mathcal{H}_{b}-\lambda_{k})\tilde{\phi}_{b,k} & =-\sum_{j=1}^{k-1}\mu_{b,jk}(\lambda_{j}-\lambda_{k})\eta_{j}-b\Lambda\left(\eta_{k}+\sum_{j=1}^{k-1}\mu_{b,jk}\eta_{j}\right)\\
    &\quad  +\mu_{b,k}\left(\eta_{k}+\sum_{j=1}^{k-1}\mu_{b,jk}\eta_{j}\right)+\mu_{b,k}\tilde{\phi}_{b,k}
\end{align*}
Our goal is to derive asymptotic expansions for $\mu_{b,k}$ and $\mu_{b,jk}$ for $1\leq j < k$, such that $\tilde{\phi}_{b,k}$ remains small for sufficiently small $b$.

To construct such a solution, we apply the Banach Fixed Point Theorem, relying on the near-invertibility of the operator $\mathcal{H}_b-\lambda_k$ on the orthogonal complement of the
span $\{ \eta_1, \dots, \eta_k \}$. 

To formalize this construction, we first define the Gram matrix $M_{b,k}$, for each $0<k\leq K$ by
\begin{equation*}
M_{b,k}\coloneqq\left(\langle\eta_{i},\eta_{j}\rangle_{b}\right)_{1\leq i,j\leq k}.
\end{equation*}
For small $b$, this matrix is close to the identity and therefore invertible.

Next, we define the modified operator $\tilde{H}_{b,k} : H^1_b \rightarrow L^2_b$ by
\begin{equation*}
    \tilde{H}_{b,k}u \coloneqq (\mathcal{H}_{b}-\lambda_{k})u-b\left(M_{b,k}^{-1}\left(\langle u,\Lambda\eta_{j}\rangle_{b}\right)_{1\leq j\leq k},\Pi_{k}\right),
\end{equation*}
where $\Pi_k \coloneqq (\eta_1, \cdots, \eta_k)$.
This operator $\tilde{H}_{b,k}$ is designed to map functions that are orthogonal to $\text{span}\{ \eta_1, \dots, \eta_k \}$ into the same orthogonal subspace of \( L^2_b \).

Lemma~\ref{lemma:near invertibility} says that $\tilde{H}_{b,k}$ is invertible in this orthogonal complement. The proof is based on a perturbative argument based on the spectral gap estimate \eqref{spectralgapofeta}, which is derived using a standard variational method. A complete proof is provided in Appendix~\ref{appendix}.

\begin{lemma}[Near invertibility of $\mathcal{H}_{b}-\lambda_{k}$]
\label{lemma:near invertibility} Let $K$ be as above and $1\le k\le K$.
There exists a small constant $b^{*}=b^{*}(K)>0$ such that for each
$|b|<b^{*}(K) \ll 1$, the following statement is true : For $f\in L_{b}^{2}\cap L^{\infty}(B_{1}(0))$
with 
\begin{equation*}
\langle f,\eta_{j}\rangle_{b}=0,\quad 1\leq j\leq k
\end{equation*}
there is a unique $u\in H_{b}^{1}$ such that 
\begin{equation}
\begin{cases}
\tilde{H}_{b,k}u=f,\\
\langle u,\eta_{j}\rangle_{b}=0,\quad 1\leq j\leq k.
\end{cases}\label{eq:3}
\end{equation}
In this case, we have the estimate 
\begin{equation}\label{estimatesofinvertibility}
\norm{u}_{H_{b}^{1}}\lesssim\norm{f}_{L_{b}^{2}}.
\end{equation}
\end{lemma}

We are now ready to state the main proposition of this section, which offers asymptotic expansions of both the eigenvalues and eigenfunctions of the perturbed operator \( \mathcal{H}_b \) for small values of \( b \). This finding serves as the spectral basis for the modulation analysis explored in the subsequent sections. It clarifies how the dominant spectral characteristics of the original problem \eqref{eq:1} are altered by the addition of the perturbative term \( b \Lambda \), thus providing a rigorous foundation for the decomposition framework outlined previously.

\begin{proposition}[Eigenvalues for $\mathcal{H}_{b}$]
\label{prop:1} Let $K\ge1$ be a large fixed parameter. There exists
a small parameter $b^{*}(K)>0$ such that for each $|b|<b^{*}(K)$,
$\mathcal{H}_{b}$ admits a sequence of eigenvalues, for $1\leq k\leq K$,
\begin{equation}
\mathcal{H}_{b}\psi_{b,k}=\lambda_{b,k}\psi_{b,k},\quad \psi_{b,k}\in H_{b}^{1}\label{eq:10}
\end{equation}
with the following properties:

1. Expansion of Eigenvalue : 
\begin{equation}\label{asymptotic expansion of lambda_{b,k}}
\lambda_{b,k}=\lambda_{k}-b+O(b^{2}),
\end{equation}
\begin{equation}
\partial_{b}\lambda_{b,k}=-1+O(|b|).\label{differentiation of lambda_{b,k} w.r.t b}
\end{equation}
2. Expansion of Eigenfunction : 
\begin{align}\label{decompose of psi_{b,k}}
\begin{cases}
\psi_{b,k}=T_{b,k}+\tilde{\phi}_{b,k},\\
T_{b,k}(y)=\eta_{k}(y)+\sum_{j=1}^{k-1}\mu_{b,jk}\eta_{j}(y),
\end{cases}
\end{align}
with 
\begin{align}\label{value of mu_{b,jk}}
\mu_{b,jk}=\frac{b}{\lambda_{k}-\lambda_{j}}\langle\Lambda\eta_{k},\eta_{j}\rangle_{0}+O(b^{2}),
\end{align}
\begin{align}\label{value of partial_b mu_{b,jk}}
\partial_{b}\mu_{b,jk}=\frac{1}{\lambda_{k}-\lambda_{j}}\langle\Lambda\eta_{k},\eta_{j}\rangle_{0}+O(|b|),
\end{align}
and 
\begin{align}\label{estimates for tilde{phi}_{b,k}}
\norm{\mathcal{H}_{b}\tilde{\phi}_{b,k}}_{L_{b}^{2}}+\norm{\tilde{\phi}_{b,k}}_{H_{b}^{1}}+\norm{b\partial_{b}\tilde{\phi}_{b,k}}_{H_{b}^{1}}+|\partial_{y}\tilde{\phi}_{b,k}(1)|\lesssim|b|.
\end{align}
3. Spectral gap estimates: For $u\in H_{b}^{1}$, and $\langle u,\psi_{b,j}\rangle_{b}=0$
for $1\leq j\leq k$, 
\begin{equation}\label{spectralgap1}
\norm{\partial_{y}u}_{L_{b}^{2}}^{2}\geq(\lambda_{k+1}+O(|b|))\norm{u}_{L_{b}^{2}}^{2}.
\end{equation}
Moreover, 
\begin{equation}\label{minimalityofeigen}
\lambda_{b,1}=\inf_{u\in H_{b}^{1}}\frac{\norm{\partial_{y}u}_{L_{b}^{2}}^{2}}{\norm{u}_{L_{b}^{2}}^{2}},
\end{equation}
\begin{equation}\label{maximumprin}
\psi_{b,1}(y)>0\quad\text{for}\quad|y|<1.
\end{equation}
4. Further estimates :
 $L_{b}^{2}$ norm of $\psi_{b,k}$ and $\partial_{b}\psi_{b,k}$:
\begin{equation}\label{Norm of psi}
\norm{\psi_{b,k}}_{L_{b}^{2}}=1+O(|b|),
\end{equation}
and
\begin{equation}\label{Norm of partial _b psi}
\norm{b\partial_{b}\psi_{b,k}}_{H_{b}^{1}}=O(|b|).
\end{equation}
Inner product of $\psi_{b,i}$ with $b\partial_{b}\psi_{b,j}$: 
\begin{equation}\label{IVTesitmates}
\begin{aligned} & \langle b\partial_{b}\psi_{b,i},\psi_{b,j}\rangle_{b}=O(|b|) \text{ for } 1\leq j<i\leq k,\\
 & \langle b\partial_{b}\psi_{b,j},\psi_{b,k}\rangle_{b}=O(b^{2}) \text{ for } 1\leq j\leq k.
\end{aligned}
\end{equation}
\end{proposition}

\begin{proof} We fix $K \in \mathbb{N}$ and choose $b^*(K) > 0$ as in Lemma~\ref{lemma:near invertibility}, and assume that $|b| < b^*(K)$. We may further reduce $b^*(K)$, if necessary, and will indicate this explicitly each time, while continuing to use the same notation.

\noindent \textbf{Step1:}
We employ the Lyapunov--Schmidt reduction and seek a perturbed eigenpair of $\mathcal{H}_b$ in the form
\begin{align}
\lambda_{b,k} & = \lambda_k + \mu_{b,k}, \label{eq:11}\\
\psi_{b,k} & = T_{b,k} + \tilde{\phi}_{b,k}, \label{eq:12}
\end{align}
where $T_{b,k}$ is the finite-dimensional ansatz defined by
\begin{equation}
T_{b,k}(y) \coloneqq \eta_k(y) + \sum_{j=1}^{k-1} \mu_{b,jk} \eta_j(y), \label{c:1}
\end{equation}
and $\tilde{\phi}_{b,k}$ is a correction term orthogonal to $\text{span} \{ \eta_1, \dots, \eta_k \}$ with respect to the weighted inner product $\langle \cdot, \cdot \rangle_b$.

Both $\sum_{j=1}^{k-1}\mu_{b,jk} \eta_j$ and $\tilde{\phi}_{b,k}$ are of the same order in $b$ in the asymptotic expansion. This separation is made in order to isolate the part of the error that lies in the coercive subspace, where the invertibility of the linearized operator can be established: we isolate the components along $\{\eta_j\}_{1 \le j \le k-1}$ and absorb the remainder in $\tilde{\phi}_{b,k}$, allowing us to apply a fixed-point argument to establish existence.

Substituting \eqref{eq:11}--\eqref{eq:12} into the eigenvalue equation \eqref{eq:10}, we obtain
$(\mathcal{H}_b - \lambda_k)\tilde{\phi}_{b,k} = -(\mathcal{H}_b-\lambda_k)T_{b,k} + \mu_{b,k}(T_{b,k} + \tilde{\phi}_{b,k})$. We now expand the left-hand side using the linearity of the operator. Recalling that $\eta_j$ is an eigenfunction of the unperturbed equation \eqref{eq:1} we compute
\begin{equation}\label{(H_b-lambda_k)T_{b,k}}
\begin{aligned}
(\mathcal{H}_b - \lambda_k) T_{b,k}(y)
&= \left( -\Delta + b \Lambda - \lambda_k \right) \left( \eta_k(y) + \sum_{j=1}^{k-1} \mu_{b,jk} \eta_j(y) \right) \\
&= \sum_{j=1}^{k-1} \mu_{b,jk} (\lambda_j - \lambda_k) \eta_j(y)
+ b \Lambda \left( \eta_k(y) + \sum_{j=1}^{k-1} \mu_{b,jk} \eta_j(y) \right).
\end{aligned}
\end{equation}
Using the identity \eqref{(H_b-lambda_k)T_{b,k}} and the collecting terms, we obtain the following equation for the correction term $\tilde{\phi}_{b,k}$:  
\begin{equation}\label{equation for tilde{phi}_{b,k}}
\begin{aligned}
(\mathcal{H}_b - \lambda_k) \tilde{\phi}_{b,k}(y)
&= -\sum_{j=1}^{k-1} \mu_{b,jk} (\lambda_j - \lambda_k) \eta_j(y)
- b \Lambda \left( \eta_k(y) + \sum_{j=1}^{k-1} \mu_{b,jk} \eta_j(y) \right) \\
&\quad + \mu_{b,k} \left( \eta_k(y) + \sum_{j=1}^{k-1} \mu_{b,jk} \eta_j(y) \right)
+ \mu_{b,k} \tilde{\phi}_{b,k}(y).
\end{aligned}
\end{equation}
Now, define $B_{\alpha}$ as a small ball in $H^1_b$, centered at zero, for a sufficiently large universal constant
$\alpha>0$:
\[
\begin{aligned}B_{\alpha}\coloneqq \{\tilde{\phi}\in H_{b}^{1}: & |\langle\tilde{\phi},\Lambda\eta_{j}\rangle_{b}|+\norm{\tilde{\phi}}_{H_{b}^{1}}\leq\alpha|b|,\\
 & \text{ and } \langle\tilde{\phi},\eta_{j}\rangle_{b}=0 \text{ for } j=1,2,\cdots,k\}.
\end{aligned}
\]
We define a nonlinear map $f:B_{\alpha}\rightarrow H_{b}^{1}$ by
\begin{equation}\label{f(phi)}
\begin{aligned}f(\tilde{\phi}) & \coloneqq -\sum_{j=1}^{k-1}\mu_{b,jk}(\lambda_{j}-\lambda_{k})\eta_{j}-b\Lambda\left(\eta_{k}+\sum_{j=1}^{k-1}\mu_{b,jk}\eta_{j}\right)\\
 &\quad +\mu_{b,k}\left(\eta_{k}+\sum_{j=1}^{k-1}\mu_{b,jk}\eta_{j}\right)+\mu_{b,k}\tilde{\phi}.
\end{aligned}
\end{equation}
Then, the equation for $\tilde{\phi}_{b,k}$ \eqref{equation for tilde{phi}_{b,k}} is equivalent to 
\begin{equation}\label{equation for tilde{phi}}
    (\mathcal{H}_b-\lambda_k)\tilde{\phi}_{b,k} = f(\tilde{\phi}_{b,k}),
\end{equation} 
in a ball $B_{\alpha}$, with a well-chosen coefficient $\mu_{b,k}$ and $\mu_{b,jk}$ for $1\leq j <k$.

We now initiate the Lyapunov--Schmidt reduction to construct the eigenfunction $\psi_{b,k}$. The construction proceeds in two steps: we first determine the finite-dimensional coefficients $\mu_{b,k}$ and $\mu_{b,jk}$ depending on $\tilde{\phi} \in B_{\alpha}$, and then find the correction term $\tilde{\phi} = \tilde{\phi}_{b,k}$ satisfying \eqref{equation for tilde{phi}} by a fixed-point argument.\\

\noindent \textbf{1. Computations of coefficients $\mu_{b,k}$ and $\mu_{b,jk}$:} To simplify the notation, we define the coefficient vector $\vec{\mu}$ and its supremum norm as $\norm{\vec{\mu}}_{\infty}$ to be 
\begin{align*}
    & \vec{\mu}\coloneqq \left(\left(\mu_{b,jk}\right)_{1\leq j\leq k-1},\ \mu_{b,k}\right),\\
    & \norm{\vec{\mu}}_{\infty}\coloneqq \max\{|(\mu_{b,jk})_{1\leq j\leq k-1}|,|\mu_{b,k}|\}.
\end{align*}
For each $\tilde{\phi}\in B_{\alpha}$, we choose $\vec{\mu}(\tilde{\phi})$ so that $\norm{\vec{\mu}(\tilde{\phi})}_{\infty}=O(|b|)$ and satisfy
\begin{equation}\label{eqformu}
\langle f(\tilde{\phi}),\eta_{j}\rangle_{b}=b\left(\langle\tilde{\phi},\Lambda\eta_{j}\rangle_{b}\right) \text{ for } 1\leq j\leq k.
\end{equation}
The condition \eqref{eqformu} is imposed to ensure that $f(\tilde{\phi})$ is in the range of the operator $\mathcal{H}_b - \lambda_k$. More precisely, it guarantees that the projection of $f(\tilde{\phi})$ onto $\text{span}\{\eta_1, \dots, \eta_k\}^\perp$ lies in the domain of $\tilde{H}_{b,k}^{-1}$. This compatibility condition \eqref{eqformu} is essential for applying the near-invertibility lemma~\ref{lemma:near invertibility}, and it leads to the following system:
\begin{equation}\label{eq:14}
    \begin{aligned}
        \langle f(\tilde{\phi}),\eta_{j}\rangle_{b} & =-\sum_{i=1}^{k-1}\mu_{b,ik}(\lambda_{i}-\lambda_{k})\langle\eta_{i},\eta_{j}\rangle_{b}\\
        &\quad -b\langle\Lambda\eta_{k},\eta_{j}\rangle_{b}-b\sum_{i=1}^{k-1}\mu_{b,ik}\langle\Lambda\eta_{i},\eta_{j}\rangle_{b}\\
        &\quad +\mu_{b,k}\left(\langle\eta_{k},\eta_{j}\rangle_{b}+\sum_{i=1}^{k-1}\mu_{b,ik}\langle\eta_{i},\eta_{j}\rangle_{b}\right)\\
        & =b(\langle\tilde{\phi},\Lambda\eta_{j}\rangle_{b}).
    \end{aligned}
\end{equation}
for $1\leq j\leq k$. This yields a nonlinear system of $k$ equations for the vector $\vec{\mu} \in \mathbb{R}^k$, which depends smoothly on the parameter $b$ and the function $\tilde{\phi} \in B_\alpha$. We claim that for each $\tilde{\phi} \in B_{\alpha}$ there is a unique $\vec{\mu}$ with $\norm{\mu}_{\infty} = O(b)$ that solves \eqref{eq:14}, which can be obtained via the Banach Fixed Point Theorem. To see this, let $A_{b,k}$ be a $k\times k$ matrix whose $j$-th row vector, denoted by $(A_{b,k})_j$, is given by
\[(A_{b,k})_j \coloneqq \left(\left((\lambda_k - \lambda_i)\langle \eta_i,\eta_j\rangle_b - b\langle \Lambda\eta_i,\eta_j\rangle_b\right)_{1\leq i \leq k-1} , \langle \eta_k,\eta_j\rangle_b\right). \]
Note that $A_{b,k} = \text{diag}(\lambda_k-\lambda_1,\cdots, \lambda_k - \lambda_{k-1},1) \cdot I_k + O(|b|)$, and hence invertible for sufficiently small $b$. 
Then \eqref{eq:14} is equivalent to 
\begin{equation}\label{eqformu matrix form}
    \vec{\mu} = A_{b,k}^{-1}G_k(\vec{\mu},\tilde{\phi}),
\end{equation}
where $G_k$ is a $k$-column vector whose $j$-th entry is defined by
\[
 (G_k)_j \coloneqq b\langle \tilde{\phi},\Lambda\eta_j \rangle_b + b \langle \Lambda \eta_k,\eta_j \rangle_b + \mu_{b,k}\sum_{i=1}^{k-1}\mu_{b,ik}\langle \eta_i ,\eta_j \rangle_b.
\]
For each $\tilde{\phi} \in B_{\alpha}$, consider the map
\begin{equation}\label{functional for finite dimension}
    \vec{\mu} \mapsto A_{b,k}^{-1}G_{k}(\vec{\mu},\tilde{\phi}),
\end{equation}
for $\norm{\vec{\mu}}_{\infty} = O(|b|)$. Possibly after reducing $b^*(K)>0$ and keeping the same notation, one can verify that the map in \eqref{functional for finite dimension} is a contraction on the ball of radius $O(b)$ in $\mathbb{R}^k$ for each $|b|<b^*(K)$. Therefore, by the Banach Fixed Point Theorem, for each $\tilde{\phi} \in B_{\alpha}$, there exists a unique $\vec{\mu} \in \mathbb{R}^k$ that solves \eqref{eqformu matrix form} with $\|\vec{\mu}\|_{\infty} = O(b)$.

Now, we compute the asymptotic expansion of $\vec{\mu}$. For each $\tilde{\phi} \in B_{\alpha}$, the solution $\vec{\mu}$ obtained from \eqref{eqformu matrix form} satisfies the system of equations:
\begin{align*}
 &  b\langle\tilde{\phi},\Lambda\eta_{j}\rangle_{b} =(\lambda_{k}-\lambda_{j})\mu_{b,jk}+b(C\vec{\mu})_{j}\\
 & \quad -b\langle\Lambda\eta_{k},\eta_{j}\rangle_{0}+\mu_{b,k}\sum_{i=1}^{k-1}\mu_{b,ik}\langle\eta_{i},\eta_{j}\rangle_{b} + O(b^2),
\end{align*}
for $1\leq j\leq k-1$, and for $j=k$ 
\begin{align*}
b\langle\tilde{\phi},\Lambda\eta_{k}\rangle_{b} & =-b\langle\Lambda\eta_{k},\eta_{k}\rangle_{0}+\mu_{b,k}+b(C\vec{\mu})_{k}\\
 & \quad  +\mu_{b,k}\sum_{j=1}^{k-1}\mu_{b,jk}\langle\eta_{k},\eta_{j}\rangle_{b}+O(b^{2}).
\end{align*}
Here, $C$ denotes a $k \times k$ matrix with entries satisfying $C_{ij} = O(1)$.
Since $\langle \tilde{\phi},\Lambda \eta_j \rangle_b = O(b)$ for all $1 \leq j \leq k$, we deduce the asymptotic expansions
\begin{equation}\label{choice of mu}
    \begin{aligned}
        &\mu_{b,k}=b\langle\Lambda\eta_{k},
        \eta_{k}\rangle_{0}+O_{\tilde{\phi}}(b^{2}),\\
        & \mu_{b,jk}=\frac{b}{\lambda_{k}-\lambda_{j}}\langle\Lambda\eta_{k},\eta_{j}\rangle_{0}+O_{\tilde{\phi}}(b^{2}),
    \end{aligned}
\end{equation}
for $1 \leq j \leq k-1$.
Finally, since $\norm{\eta_{k}}_{L_{0}^{2}}=1$, we compute the leading-order coefficient using integration by parts:
\begin{equation}\label{maincoeff of mu_=00007Bb,k=00007D}
\langle\Lambda\eta_{k},\eta_{k}\rangle_{0}=\int_{0}^{1}y\partial_{y}\eta_{k}\eta_{k}y\ dy=\left[\frac{1}{2}\eta_{k}^{2}(y)y^{2}\right]_{0}^{1}-\int_{0}^{1}\eta_{k}^{2}(y)y\ dy=-1.
\end{equation}

\noindent \textbf{2. Existence of $\tilde{\phi}_{b,k}$:} Now we turn to the existence of $\tilde{\phi}_{b,k}$ that solves \eqref{equation for tilde{phi}_{b,k}}. Define a map $F : B_{\alpha} \rightarrow L_b^2$ by
\begin{equation}\label{eq:15}
\begin{aligned}
F(\tilde{\phi}) 
&\coloneqq f(\tilde{\phi}) - \sum_{j=1}^{k} \left\{ M_{b,k}^{-1} \left( \langle f(\tilde{\phi}), \eta_i \rangle_b \right)_{1 \leq i \leq k} \right\}_j \eta_j \\
&= f(\tilde{\phi}) - \sum_{j=1}^{k} b \left( M_{b,k}^{-1} \left( \langle \tilde{\phi}, \Lambda \eta_i \rangle_b \right)_{1 \leq i \leq k}, \Pi_k \right).
\end{aligned}
\end{equation}
We view $\left( \langle \tilde{\phi}, \Lambda \eta_i \rangle_b \right)_{1 \leq i \leq k}$ as a column vector of size $k$ with $i$-th entry $\langle \tilde{\phi}, \Lambda \eta_i \rangle_b$. Hence, the expression
\[
M_{b,k}^{-1} \left( \langle \tilde{\phi}, \Lambda \eta_i \rangle_b \right)_{1 \leq i \leq k}
\]
yields a column vector of length $k$. We denote $(\cdot, \cdot)$ as the standard Euclidean inner product in $\mathbb{R}^k$. Hence, the map $F$ in \eqref{eq:15} represents the $L_b^2$-orthogonal projection of $f(\tilde{\phi})$ onto the orthogonal complement of $\operatorname{span} \{ \eta_1, \eta_2, \dots, \eta_k \}$.
Then, from the compatibility condition \eqref{eqformu}, the equation reduces to find $\tilde{\phi} \in B_{\alpha}$ to solve
\begin{equation}\label{H_{b,k}tilde{phi} = F(tilde{phi})}
    \tilde{H}_{b,k} \tilde{\phi} = F(\tilde{\phi}).
\end{equation}
To find such $\tilde{\phi}$, we consider \eqref{H_{b,k}tilde{phi} = F(tilde{phi})} as a fixed point problem, by considering the functional 
\[
    \tilde{\phi} \mapsto \tilde{H}_{b,k}^{-1}\circ F(\tilde{\phi})
\]
from $B_{\alpha}$ to itself. The invertibility of $\tilde{H}_{b,k}$ follows from Lemma~\ref{lemma:near invertibility}. In fact, since $F(\tilde{\phi}) \in \text{span}\{\eta_1,\cdots,\eta_k\}^{\perp}$ and $\vec{\mu}$ in \eqref{choice of mu} for each $\tilde{\phi} \in B_{\alpha}$, we have 
\begin{align*}
 & \norm{F(\tilde{\phi})}_{L_{b}^{2}}\leq C|b|,\quad \ F(\tilde{\phi})\in L^{\infty}(B_{1}(0))\\
 & \langle F(\tilde{\phi}),\eta_{j}\rangle_{b}=0  \text{ for } 1\leq j\leq k, \text{ and for each } \tilde{\phi}\in B_{\alpha},
\end{align*}
for some universal constant $C>0$ that does not depend on $b$ and $\alpha$. Therefore, from Lemma~\ref{lemma:near invertibility}, $\tilde{H}_{b,k}$ is invertible, and \eqref{estimatesofinvertibility} yields 
\begin{equation}\label{e:3}
\norm{\tilde{H}_{b,k}^{-1}\circ F(\tilde{\phi})}_{H_{b}^{1}}\leq C\norm{F(\tilde{\phi})}_{L_{b}^{2}}.
\end{equation}
for some universal constant $C>0$ independent of $b$. Hence, for large universal constant $\alpha>0$, $(\tilde{H}_{b,k}^{-1}\circ F)(B_{\alpha})\subset B_{\alpha}$.

We now claim that $\tilde{H}_{b,k}^{-1} \circ F : B_{\alpha} \rightarrow B_{\alpha}$ is a contraction, provided that $|b| < b^*(K)$ for a possibly smaller choice of $b^*(K)$. We first show
\[
\norm{f(\phi_{1})-f(\phi_{2})}_{L_{b}^{2}}\lesssim|b|\norm{\phi_{1}-\phi_{2}}_{H_{b}^{1}},
\]
for $\phi_{1},\ \phi_{2}\in B_{\alpha}$. Indeed, we have 
\begin{equation}
\begin{aligned} & f(\phi_{1})-f(\phi_{2})\\
= & -\sum_{j=1}^{k-1}(\mu_{b,jk}(\phi_{1})-\mu_{b,jk}(\phi_{2}))(\lambda_{j}-\lambda_{k})\eta_{j}-b\Lambda\sum_{j=1}^{k-1}(\mu_{b,jk}(\phi_{1})-\mu_{b,jk}(\phi_{2}))\eta_{j}\\
 & +(\mu_{b,k}(\phi_{1})-\mu_{b,k}(\phi_{2}))\eta_{k}+(\mu_{b,k}(\phi_{1})-\mu_{b,k}(\phi_{2}))\phi_{1}+\mu_{b,k}(\phi_{2})(\phi_{1}-\phi_{2})\\
 & +\mu_{b,k}(\phi_{1})\sum_{j=1}^{k-1}\mu_{b,jk}(\phi_{1})\eta_{j}-\mu_{b,k}(\phi_{2})\sum_{j=1}^{k-1}\mu_{b,jk}(\phi_{2})\eta_{j}.
\end{aligned}
\label{e:2}
\end{equation}
Taking the inner product $\langle \cdot, \cdot \rangle_b$ of both sides of \eqref{e:2} with $\eta_j$ for $j = 1, \dots, k$, we estimate the quantities $|\mu_{b,jk}(\phi_1) - \mu_{b,jk}(\phi_2)|$ and $|\mu_{b,k}(\phi_1) - \mu_{b,k}(\phi_2)|$. For some constant $C > 0$ depending on $K$, we obtain
\[
\left| \langle f(\phi_1) - f(\phi_2), \eta_j \rangle_b \right| 
\geq \frac{1}{2} |\mu_{b,jk}(\phi_1) - \mu_{b,jk}(\phi_2)| \cdot |\lambda_j - \lambda_k| 
- C |b| \| \vec{\mu}(\phi_1) - \vec{\mu}(\phi_2) \|_{\infty},
\]
for $1 \leq j \leq k-1$, and for $j = k$,
\[
\left| \langle f(\phi_1) - f(\phi_2), \eta_k \rangle_b \right| 
\geq \frac{1}{2} |\mu_{b,k}(\phi_1) - \mu_{b,k}(\phi_2)| 
- C |b| \| \vec{\mu}(\phi_1) - \vec{\mu}(\phi_2) \|_{\infty}.
\]
Hence, from the compatibility conditions \eqref{eqformu}, we have 
\begin{equation}\label{estimate for difference of phi}
\norm{\vec{\mu}(\phi_{1})-\vec{\mu}(\phi_{2})}_{\infty}\lesssim|b|\norm{\phi_{1}-\phi_{2}}_{L_{b}^{2}}\lesssim|b|\norm{\phi_{1}-\phi_{2}}_{H_{b}^{1}}.
\end{equation}
From \eqref{estimate for difference of phi} and the fact that $F(\tilde{\phi})$ is the $\langle \cdot, \cdot \rangle_b$-orthogonal projection of $f(\tilde{\phi})$ onto the orthogonal complement of $\operatorname{span} \{ \eta_1, \eta_2, \dots, \eta_k \}$, we obtain
\begin{equation*}
\norm{F(\phi_{1}) - F(\phi_{2})}_{L_{b}^{2}} \leq \norm{f(\phi_{1}) - f(\phi_{2})}_{L_{b}^{2}} \lesssim |b| \norm{\phi_{1} - \phi_{2}}_{H_{b}^{1}}.
\end{equation*}
Hence, possibly after further reducing $b^*(K)>0$, we conclude that for all $|b| < b^*(K)$, the map $\tilde{H}_{b,k}^{-1} \circ F$ is a contraction:
\[
\norm{\tilde{H}_{b,k}^{-1} \circ F(\phi_{1}) - \tilde{H}_{b,k}^{-1} \circ F(\phi_{2})}_{H_{b}^{1}} \lesssim |b| \norm{\phi_{1} - \phi_{2}}_{H_{b}^{1}}.
\]
According to the Banach fixed point Theorem, there exists a unique solution $\tilde{\phi}_{b,k} \in B_{\alpha}$ to \eqref{H_{b,k}tilde{phi} = F(tilde{phi})}.

\noindent \textbf{Step 2}: We now prove a spectral gap property for the perturbed operator $\mathcal{H}_b$. Suppose $u\in H_{b}^{1}$ satisfied
\begin{equation*}
\langle u,\psi_{b,j}\rangle_{b}=0 \text{ for } 1\leq j\leq k.
\end{equation*}
Let $v\in H_{0}^{1}(B_{1}(0))$ be defined by 
\begin{align*}
 & v\coloneqq u-\sum_{j=1}^{k}\kappa_{j}\eta_{j},\\
 & \text{ where } \kappa_{j}=\left(M_{b,k}^{-1}(\langle u,\eta_{i}\rangle_{b})_{1\leq i\leq k}\right)_{j}
\end{align*}
Then, by construction, $\langle v,\eta_{j}\rangle_{b}=0$ for $1\leq j\leq k$. Moreover, from \eqref{eq:7} in Appendix~\ref{appendix} and \eqref{eq:12}, we have
\begin{align}
 & \norm{\partial_{y}v}_{L_{b}^{2}}^{2}\geq(\lambda_{k+1}+O(|b|))\norm{v}_{L_{b}^{2}}^{2}\label{spec1}\\
 & 0=\langle u,\psi_{b,j}\rangle_{b}=\langle u,\eta_{j}+\sum_{i=1}^{j-1}\mu_{b,ij}\eta_{i}+\tilde{\phi}_{b,j}\rangle_{b}.\label{orthocondition of u}
\end{align}
Since $\mu_{b,ij}=O(|b|)$ and $\tilde{\phi}_{b,j}\in B_{\alpha}$,
it follows from \eqref{orthocondition of u} and the Cauchy--Schwarz inequality that
\[
|\langle u,\eta_{j}\rangle_{b}|\lesssim|b|\norm{u}_{L_{b}^{2}},
\]
and hence
\[
\norm{u-v}_{H_{b}^{1}}\lesssim|b|\norm{u}_{L_{b}^{2}}.
\]
Also, from \eqref{spec1} we estimate 
\[
\begin{aligned}\norm{\partial_{y}u}_{L_{b}^{2}} & \geq\norm{\partial_{y}v}_{L_{b}^{2}}-\norm{\partial_{y}u-\partial_{y}v}_{L_{b}^{2}}\\
 & \geq\sqrt{\lambda_{k+1}+O(|b|)}\norm{v}_{L_{b}^{2}}-O(|b|)\norm{u}{L_{b}^{2}}\\
 & \geq\sqrt{\lambda_{k+1}+O(|b|)}\norm{u}_{L_{b}^{2}}.
\end{aligned}
\]
This proves \eqref{spectralgap1}. To prove \eqref{minimalityofeigen}
and \eqref{maximumprin}, let 
\[
\sigma(b)=\text{inf}_{u\in H_{b}^{1}}\frac{\norm{\partial_{y}u}_{L_{b}^{2}}^{2}}{\norm{u}_{L_{b}^{2}}^{2}}.
\]
Given a minimizing sequence of functions $\{u_n\}_{n \in \mathbb{N}}$ normalized in $L_b^2$ and bounded in $H_b^1$, the compact embedding of $H_b^1$ into $L_b^2$ implies strong convergence (up to a subsequence). Let $u_n \to \phi_b$ in $L_b^2$. Then, the Lagrange multiplier Theorem implies that
\begin{equation*}
\mathcal{H}_b \phi_b = \sigma(b) \phi_b.
\end{equation*}
Since $|\phi_b|$ also achieves the infimum, we may assume that $\phi_b \geq 0$. Suppose for contradiction that $\sigma(b) \ne \lambda_{b,1}$. It follows from \eqref{self adjointness of H_b} that $\phi_b$ is orthogonal to $\psi_{b,1}$. On the other hand, from the construction of $\psi_{b,1}$, we have
\[
\lambda_{b,1} = \frac{\| \partial_y \psi_{b,1} \|_{L_b^2}^2}{\| \psi_{b,1} \|_{L_b^2}^2}.
\]
Thus, we have $\lambda_{b,1} \geq \sigma(b)$. However, the spectral gap property \eqref{spectralgap1} implies that this leads to a contradiction. Therefore, we conclude that $\sigma(b) = \lambda_{b,1}$. Now, the simplicity of the first eigenvalue $\lambda_{b,1}$, together with a classical argument based on the strong maximum principle, implies that $\psi_{b,1} > 0$ in $|y| < 1$. This establishes \eqref{maximumprin}.

\noindent \textbf{Step 3:} We prove \eqref{value of mu_{b,jk}}, \eqref{value of partial_b mu_{b,jk}} and \eqref{estimates for tilde{phi}_{b,k}}.

\textbf{Computation of $\partial_b \lambda_{b,k}$.}  
We begin by differentiating both sides of \eqref{eq:10} with respect to $b$:
\begin{equation}
\partial_{b}(-\Delta+b\Lambda)\psi_{b,k}=\mathcal{H}_{b}\left(\partial_{b}\psi_{b,k}\right)+\Lambda\psi_{b,k}=(\partial_{b}\lambda_{b,k})\psi_{b,k}+\lambda_{b,k}\partial_{b}\psi_{b,k}.\label{e:5}
\end{equation}
Taking the $\langle\cdot,\cdot\rangle_{b}$-inner product
of \eqref{e:5} with $\psi_{b,k}$ yields
\[
\langle\mathcal{H}_{b}\partial_{b}\psi_{b,k},\psi_{b,k}\rangle_{b}+\langle\Lambda\psi_{b,k},\psi_{b,k}\rangle_{b}=\partial_{b}\lambda_{b,k}\langle\psi_{b,k},\psi_{b,k}\rangle_{b}+\lambda_{b,k}\langle\partial_{b}\psi_{b,k},\psi_{b,k}\rangle_{b}.
\]
From \eqref{self adjointness of H_b} and the estimate from \textbf{Step 1}, we obtain
\[
\| \psi_{b,k} - \eta_k \|_{L_b^2} = O(|b|)
\]
we deduce that
\begin{equation*}
\partial_{b}\lambda_{b,k}=\frac{\langle\Lambda\psi_{b,k},\psi_{b,k}\rangle}{\langle\psi_{b,k},\psi_{b,k}\rangle_{b}}=\frac{\langle\Lambda\eta_{k},\eta_{k}\rangle_{0}}{\langle\eta_{k},\eta_{k}\rangle_{0}}+O(|b|).
\end{equation*}
Now, combining this with \eqref{maincoeff of mu_=00007Bb,k=00007D} and the normalization $\norm{\eta_{k}}_{L_{0}^{2}}=1$
yields \eqref{differentiation of lambda_{b,k} w.r.t b}.

\textbf{Computation of $\partial_{b}\mu_{b,jk}$ and proof of \eqref{estimates for tilde{phi}_{b,k}}:}
Substituting \eqref{eq:12} into \eqref{e:5}, we have 
\begin{equation}\label{a:2}
\begin{aligned} & \mathcal{H}_{b}\left(\partial_{b}T_{b,k}+\partial_{b}\tilde{\phi}_{b,k}\right)+\Lambda T_{b,k}+\Lambda\tilde{\phi}_{b,k}\\
 & =\left(\partial_{b}\lambda_{b,k}\right)T_{b,k}+\left(\partial_{b}\lambda_{b,k}\right)\tilde{\phi}_{b,k}+\lambda_{b,k}\partial_{b}\left(T_{b,k}+\tilde{\phi}_{b,k}\right).
\end{aligned}
\end{equation}
Applying the $\langle \cdot, \cdot \rangle_b$-inner product to \eqref{a:2} against $\eta_j$ for $1\leq j\leq k-1$, we have
\begin{equation}\label{(H_b partial_b psi_{b,k},eta_j)1}
\begin{aligned} & \langle\mathcal{H}_{b}\partial_{b}\psi_{b,k},\eta_{j}\rangle_{b}=-\langle\Lambda T_{b,k},\eta_{j}\rangle_{b}-\langle\Lambda\tilde{\phi}_{b,k},\eta_{j}\rangle_{b}\\
 & +\langle\partial_{b}\left(\lambda_{b,k}T_{b,k}\right),\eta_{j}\rangle_{b}+\partial_{b}\lambda_{b,k}\langle\tilde{\phi}_{b,k},\eta_{j}\rangle_{b}+\lambda_{b,k}\langle\partial_{b}\tilde{\phi}_{b,k},\eta_{j}\rangle_{b}.
\end{aligned}
\end{equation}
We now estimate each term on the right-hand side of \eqref{(H_b partial_b psi_{b,k},eta_j)1}.

\textbf{1.} Estimate for $\lambda_{b,k}\langle\partial_{b}\tilde{\phi}_{b,k},\eta_{j}\rangle_{b}$
: Differentiating the orthogonality condition $\langle\tilde{\phi}_{b,k},\eta_{j}\rangle_{b} = 0$
with respect to $b$ yields 
\begin{align*}
0=\partial_{b}\int_{0}^{1}\tilde{\phi}_{b,k}\eta_{j}\rho_{b}y\ dy=\int_{0}^{1}\left\{ \partial_{b}\tilde{\phi}_{b,k}\eta_{j}-\frac{y^{2}}{2}\tilde{\phi}_{b,k}\eta_{j}\right\} \rho_{b}y\ dy.
\end{align*}
for $1\leq j\leq k$. Therefore, we obtain
\begin{equation}\label{(partial_b tilde{phi}_{b,k},eta_j)=O(b)}
\langle\partial_{b}\tilde{\phi}_{b,k},\eta_{j}\rangle_{b}=O(|b|).
\end{equation}
It then follows from \eqref{(partial_b tilde{phi}_{b,k},eta_j)=O(b)} and \eqref{differentiation of lambda_{b,k} w.r.t b} that
\begin{equation}\label{estimate 1}
\lambda_{b,k}\langle\partial_{b}\tilde{\phi}_{b,k},\eta_{j}\rangle_{b}=O(|b|).
\end{equation}

\textbf{2.} Estimate for $\partial_{b}\lambda_{b,k}\langle\tilde{\phi}_{b,k},\eta_{j}\rangle_{b}$:
It follows from \eqref{differentiation of lambda_{b,k} w.r.t b} and the Cauchy--Schwarz inequality that
\begin{equation}\label{estimate 2}
|\partial_b \lambda_{b,k} \langle \tilde{\phi}_{b,k}, \eta_j \rangle_b| \lesssim \norm{\tilde{\phi}_{b,k}}_{L_b^2} \norm{\eta_j}_{L_b^2} = O(|b|),
\end{equation}
where we used the fact that $\tilde{\phi}_{b,k} \in B_{\alpha}$.

\textbf{3.} Estimate of $\langle\partial_{b}\left(\lambda_{b,k}T_{b,k}\right),\eta_{j}\rangle_{b}$:
Combining \eqref{value of mu_{b,jk}} with the fact that $\tilde{\phi}_{b,k} \in \operatorname{span} 
\{ \eta_1, \dots, \eta_k \}^\perp$, we obtain
\begin{align*}
\langle T_{b,k},\eta_{j}\rangle_{b}=\langle\eta_{k}+\sum_{i=1}^{k-1}\mu_{b,ik}\eta_{i}+\tilde{\phi}_{b,k},\eta_{j}\rangle_{b}=O(|b|).
\end{align*}
Next, taking the $\langle \cdot, \cdot \rangle_b$-inner product of $\partial_b T_{b,k}$ with $\eta_j$ and using \eqref{value of mu_{b,jk}}, we obtain
\begin{align*}
\langle\partial_{b}T_{b,k},\eta_{j}\rangle_{b}=\langle\sum_{i=1}^{k-1}\partial_{b}\mu_{b,ik}\eta_{i}+\partial_{b}\tilde{\phi}_{b,k},\eta_{j}\rangle_{b}=\partial_{b}\mu_{b,jk}+\sum_{i=1}^{k-1}O(|b|)\partial_{b}\mu_{b,ik}+O(|b|).
\end{align*}
Therefore, by \eqref{asymptotic expansion of lambda_{b,k}} and \eqref{differentiation of lambda_{b,k} w.r.t b}, we conclude that
\begin{equation}\label{estimate 3}
\langle\partial_{b}\left(\lambda_{b,k}T_{b,k}\right),\eta_{j}\rangle_{b}=\lambda_{k}\partial_{b}\mu_{b,jk}+\sum_{i=1}^{k-1}O(|b|)\partial_{b}\mu_{b,ik}+O(|b|).
\end{equation}

\textbf{4.} Estimate of $\langle\Lambda\tilde{\phi}_{b,k},\eta_{j}\rangle_{b}$
and $\langle\Lambda T_{b,k},\eta_{j}\rangle_{b}$: Since $\norm{ \tilde{\psi}_{b,k}}_{H_b^1} = O(|b|)$, and using the decomposition \eqref{decompose of psi_{b,k}}, we obtain
\begin{equation}\label{estimate 4}
\langle \Lambda \tilde{\phi}_{b,k}, \eta_j \rangle_b = O(|b|).
\end{equation}
Moreover, from \eqref{value of mu_{b,jk}}, we have
\begin{equation}\label{estimate 5}
\langle \Lambda T_{b,k}, \eta_j \rangle_b = \langle \Lambda \eta_k, \eta_j \rangle_b + O(|b|).
\end{equation}

\textbf{5.} Estimate of $\langle\mathcal{H}_{b}\partial_{b}\psi_{b,k},\eta_{j}\rangle_{b}$:
From \eqref{decompose of psi_{b,k}}, \eqref{value of mu_{b,jk}} and \eqref{(partial_b tilde{phi}_{b,k},eta_j)=O(b)}, we obtain
\begin{equation}\label{(H_b partial_b psi_{b,k},eta_j)2}
\begin{aligned}\langle\mathcal{H}_{b}\partial_{b}\psi_{b,k},\eta_{j}\rangle_{b} & =\langle\partial_{b}T_{b,k}+\partial_{b}\tilde{\phi}_{b,k},\mathcal{H}_{b}\eta_{j}\rangle_{b}\\
 & =\langle\sum_{i=1}^{k-1}\partial_{b}\mu_{b,ik}\eta_{i}+\partial_{b}\tilde{\phi}_{b,k},\lambda_{j}\eta_{j}+b\Lambda\eta_{j}\rangle_{b}\\
 & =\lambda_{j}\partial_{b}\mu_{b,jk}+\sum_{i=1}^{k-1}O(|b|)\partial_{b}\mu_{b,ik}+\norm{\partial_{b}\tilde{\phi}_{b,k}}_{L_{b}^{2}}O(|b|)+O(|b|).
\end{aligned}
\end{equation}
Substituting \eqref{(H_b partial_b psi_{b,k},eta_j)2} into \eqref{(H_b partial_b psi_{b,k},eta_j)1} for $1 \leq j \leq k-1$, and using the estimates \eqref{estimate 1}, \eqref{estimate 2}, \eqref{estimate 3}, \eqref{estimate 4}, and \eqref{estimate 5}, we solve the resulting system of $k - 1$ equations to obtain
\begin{equation}\label{difmujkbefphi}
\partial_{b}\mu_{b,jk}=\frac{\langle\Lambda\eta_{k},\eta_{j}\rangle_{0}}{\lambda_{k}-\lambda_{j}}+O(|b|)\left(1+\norm{\partial_{b}\tilde{\phi}_{b,k}}_{L_{b}^{2}}\right).
\end{equation}
To derive \eqref{value of partial_b mu_{b,jk}}, it remains to show that $\norm{\partial_b\tilde{\psi}_{b,k}}_{L^2_b} =O(1)$. This will follow from the estimate \eqref{estimates for tilde{phi}_{b,k}}. 

To do this, we take the inner product $\langle \cdot, \cdot \rangle_b$ of \eqref{a:2} with $\partial_b \tilde{\phi}_{b,k}$, which yields
\begin{equation}\label{estimates for dif tilde phi}
\begin{aligned}\langle\mathcal{H}_{b}\partial_{b}\tilde{\phi}_{b,k},\partial_{b}\tilde{\phi}_{b,k}\rangle_{b}= & -\langle\mathcal{H}_{b}\partial_{b}T_{b,k},\partial_{b}\tilde{\phi}_{b,k}\rangle_{b}-\langle\Lambda T_{b,k},\partial_{b}\tilde{\phi}_{b,k}\rangle_{b}\\
 & -\langle\Lambda\tilde{\phi}_{b,k},\partial_{b}\tilde{\phi}_{b,k}\rangle_{b}+\langle\partial_{b}\left(\lambda_{b,k}T_{b,k}\right),\partial_{b}\tilde{\phi}_{b,k}\rangle_{b}\\
 & +\partial_{b}\lambda_{b,k}\langle\tilde{\phi}_{b,k},\partial_{b}\tilde{\phi}_{b,k}\rangle_{b}+\lambda_{b,k}\langle\partial_{b}\tilde{\phi}_{b,k},\partial_{b}\tilde{\phi}_{b,k}\rangle_{b}.
\end{aligned}
\end{equation}
We now estimate both sides of \eqref{estimates for dif tilde phi} in order to conclude that
\[
\| \partial_b \tilde{\phi}_{b,k} \|_{H_b^1} = O(1).
\]

\textbf{1.} Estimate of the right-hand side of \eqref{estimates for dif tilde phi}: Since $\tilde{\phi}_{b,k} \in B_{\alpha}$, we have $\tilde{\phi}_{b,k}(1) = 0$, and hence $\partial_b \tilde{\phi}_{b,k}(1) = 0$. Moreover, from the proof of the estimate \eqref{(partial_b tilde{phi}_{b,k},eta_j)=O(b)}, it follows that $\partial_b \tilde{\phi}_{b,k}$ is almost orthogonal to $\eta_j$ with respect to the inner product $\langle \cdot, \cdot \rangle_0$ for all $1 \leq j \leq k$. Therefore, an argument analogous to that used in the proof of \eqref{eq:7} yields
\begin{equation}\label{b:1}
\norm{\partial_{z}\partial_{b}\tilde{\phi}_{b,k}}_{L_{b}^{2}}^{2}\geq\left(\lambda_{k+1}+O(|b|)\right)\norm{\partial_{b}\tilde{\phi}_{b,k}}_{L_{b}^{2}}^{2}.
\end{equation}

\textbf{2.} Estimate of $\langle\mathcal{H}_{b}\partial_{b}T_{b,k},\partial_{b}\tilde{\phi}_{b,k}\rangle_{b}$:
Combining \eqref{decompose of psi_{b,k}}, \eqref{difmujkbefphi}, and \eqref{(partial_b tilde{phi}_{b,k},eta_j)=O(b)}, we obtain
\begin{equation}\label{b:2}
\begin{aligned}
\langle\mathcal{H}_{b}\partial_{b}T_{b,k},\partial_{b}\tilde{\phi}_{b,k}\rangle_{b} & =\langle\mathcal{H}_{b}\left(\sum_{j=1}^{k-1}\partial_{b}\mu_{b,jk}\eta_{j}\right),\partial_{b}\tilde{\phi}_{b,k}\rangle_{b}\\
 & =\langle\sum_{j=1}^{k-1}\lambda_{j}\partial_{b}\mu_{b,jk}\eta_{j}+b\sum_{j=1}^{k-1}\partial_{b}\mu_{b,jk}\Lambda\eta_{j},\partial_{b}\tilde{\phi}_{b,k}\rangle_{b}\\
 & \lesssim \max_{1\leq j\leq k-1}|\partial_{b}\mu_{b,jk}||b|(1+\norm{\partial_{b}\tilde{\phi}_{b,k}}_{L_{b}^{2}})\\
 & \lesssim|b|\left(1+\norm{\partial_{b}\tilde{\phi}_{b,k}}_{L_{b}^{2}}\right)+b^{2}\left(1+\norm{\partial_{b}\tilde{\phi}_{b,k}}_{L_{b}^{2}}\right)^{2}.
\end{aligned}
\end{equation}

\textbf{3.} Estimate of $\langle\Lambda T_{b,k},\partial_{b}\tilde{\phi}_{b,k}\rangle_{b}$:
First, observe that since $|y| \leq 1$, by possibly reducing $b^*(K) > 0$ and retaining the same notation, we have
\begin{equation}\label{value of norm{Lambda eta_k}_{L^2}}
    \norm{\Lambda \eta_k}_{L^2_b} \leq 2\norm{\eta_k}_{L^2_0} = 2\sqrt{\lambda_k}\norm{\eta_k}_{L^2_0} = 2\sqrt{\lambda_k}.
\end{equation}
Combining \eqref{decompose of psi_{b,k}}, \eqref{asymptotic expansion of lambda_{b,k}} and \eqref{value of norm{Lambda eta_k}_{L^2}}, and applying the Cauchy--Schwarz inequality, we obtain
\begin{equation}\label{b:3}
\langle\Lambda T_{b,k},\partial_{b}\tilde{\phi}_{b,k}\rangle_{b}\leq\left(\norm{\Lambda\eta_{k}}_{L_{b}^{2}}+O(|b|)\right)\norm{\partial_{b}\tilde{\phi}_{b,k}}_{L_{b}^{2}}\leq\left(2\sqrt{\lambda_{k}}+O(|b|)\right)\norm{\partial_{b}\tilde{\phi}_{b,k}}_{L_{b}^{2}}.
\end{equation}

\textbf{4.} Estimate of $\langle\Lambda\tilde{\phi}_{b,k},\partial_{b}\tilde{\phi}_{b,k}\rangle_{b}$
: Since $\tilde{\phi}_{b,k} \in B_{\alpha}$, we have $\norm{\tilde{\phi}_{b,k}}_{H_b^1} = O(|b|)$.  
Applying the Cauchy–Schwarz inequality, we obtain
\begin{equation}\label{b:4}
\begin{aligned}
|\langle\Lambda\tilde{\phi}_{b,k},\partial_{b}\tilde{\phi}_{b,k}\rangle_{b}|\leq\norm{\Lambda\tilde{\phi}_{b,k}}_{L_{b}^{2}}\norm{\partial_{b}\tilde{\phi}_{b,k}}_{L_{b}^{2}}\leq\norm{\tilde{\phi}_{b,k}}_{\dot{H}^{1}}\norm{\partial_{b}\tilde{\phi}_{b,k}}_{L_{b}^{2}}=O(|b|)\norm{\partial_{b}\tilde{\phi}_{b,k}}_{L_{b}^{2}},
\end{aligned}
\end{equation}
where the last inequality follows from the fact that $|y|\leq1$.

\textbf{5.} Estimate of $\langle\partial_{b}\left(\lambda_{b,k}T_{b,k}\right),\partial_{b}\tilde{\phi}_{b,k}\rangle_{b}$:
From \eqref{differentiation of lambda_{b,k} w.r.t b}, \eqref{decompose of psi_{b,k}}, \eqref{(partial_b tilde{phi}_{b,k},eta_j)=O(b)}, and \eqref{difmujkbefphi}, we have 
\begin{equation}\label{b:5}
\begin{aligned}
 & |\langle\partial_{b}\left(\lambda_{b,k}T_{b,k}\right),\partial_{b}\tilde{\phi}_{b,k}\rangle_{b}|\\
 & \leq\left|(-1+O(|b|))\langle\eta_{k}+\sum_{j=1}^{k-1}\mu_{b,jk}\eta_{j},\partial_{b}\tilde{\phi}_{b,k}\rangle_{b}\right|+\left|(\lambda_{k}+O(|b|))\langle\sum_{j=1}^{k-1}\partial_{b}\mu_{b,jk}\eta_{j},\partial_{b}\tilde{\phi}_{b,k}\rangle_{b}\right|\\
 & \lesssim|b|\left(1+\max_{1\leq j\leq k-1}|\partial_{b}\mu_{b,jk}|\right)\\
 & \lesssim|b|+b^{2}\left(1+\norm{\tilde{\phi}_{b,k}}_{L_{b}^{2}}^{2}\right).
\end{aligned}
\end{equation}

\textbf{6.} Estimate of $\partial_{b}\lambda_{b,k}\langle\tilde{\phi}_{b,k},\partial_{b}\tilde{\phi}_{b,k}\rangle_{b}$
and $\lambda_{b,k}\langle\partial_{b}\tilde{\phi}_{b,k},\partial_{b}\tilde{\phi}_{b,k}\rangle_{b}$:
Using \eqref{asymptotic expansion of lambda_{b,k}} and \eqref{differentiation of lambda_{b,k} w.r.t b}, together with the estimate $\| \tilde{\phi}_{b,k} \|_{L_b^2} = O(|b|)$, we deduce that
\begin{equation}\label{b:6}
|\partial_{b}\lambda_{b,k}\langle\tilde{\phi}_{b,k},\partial_{b}\tilde{\phi}_{b,k}\rangle_{b}|\lesssim|b|\norm{\partial_{b}\tilde{\phi}_{b,k}}_{L_{b}^{2}},
\end{equation}
and 
\begin{equation}\label{b:7}
|\lambda_{b,k}\langle\partial_{b}\tilde{\phi}_{b,k},\partial_{b}\tilde{\phi}_{b,k}\rangle_{b}|\leq\left(\lambda_{k}+O(|b|)\right)\norm{\tilde{\phi}_{b,k}}_{L_{b}^{2}}.
\end{equation}
Finally, by collecting the estimates \eqref{b:1}–\eqref{b:7}, we obtain
\[
\begin{aligned}(\lambda_{k+1}+O(|b|)) & \norm{\partial_{b}\tilde{\phi}_{b,k}}_{L_{b}^{2}}^{2}\leq\norm{\partial_{y}\partial_{b}\tilde{\phi}_{b,k}}_{L_{b}^{2}}\\
 & \leq O(|b|)\left(1+\norm{\partial_{b}\tilde{\phi}_{b,k}}_{L_{b}^{2}}\right)+O(b^{2})\left(1+\norm{\partial_{b}\tilde{\phi}_{b,k}}_{L_{b}^{2}}\right)^{2}\\
 & \quad +(2\sqrt{\lambda_{k}}+O(|b|))\norm{\partial_{b}\tilde{\phi}_{b,k}}_{L_{b}^{2}}+O(|b|)\norm{\partial_{b}\tilde{\phi}_{b,k}}_{L_{b}^{2}}+O(|b|)\\
 & \quad +O(b^{2})\left(1+\norm{\partial_{b}\tilde{\phi}_{b,k}}_{L_{b}^{2}}^{2}\right)+(\lambda_{k}+O(|b|))\norm{\partial_{b}\tilde{\phi}_{b,k}}_{L_{b}^{2}}^{2}\\
 & =(\lambda_{k}+O(|b|))\norm{\partial_{b}\tilde{\phi}_{b,k}}_{L_{b}^{2}}^{2}+\left(2\sqrt{\lambda_{k}}+O(|b|)\right)\norm{\partial_{b}\tilde{\phi}_{b,k}}_{L_{b}^{2}}+O(|b|).
\end{aligned}
\]
Then, possibly reducing $b^*(K) > 0$ and retaining the same notation, applying \eqref{increment}, we obtain the following result.
\[
\norm{\partial_{b}\tilde{\phi}_{b,k}}_{L_{b}^{2}} = O(1) \quad \text{and hence} \quad \norm{\partial_{y}\partial_{b}\tilde{\phi}_{b,k}}_{L_{b}^{2}} = O(1).
\]
Therefore, we conclude that
\begin{equation} \label{ineq221:1}
\norm{b\, \partial_{b}\tilde{\phi}_{b,k}}_{H_{b}^{1}} = O(|b|).
\end{equation}

On the other hand, in view of \eqref{f(phi)} and the fact that we chose $\vec{\mu}$ such that $\norm{\vec{\mu}}_{\infty} = O(b)$, it follows that $\norm{f(\phi)}_{H_b^2} = O(b)$. Moreover, by standard elliptic regularity theory, we have $\tilde{\phi}_{b,k} \in H^2(B_1(0))$. Therefore,
\begin{equation}\label{ineq221:2}
\begin{aligned}
\norm{\mathcal{H}_{b}\tilde{\phi}_{b,k}}_{L_{b}^{2}} &= O(|b|), \\
\norm{\Delta\tilde{\phi}_{b,k}}_{L_{b}^{2}} &= O(|b|).
\end{aligned}
\end{equation}

\textbf{7.} Estimate of $|\partial_{y}\tilde{\phi}_{b,k}(1)|$: Integration
by parts yields 
\begin{align*}
\langle\mathcal{H}_{b}\tilde{\phi}_{b,k},y\rangle_{b} & =\int_{0}^{1}-\frac{1}{y\rho_{b}}\partial_{y}\left(y\rho_{b}\partial_{y}\tilde{\phi}_{b,k}\right)y\rho_{b}y\ dy\\
 & =-\left[y^{2}\rho_{b}(y)\partial_{y}\tilde{\phi}_{b,k}(y)\right]_{0}^{1}+\int_{0}^{1}\partial_{y}\tilde{\phi}_{b,k}\rho_{b}y\ dy.
\end{align*}
From \eqref{ineq221:1}, \eqref{ineq221:2} and the Cauchy--Schwarz inequality give \eqref{ineq221:1}, 
\begin{equation}
|\partial_{y}\tilde{\phi}_{b,k}(1)|=O(|b|).\label{bdyestimates}
\end{equation}
Therefore, combining \eqref{ineq221:1}, \eqref{ineq221:2}, \eqref{bdyestimates},
and the fact that $\tilde{\phi}_{b,k}\in B_{\alpha}$, we conclude the estimate \eqref{estimates for tilde{phi}_{b,k}}. In addition, this yields \eqref{value of partial_b mu_{b,jk}} by combining with \eqref{diffofpsiwrtb}.

\textbf{Step 5} : Proof of \eqref{Norm of psi} - \eqref{Norm of partial _b psi} and \eqref{IVTesitmates}.

\textbf{1.} Estimates for $\norm{\psi_{b,k}}_{L_{b}^{2}}$: Since $\tilde{\phi}_{b,k}\in B_{\alpha}$
and $\norm{\eta_{k}}_{L_{b}^{2}}=1+O(|b|)$, we have 
\begin{align*}
\norm{\psi_{b,k}}_{L_{b}^{2}}^{2} & =\langle\eta_{k}+\sum_{j=1}^{k-1}\mu_{b,jk}\eta_{j}+\tilde{\phi}_{b,k},\eta_{k}+\sum_{j=1}^{k-1}\mu_{b,jk}\eta_{j}+\tilde{\phi}_{b,k}\rangle_{b}=\norm{\eta_{k}}_{L_{b}^{2}}^{2}+O(|b|)\\
 & =\norm{\eta_{k}}_{L_{0}^{2}}^{2}=1+O(|b|).
\end{align*}

\textbf{2.} Estimate of $\norm{b\partial_{b}\psi_{b,k}}_{L_{b}^{2}}$:
From \eqref{decompose of psi_{b,k}}, \eqref{value of partial_b mu_{b,jk}}, and \eqref{estimates for tilde{phi}_{b,k}}, the derivative of $\psi_{b,k}$ with respect to $b$ is given by
\begin{equation}\label{diffofpsiwrtb}
b\partial_{b}\psi_{b,k}=\sum_{j=1}^{k-1}b\partial_{b}\mu_{b,jk}\eta_{j}+b\partial_{b}\tilde{\phi}_{b,k},
\end{equation}
Consequently, we have
\[
\norm{b\partial_{b}\psi_{b,k}}_{H_{b}^{1}}=O(|b|).
\]

\textbf{3.} Estimate of $\langle b\partial_{b}\psi_{b,i},\psi_{b,j}\rangle_{b}$
where $1\leq j,i\leq k$: From \eqref{IVTesitmates}, \eqref{Norm of partial _b psi}, and the Cauchy--Schwarz inequality, we obtain  
\begin{align*}
\langle b\,\partial_{b}\psi_{b,i},\,\psi_{b,j} \rangle_{b} 
&\leq \| b\,\partial_{b}\psi_{b,i} \|_{L_{b}^{2}} \cdot \| \psi_{b,j} \|_{L_{b}^{2}} \\
&= O(|b|),
\end{align*}
for $j < i$. On the other hand, for $i \geq j$, using \eqref{diffofpsiwrtb}, \eqref{value of partial_b mu_{b,jk}}, and the estimate $\langle \eta_j, \eta_k \rangle_b = O(|b|)$ for $1 \leq j < k$, we deduce
\begin{align*}
\langle b\,\partial_{b} \psi_{b,j},\, \psi_{b,k} \rangle_{b}
&= \langle \sum_{s=1}^{j-1} b\,\partial_{b} \mu_{b,sj} \eta_s + b\,\partial_{b} \tilde{\phi}_{b,j},\, \eta_k + \sum_{s=1}^{k-1} \mu_{b,sk} \eta_s + \tilde{\phi}_{b,k} \rangle_{b} \\
&= O(b^2).
\end{align*}
\end{proof}

\begin{remark}
The spectral gap property \eqref{spectralgap1}, along with the minimality of the first eigenvalue \eqref{minimalityofeigen} (which follows from the maximum principle), excludes the existence of eigenfunctions exhibiting logarithmic singularities (e.g., $\log r$) in the limit $b \to 0$. This prevents the corresponding eigenvalues $\lambda_{b,k}$ from vanishing, ensuring that $\lambda_{b,k} \to \lambda_k > 0$, for each $k \in \bbN$. Such spectral non-degeneracy plays a crucial role in the derivation of a non-degenerate modulation law for $b_s$, thus ruling out the regime in which the domain collapses to a point in finite time. We also note that, for $k=1$, the corresponding eigenfunction $\psi_{b,1}$ is strictly positive in $B_1(0)$. This implies that the initial data we construct for \eqref{stefanradial}, which corresponds to the stable regime, are sign-definite.
\end{remark}

\section{Infinite time freezing and melting regime}
In this section, we start the modulation analysis for the radial Stefan problem \eqref{stefanradial}. We consider a class of initial data for which the analysis can be carried out and introduce a nonlinear decomposition of the solution together with bootstrap assumptions on key modulation parameters and energy norms that will be propagated in the subsequent argument. As in \eqref{renormalised solution} and \eqref{renormalised stefan}, we begin by renormalizing the solution and the equation. Let $a$ denote a small parameter that captures the precise evolution of the scaling factor $\lambda(s)$. We define the renormalized variables by
\begin{equation*}
u(t,r) \coloneqq v(s,y)\big|_{t = t(s)}, \quad y \coloneqq \frac{r}{\lambda(t)}, \quad s(t) \coloneqq \int_0^t \frac{1}{\lambda(\tau)^2} \, d\tau,
\end{equation*}
so that the original problem \eqref{stefanradial} transforms into the following non-autonomous equation for $v(s,y)$:
\begin{equation*}
\begin{cases}
\partial_s v + \mathcal{H}_a v = 0, \quad a = -\frac{\lambda_s}{\lambda}, \\
v(s,1) = 0, \quad \partial_y v(s,1) = a.
\end{cases}
\end{equation*}
\subsection{Set up profiles and bootstrap assumptions}\label{profile set up and bootstrap assumptions}
We divide the cases into the stable case $k=1$ and the unstable case $k>1$, depending on the nature of the dynamics. For each $k \in \mathbb{N}$, we introduce the length $k$ vector $\beta_k$ to be
\begin{equation*}
    \beta_k \coloneqq (b_1,\cdots,b_k),
\end{equation*} where $b_k$ is the parameter that we track.

We now consider the decomposition of the solution $v$ to \eqref{renormalised stefan} as
\[
    v(s,y) = Q_{\beta_k(s)}(y) + \epsilon(s,y),
\]
where the leading profile $Q_{\beta_k(s)}$ is defined by
\begin{equation}\label{profile Q_{beta_k}}
    Q_{\beta_k(s)}(y) \coloneqq
    \begin{cases}
        b_1(s)\, \psi_{b_1(s),1}(y), & \text{if } k = 1, \\
        \sum_{j=1}^{k} b_j(s)\, \psi_{b(s),j}(y), & \text{if } k > 1.
    \end{cases}
\end{equation}
In the case $k = 1$, we simply set the modulation parameter $ b_1$ equal to $b$, which also appears as the parameter in the linearized operator $\mathcal{H}_b$. However, when $k > 1$, we introduce a fixed adiabatic parameter $b(s)$, chosen slightly smaller than the actual modulation parameter $b_k(s)$ as specified in \eqref{hyperparameter}, for a technical reason. This distinction is essential to define a consistent family of eigenfunctions $\{ \psi_{b,j} \}_{j=1}^k$ with smooth dependence on $b$, where each $\psi_{b,j}$ is constructed in Proposition~\ref{prop:1} as an eigenfunction of the operator $\mathcal{H}_b$. To avoid notational confusion, we emphasize that for $k > 1$, the parameter $b$ is fixed throughout the construction of the basis and appears in the spectral analysis, while $b_k$ governs the actual modulation dynamics.

However, to uniquely determine the modulation parameters $\{b_j\}_{1\leq j\leq k}$ associated with a given solution, we now introduce a set of orthogonality conditions for the perturbation $\epsilon$.
\[
\langle \epsilon(s), \psi_{b,j} \rangle_b = 0 \quad \text{for all } 1 \leq j \leq k.
\]
After fixing the modulation parameters, these orthogonal decompositions play an additional role: they ensure that the error term $\epsilon$ lies in the subspace of $\mathcal{H}_b$  where a spectral gap holds. The resulting coercivity of the linearized operator in this subspace will be a key ingredient in establishing energy estimates and closing the bootstrap argument on the smallness of $\epsilon$. To construct such a decomposition, we apply the Implicit Function Theorem so that the orthogonality conditions are satisfied.

\textbf{Case $k=1$.} 
We claim that for the constant smaller than $b^*(K)$ which we choose in Proposition~\ref{prop:1}, we still denote $b^*(K)$, for $0<|\wt b|<b^*(K)$
and for small $\norm{\wt \epsilon}_{H_{\wt b^{1}}}$, there exists
a unique decomposition 
\[
\wt b\psi_{\wt b,1}+\wt\epsilon=b_1\psi_{b_1,1}+\epsilon \text{ with } \langle\epsilon,\psi_{b_1,1}\rangle_{b_1}=0,
\]
for small $|b_1|$, and $\norm{\epsilon}_{H_{b}^{1}}$. In fact, let
$F$ be a map such that 
\[
F(b_1,\wt\epsilon)=\langle\wt b_1\psi_{\wt b_1,1}-b_1\psi_{b_1,1}+\wt\epsilon,\psi_{b_1,1}\rangle_{b_1}.
\]
Then, from the estimates for the Frechet derivative of $\psi_{b,1}$ in \eqref{IVTesitmates} replacing $b$ with $b_1$, the
derivative of $F$ with respect to $b_1$ at $(\wt b,0)$ is 
\[
\partial_{b_1}F(\wt b,0)=-\langle\psi_{\wt b,1}+\wt b\partial_{b_1}\psi_{b_1,1}|_{b_1=\wt b},\psi_{\wt b,1}\rangle_{\wt b}=-1+O(|\wt b|)<0,
\]
and $F(\hat{b},0)=0$. Therefore, the Implicit Function
Theorem yields that by choosing initial data of equation \eqref{stefanradial}
$v_{0}$ and $\epsilon_{0}$ to be 
\[
v_{0}=b_{0}\psi_{b_{0},1}+\epsilon_{0},
\]
where, $|b_{0}|$ and $\norm{\epsilon_0}_{H_{b_{0}}^{1}}$ sufficiently
small, then we can decompose the solution for \eqref{stefanradial} whose initial data is $v_{0}$
to be 
\begin{equation}
v(s,y)=b_1(s)\psi_{b_1(s),1}+\epsilon(s,y), \text{ where } \langle\epsilon,\psi_{b_1,1}\rangle_{b_1}=0,\label{profile set up k=1}
\end{equation}
and 
\[
b_0\psi_{b_0,1}+\epsilon_0=b_1(0)\psi_{b_1(0),1}+\epsilon(0).
\]

\noindent \textbf{Case $K \geq k>1$.} 
For some constant smaller than $b^*(K)$ which is chosen in Proposition~\ref{prop:1} possibly smaller, we claim that for $|b|,|\wt b_j|<b^*(K)$, and small $\norm{\wt\epsilon}_{H^1_b}$ we have a decomposition
\begin{equation*}
    \sum_{j=1}^k \wt b_j \psi_{b,j} + \wt\epsilon = \sum_{j=1}^k b_j \psi_{b,j} + \epsilon,
\end{equation*} where $\langle \psi_{b,j},\epsilon \rangle_b =0$ for some $b_j$ and $\norm{\epsilon}_{H^1_b}$. 
Indeed, as similar to case $k=1$, we define
a function $\overrightarrow{F}$ to be 
\[
\overrightarrow{F}(b_{1},b_{2},\cdots,b_{k},\wt \epsilon) \coloneqq \left(\langle \sum_{j=1}^k \wt b_j \psi_{b,j}-\sum_{j=1}^{k}b_{j}\psi_{b,j}+\wt \epsilon,\psi_{b,j}\rangle_{b}\right)_{1\leq j\leq k}.
\]
Then, $\overrightarrow{F}(\wt b_{1},\wt b_{2},\cdots,\wt b_{k},0)=0$ and 
derivative of $\overrightarrow{F}$ with respect to $(b_{j})_{1\leq j\leq k}$ at $(\wt b_{1},\wt b_{2},\cdots,\wt b_{k},0)$ is 
\[
(\partial_{b_{j}}\overrightarrow{F})_{1\leq j\leq k}(\wt b_{1},\wt b_{2},\cdots,\wt b_{k},0)=-\text{diag}\left(\langle\psi_{b,j},\psi_{b,j}\rangle_{b}\right)_{1\leq j\leq k}=-I_{k}+O(|b|).
\]
Hence, Implicit Function Theorem yields the decomposition. Now, for a technical reason for closing the modulation estimates and closing the bootstrap argument (see Remark~\ref{motivation for introduce b(s)}), we choose an adiabatic parameter $b : [0,\infty) \rightarrow \mathbb{R}$ such that 
\begin{equation}
b(s)\coloneqq A_{k}\frac{e^{-\lambda_{k}s}}{s+1},\label{hyperparameter}
\end{equation}
for $0<A_{k}<b^*(K)$ so that $b(s)$ sufficiently small for $s \in [0,\infty)$. Let $v_{0}$ be 
\[
v_{0}=\sum_{j=1}^{k}\wt b_{j}\psi_{b(0),j}+\wt \epsilon_0.
\]
Then, the solution of \eqref{stefanradial} with initial data is $v_0$ and can be decomposed as
\begin{equation}
v(s,y)=Q_{\beta_k(s)}(y)+\epsilon(s,y),\quad \langle\psi_{b,j},\epsilon\rangle_{b}=0 \text{ for } 1\leq j\leq k.\label{profile set up k>1}
\end{equation}

We now state the bootstrap hypothesis. We first define the second order energy: For $k=1$,
\begin{equation}\label{define energy for k=1}
\epsilon_{2}\coloneqq \mathcal{H}_{b_1}\epsilon,\quad \mathcal{E} \coloneqq \norm{\epsilon_{2}}_{L_{b_1}^{2}}^{2},
\end{equation}
and for $k>1$
\begin{equation}\label{define energy for k>1}
\epsilon_{2} \coloneqq \mathcal{H}_{b}\epsilon,\quad \mathcal{E}\coloneqq \norm{\epsilon_{2}}_{L_{b}^{2}}^{2}.
\end{equation}
The control of $\partial_y \varepsilon$ at the boundary $y = 1$ cannot be achieved solely with $H_b^1$-bounds. Through an integration by parts argument, we see that pointwise estimates near the boundary require $ H^2 $-Sobolev regularity. This justifies the inclusion of $H^2$-control in our bootstrap assumptions (see Lemma~\ref{boundarylemma}). We are now ready to state the main bootstrap proposition.

\begin{proposition}
\label{bootstrap assumptions}(Bootstrap proposition for $b_{j}$ 
and $\mathcal{E}$) Let $b^{*}=b^{*}(K)$ be as in (decomposition)
and $\epsilon$ and $\mathcal{E}$ as above. 

\textbf{Case $k=1$.} With replacing $b^{*}(1)$ by a possibly smaller parameter,
still called $b^{*}(1)$, we have the following. If initial conditions
on $b$ and $\mathcal{E}$ satisfy 
\begin{equation}
\begin{cases}
0<|b_1(0)|<b^{*},\\
\mathcal{E}(0)\leq|b_1(0)|^{3},
\end{cases}\label{initial conditions for k=1}
\end{equation}
then we have for all $s\in[0,\infty)$ 
\begin{equation}
\begin{aligned} & 0<|b_1(s)|<3b^{*},\\
 & \mathcal{E}(s)\leq C_{1}|b_1(s)|^{3}.
\end{aligned}
\label{bootstrap assumption for k=1}
\end{equation} for some positive constant $C_1$.

\textbf{Case $k>1$.} From here, let $A_k = b^*(k)/2$, where $A_k$ is the constant depending on $k > 1$ in \eqref{hyperparameter}. By possibly choosing a smaller value for $b^*(k)$ (while keeping the same notation), we may assume the following.
 If the initial conditions on $b_k$ and $\mathcal{E}$ satisfy
\begin{equation}
\begin{cases}
b(0)<|b_{k}(0)|<b^*,\\
\mathcal{E}(0)\leq b^{2}(0).
\end{cases}\label{initial conditions for k>1}
\end{equation}
then there exist a constant $C_k, D_k>0$ and $(b_j(0))_{1\leq j\leq k-1}$,
\begin{equation*}
    |b_{j}(0)|<b^{*}\text{ for } 1\leq j\leq k-1,
\end{equation*}
depending on $\epsilon(0)$ such that for all $s \in [0,\infty)$,
\begin{equation}
\begin{aligned} & 0<b(s)<|b_{k}(s)|<3b^{*},\quad \sum_{j=1}^{k-1}|V_{j}(s)|^{2}\leq D_k^{2},\\
 & \text{and } \mathcal{E}(s)\leq C_{k}b^{2}(s),
\end{aligned}
\label{bootstrap assumption for k>1}
\end{equation}
where $(V_{j})_{1\leq j\leq k-1}$ is a sequence of functions denoted by
\begin{equation}\label{def of V_j}
V_{j}(s) \coloneqq b_{j}(s)e^{(\lambda_{k}+\eta_k)s},
\end{equation}  
where $0<\eta_k<(\lambda_{k}-\lambda_{k-1})/2$ is a small constant depending on $k$.
\end{proposition}

To prove Proposition~\ref{bootstrap assumptions}, we proceed via a continuity argument. Accordingly, starting from the next subsection, we assume the bootstrap bounds \eqref{bootstrap assumption for k=1} for $k=1$ and \eqref{bootstrap assumption for k>1} for $k>1$, respectively.  Closing the continuity argument is in Section~\ref{Energy bounds and closing bootstrap arguments}.

\subsection{Derivations of leading order ODE and modulation estimates} In this section, we derive the modulation equation to yield leading dynamics. We begin by estimating the parameter $a$.
\begin{lemma}(Boundary conditions)\label{boundarylemma} We have
\begin{equation}
a=-\frac{\lambda_s}{\lambda}=\begin{cases}
-\sqrt{2\lambda_{1}}b_1+O(|b_1|^{3/2}),\quad k=1,\\
\\
(-1)^{k}\sqrt{2\lambda_{k}}b_{k}+O(|b|),\quad k>1.
\end{cases}\label{boundaryconditions}
\end{equation}
\end{lemma} 
\begin{proof}
Recall from the boundary condition that
\[
a = -\frac{\lambda_s}{\lambda} = \partial_{y} Q_{\beta_k}(s,1) + \partial_{y} \epsilon(s,1).
\]
We begin by estimating the term involving $Q_{\beta_k}$. Recall the definition of $\eta_j$ given in \eqref{renormalization of unperturbed eigenfunctions}. Differentiating $\eta_j$ with respect to $y$ at $y = 1$, we obtain
\[
\partial_{y} \eta_j(1) = (-1)^j \sqrt{2 \lambda_j}.
\]
Therefore, from \eqref{decompose of psi_{b,k}}, the differentiation of $\psi_{b,j}$
at $y=1$ for $1\leq j\leq k$ is 
\begin{equation}\label{differentiation of psi_{b,j} at y=1}
\begin{aligned}
\partial_{y}\psi_{b,j}(1) & =\partial_{y}\eta_{j}(1)+\sum_{i=1}^{j-1}\mu_{b,ij}\partial_{y}\eta_{i}(1)+\partial_{y}\tilde{\phi}_{b,j}(1)\\
 & =(-1)^{j}\sqrt{2\lambda_{j}}+O(|b|).
\end{aligned}
\end{equation}
The last equality follows from the estimates on $\mu_{b,ij}$ in \eqref{value of mu_{b,jk}} and the bounds on $\partial_{y} \tilde{\phi}_{b,j}(1)$ in \eqref{estimates for tilde{phi}_{b,k}}. Thus, for $k = 1$, replacing $b$ with $b_1$, the above differentiation yields
\[
\partial_{y} Q_{\beta_1}(1) = -\sqrt{2 \lambda_1} \, b_1 + O(|b_1|^2),
\]
whereas for $k > 1$,
\[
\partial_{y} Q_{\beta_k}(1) = \sum_{j=1}^{k} b_j \left( (-1)^j \sqrt{2 \lambda_j} + O(|b|) \right).
\]
Since $\epsilon=v-Q_{\beta_k}$, we have 
\begin{align*}
\partial_{y}\epsilon|_{y=1} & =\partial_{y}v|_{y=1}-\partial_{y}Q_{\beta}|_{y=1}=-\lambda_{s}/\lambda-\partial_{y}Q_{\beta}|_{y=1}\\
 & =\begin{cases}
a+\sqrt{2\lambda_{1}}b_1+O(|b_1|^{2})\quad\text{ for } k=1,\\
a-\sum_{j=1}^{k}b_{j}\left((-1)^{j}\sqrt{2\lambda_{j}}+O(|b|)\right) \text{ for } k>1.
\end{cases}
\end{align*}
We now estimate $|\partial_{y} \epsilon(1)|$ using the bootstrap assumptions on $\epsilon$. This highlights the necessity of imposing $H^2$-regularity in the bootstrap hypothesis. Indeed, an integration by parts yields
\begin{align*}
\langle \mathcal{H}_{b} \epsilon, y \rangle_{b} 
&= \int_{0}^{1} -\frac{1}{y \rho_{b}} \partial_{y} \left( y \rho_{b} \partial_{y} \epsilon \right) \cdot y \rho_{b} y \, dy \\
&= -\left[ y^2 \rho_{b}(y) \, \partial_{y} \epsilon(y) \right]_{0}^{1} + \int_{0}^{1} \partial_{y} \epsilon \cdot \rho_{b} y \, dy.
\end{align*}
Applying the Cauchy--Schwarz inequality, we obtain the estimate
\[
|\partial_{y} \epsilon(1)| \lesssim \| \mathcal{H}_{b} \epsilon \|_{L_{b}^{2}} + \| \partial_{y} \epsilon \|_{L_{b}^{2}}.
\]
Since $\epsilon(s,1)=0$, the orthogonality condition on $\epsilon$
in \eqref{profile set up k=1} for
$k=1$ and \eqref{profile set up k>1} for $k>1$ respectively with spectral
gap property in \eqref{spectralgap1} yield 
\begin{equation*}
\norm{\mathcal{H}_{b}\epsilon}_{L_{b}^{2}}\norm{\epsilon}_{L_{b}^{2}}\geq\langle\mathcal{H}_{b}\epsilon,\epsilon\rangle_{b}=\norm{\partial_{y}\epsilon}_{L_{b}^{2}}^{2}\geq(\lambda_{k}+O(|b|))\norm{\epsilon}_{L_{b}^{2}}^{2}.
\end{equation*}
Hence, from the bootstrap assumptions on $\mathcal{E}$ in \eqref{bootstrap assumption for k=1},
 \eqref{bootstrap assumption for k>1} yield
\[
|\partial_{y}\epsilon(1)|\lesssim\begin{cases}
|b_1|^{3/2},\quad k=1,\\
|b|,\ \quad k>1.
\end{cases}
\]
Furthermore, for $k > 1$, the bootstrap assumption \eqref{bootstrap assumption for k>1} yields
\[
|b_j(s)| \lesssim e^{-(\lambda_k + \eta_k) s}, \quad \text{for } 1 \leq j < k,
\]
and in particular, $|b_j(s)| \lesssim b(s)$ for all $1 \leq j < k$. This concludes the proof of the lemma.
\end{proof}

From Lemma~\ref{boundarylemma} and the bootstrap assumptions on the parameters $a$, $b$, and $b_j$, we obtain
\begin{equation}\label{asymptotic for a, k=1}
|a| \sim |b_1| 
\end{equation}
for $k = 1$, and for $k > 1$, we have
\begin{equation} \label{order of modulation parameter for k>1}
|a| \sim |b_k| \quad \text{and} \quad |b_j| \lesssim b \quad \text{for } 1 \leq j \leq k-1.
\end{equation}

We now derive Modulation vector Mod, and the profile error $\Psi$. 
\begin{lemma}(Leading order modulation equations)\label{leading order modulation equations}
Under the bootstrap assumption \eqref{asymptotic for a, k=1} for $k = 1$ and \eqref{order of modulation parameter for k>1} for $k > 1$, we have
\begin{equation*}
\partial_{s} Q_{\beta_k} + \mathcal{H}_{a} Q_{\beta_k} = \text{Mod} + \Psi,
\end{equation*}
where the modulation vector $\text{Mod}$ and the profile error $\Psi$ are defined as follows:

\textbf{Case $k=1$.}  
\begin{equation}\label{def of mod k=1}
\text{Mod}\coloneqq \left[(b_1)_{s}+b_1\lambda_{b_1,1}+(a-b_1)b_1\frac{\langle\Lambda\psi_{b_1,1},\psi_{b_1,1}\rangle_{b_1}}{\langle\psi_{b_1,1},\psi_{b_1,1}\rangle_{b_1}}\right]\psi_{b_1,1},
\end{equation}
and the profile error $\Psi$ satisfies the bounds 
\begin{equation}\label{profile error estimate for k=1}
|\langle\Psi,\psi_{b_1,1}\rangle_{b_1}|\lesssim b_1^{2}|(b_1)_{s}|,\quad \norm{\mathcal{H}_{b_1}\Psi}_{L_{b_1}^{2}}\lesssim|b_1(b_1)_{s}|+b_1^{2}.
\end{equation}
\textbf{Case $k>1$.}  
\begin{align}\label{def of mod k>1}
\text{Mod}\coloneqq \sum_{j=1}^{k}\left[(b_{j})_{s}+b_{j}\lambda_{b,j}+ab_{k}\frac{\langle\Lambda\psi_{b,k},\psi_{b,j}\rangle_{b}}{\langle\psi_{b,j},\psi_{b,j}\rangle_{b}}\right]\psi_{b,j}
\end{align}
and the profile error $\Psi$ satisfies the bounds
\begin{equation}\label{profile error estimate for k>1}
|\langle\Psi,\psi_{b,j}\rangle_{b}|\lesssim bb_{k},\quad \norm{\mathcal{H}_{b}\Psi}_{L_{b}^{2}}\lesssim b_{k}^{2},
\end{equation}
for $1\leq j\leq k$. \end{lemma} 
\begin{proof}

\textbf{Case $k=1$.} Recalling that $\mathcal{H}_{a}=\mathcal{H}_{b_1}+(a-b_1)\Lambda$ and $Q_{\beta_1}=b_1\psi_{b_1,1}$
for $k=1$, we have
\begin{align*}
\partial_{s}Q_{\beta_1}+\mathcal{H}_{a}Q_{\beta_1} & =(b_1)_{s}\psi_{b_1,1}+b_1(b_1)_{s}\partial_{b_1}\psi_{b_1,1}+b_1\lambda_{b_1,1}\psi_{b_1,1}+(a-b_1)b_1\Lambda\psi_{b_1,1}.
\end{align*}
We define Mod by
\[
\text{Mod} \coloneqq \left[(b_1)_{s}+b_1\lambda_{b_1,1}+(a-b_1)b_1\frac{\langle\Lambda\psi_{b_1,1},\psi_{b_1,1}\rangle_{b_1}}{\langle\psi_{b_1,1},\psi_{b_1,1}\rangle_{b_1}}\right]\psi_{b_1,1}.
\]
Then, the profile error $\Psi$ becomes
\begin{equation*} \label{psik=1}
    \Psi = b_1 (b_1)_s \, \partial_{b_1} \psi_{b_1,1} + (a - b_1) b_1 \left( \Lambda \psi_{b_1,1} - \frac{\langle \Lambda \psi_{b_1,1}, \psi_{b_1,1} \rangle_{b_1}}{\langle \psi_{b_1,1}, \psi_{b_1,1} \rangle_{b_1}} \, \psi_{b_1,1} \right).
\end{equation*}
From \eqref{IVTesitmates}, and noting that the second term on the right-hand side of \eqref{psik=1} is orthogonal to $\psi_{b_1,1}$, we obtain
\begin{equation*}
|\langle \Psi, \psi_{b_1,1} \rangle_{b_1}| \lesssim b_1^2 |(b_1)_s|.
\end{equation*}
Also, using \eqref{e:5} together with the estimates from \eqref{decompose of psi_{b,k}}, \eqref{value of partial_b mu_{b,jk}}, \eqref{asymptotic expansion of lambda_{b,k}}, and \eqref{estimates for tilde{phi}_{b,k}}, we obtain
\begin{equation} \label{part3:1}
\| \mathcal{H}_{b}(b \partial_{b} \psi_{b,1}) \|_{L_{b}^{2}} 
\leq \| b \Lambda \psi_{b,1} \|_{L_{b}^{2}} 
+ |b| \, |\partial_{b} \lambda_{b,1}| \, \| \psi_{b,1} \|_{L_{b}^{2}} 
+ |b| \, |\lambda_{b,1}| \, \| \partial_{b} \psi_{b,1} \|_{L_{b}^{2}} 
\lesssim |b|.
\end{equation}
By replacing $b$ with $b_1$ in \eqref{part3:1}, and using the asymptotic relation $|a| \sim |b_1|$ from \eqref{asymptotic for a, k=1}, we deduce
\begin{equation*}
\| \mathcal{H}_{b_1} \Psi \|_{L_{b_1}^{2}} \lesssim |b_1 (b_1)_s| + b_1^2.
\end{equation*}

\textbf{Case $k>1$.} From the definition of $Q_{\beta_k}$ in \eqref{profile Q_{beta_k}}, we have 
\begin{align*}
\partial_{s}Q_{\beta_k}+\mathcal{H}_{a}Q_{\beta_k} & =\sum_{j=1}^{k}\left\{ \left((b_{j})_{s}+b_{j}\lambda_{b,j}\right)\psi_{b,j}+b_{s}\left(\frac{b_{j}}{b}\right)(b\partial_{b}\psi_{b,j})+(a-b)b_{j}\Lambda\psi_{b,j}\right\} .
\end{align*}
We define Mod by 
\begin{align*}
\text{Mod} \coloneqq \sum_{j=1}^{k}\left[(b_{j})_{s}+b_{j}\lambda_{b,j}+ab_{k}\frac{\langle\Lambda\psi_{b,k},\psi_{b,j}\rangle_{b}}{\langle\psi_{b,j},\psi_{b,j}\rangle_{b}}\right]\psi_{b,j}.
\end{align*}
Then, profile error $\Psi$ is 
\begin{equation}\label{profile error for k>1}
\begin{aligned}
\Psi & =\sum_{j=1}^{k-1}\left\{ b_{s}\left(\frac{b_{j}}{b}\right)(b\partial_{b}\psi_{b,j})+(a-b)b_{j}\Lambda\psi_{b,j}\right\} +b_{s}\left(\frac{b_{k}}{b}\right)(b\partial_{b}\psi_{b,k})\\
 &\quad +ab_{k}\left(\Lambda\psi_{b,k}-\sum_{j=1}^{k}\frac{\langle\Lambda\psi_{b,k},\psi_{b,j}\rangle_{b}}{\langle\psi_{b,j},\psi_{b,j}\rangle_{b}}\psi_{b,j}\right)-bb_{k}\Lambda\psi_{b,k}.
\end{aligned}
\end{equation}
From \eqref{hyperparameter}, we have the bound $|b_s| \lesssim b$. Moreover, since the third term on the right-hand side of \eqref{profile error for k>1} is orthogonal to $\psi_{b,j}$ for $1 \leq j \leq k$, it follows from \eqref{order of modulation parameter for k>1} that
\begin{align*}
|\langle \Psi, \psi_{b,j} \rangle_{b}| \lesssim b |b_k|.
\end{align*}
Finally, applying \eqref{part3:1}, we estimate
\begin{equation*}
\norm{\mathcal{H}_{b} \Psi}_{L_{b}^{2}} \lesssim b_k^2.
\end{equation*}
\end{proof}

We now turn to deriving modulation estimates for the parameters $(b_j)_{1 \leq j \leq k}$.

\begin{proposition}(Modulation estimates for $b_{j}$)\label{Modulation equation proposition}

\textbf{Case $k=1$.} The $b_1$ law is given by 
\begin{equation}\label{modulation estimates for k=1}
|(b_1)_{s}+\lambda_{1}b_1+\sqrt{2\lambda_{1}}b_1^{2}|\lesssim|b_1|^{5/2}.
\end{equation}
\textbf{Case $k>1$.} The modulation dynamical system for the vector $(b_{j})_{1\leq j\leq k}$
is given by 
\begin{equation}\label{modulation esitmates for k>1}
\begin{aligned}\left|(b_{k})_{s}+\lambda_{k}b_{k}+\right. & \left.(-1)^{k+1}\sqrt{2\lambda_{k}}b_{k}^{2}\right|\\
 & +\sum_{j=1}^{k-1}\left|(b_{j})_{s}+\lambda_{j}b_{j}+(-1)^{k}\sqrt{2\lambda_{k}}b_{k}^{2}\langle\Lambda\eta_{k},\eta_{j}\rangle_{0}\right|\lesssim b|b_{k}|.
\end{aligned}
\end{equation}
\end{proposition} 
\begin{proof}
Since, for each $k \in \mathbb{N}$, the function $v = Q_{\beta_k} + \epsilon$ is a solution to \eqref{renormalised stefan}, $\epsilon$ satisfies
\begin{equation} \label{part2:2}
\partial_{s} \epsilon + \mathcal{H}_{a} \epsilon = -\text{Mod} - \Psi,
\end{equation}
where the modulation vector $\text{Mod}$ and the error term $\Psi$ are as defined in Lemma~\ref{leading order modulation equations}.

\textbf{Case $k=1$.}  
We take the $\langle \cdot, \cdot \rangle_{b_1}$ inner product of equation \eqref{part2:2} with $\psi_{b_1,1}$ on both sides:
\begin{align*}
-\langle \epsilon, \partial_{s} \psi_{b_1,1} \rangle_{b_1} + \frac{(b_1)_s}{2} \langle \epsilon, |y|^{2} \psi_{b_1,1} \rangle_{b_1} 
= -\langle \text{Mod} + \Psi, \psi_{b_1,1} \rangle_{b_1} - (a - b_1) \langle \Lambda \epsilon, \psi_{b_1,1} \rangle_{b_1}.
\end{align*}
Now, since $\epsilon$ is orthogonal to $\psi_{b_1,1}$, an integration by parts applied to $\langle \Lambda \epsilon, \psi_{b_1,1} \rangle_{b_1}$ yields
\begin{align*}
-\langle \epsilon, \partial_{s} \psi_{b_1,1} \rangle_{b_1} 
&+ \frac{(b_1)_s}{2} \langle \epsilon, |y|^{2} \psi_{b_1,1} \rangle_{b_1} 
+ (a - b_1) \langle \Lambda \epsilon, \psi_{b_1,1} \rangle_{b_1} \\
&= \langle \epsilon, 
- (b_1)_s \partial_{b_1} \psi_{b_1,1} 
+ \frac{(b_1)_s}{2} |y|^{2} \psi_{b_1,1} 
+ (a - b_1) \left[ -\Lambda \psi_{b_1,1} + b_1 y^{2} \psi_{b_1,1} \right] \rangle_{b_1}.
\end{align*}
From the bootstrap bound on $\mathcal{E}$ in \eqref{profile set up k=1} and the spectral gap estimate \eqref{spectralgap1}, we have $\| \epsilon \|_{L^2_b} \leq \sqrt{\mathcal{E}}$. Moreover, since $\langle \epsilon, \psi_{b_1,1} \rangle_{b_1} = 0$, it follows that
\[
\| \epsilon \|_{L_{b_1}^{2}} = O(|b_1|^{3/2}).
\]
Therefore, using the estimates in \eqref{IVTesitmates}, \eqref{asymptotic for a, k=1} together with the Cauchy--Schwarz inequality, we obtain
\begin{align*}
&\left| \langle \epsilon,\,
-(b_1)_s \partial_{b_1} \psi_{b_1,1}
+ \frac{(b_1)_s}{2} |y|^2 \psi_{b_1,1}
+ (a - b_1)\left[ -\Lambda \psi_{b_1,1} + b_1 y^2 \psi_{b_1,1} \right]
\rangle_{b_1} \right| \\
&\lesssim \left( |(b_1)_s| + |b_1| \right) |b_1|^{3/2}.
\end{align*}
Combining this with the estimate for $\langle \Psi, \psi_{b_1,1} \rangle_{b_1}$ in \eqref{profile error estimate for k=1} and \eqref{def of mod k=1}, we obtain
\begin{align*}
\frac{|\langle \text{Mod},\psi_{b_1,1}\rangle_{b_1}|}{\langle\psi_{b_1,1},\psi_{b_1,1}\rangle_{b_1}}=|(b_1)_{s}+\lambda_{b_1,1}b_1+&(a-b_1)b_1\frac{\langle\Lambda\psi_{b_1,1},\psi_{b_1,1}\rangle_{b_1}}{\langle\psi_{b_1,1},\psi_{b_1,1}\rangle_{b_1}}|\\
&\lesssim\ (|(b_1)_{s}|+|b_1|)|b_1|^{3/2}+b_1^{2}|(b_1)_{s}|.
\end{align*}
From \eqref{asymptotic expansion of lambda_{b,k}}, \eqref{boundaryconditions}, \eqref{Norm of psi}, and \eqref{IVTesitmates}, together with an integration by parts, we compute
\begin{equation} \label{prop3.4:eq1}
\begin{aligned}
\langle \Lambda \psi_{b_1,1}, \psi_{b_1,1} \rangle_{b_1} 
&= \int_{0}^{1} y  \partial_{y} \psi_{b_1,1}  \psi_{b_1,1}  \rho_{b_1}(y)  y  dy \\
&= \int_{0}^{1} y  \partial_{y} \eta_1  \eta_1  y  dy + O(|b_1|) \\
&= \left[ \tfrac{1}{2} \eta_1^2(y) y^2 \right]_0^1 - \int_0^1 \eta_1^2(y)  dy + O(|b_1|).
\end{aligned}
\end{equation}
As a result, we obtain the estimate
\begin{equation} \label{part3eq001}
\left| (b_1)_s + \lambda_1 b_1 + \sqrt{2 \lambda_1} b_1^2 + O(|b_1|^{5/2}) \right| 
\lesssim \left( |(b_1)_s| + |b_1| \right) |b_1|^{3/2} + b_1^2 |(b_1)_s|.
\end{equation}
Finally, from \eqref{part3eq001}, we have the bounds for $|(b_1)_{s}|$: 
\[
|(b_1)_{s}|\lesssim|b_1|.
\]
This proves Proposition~\ref{Modulation equation proposition} for $k=1$.

\textbf{Case $k>1$.}
By taking the $\langle \cdot, \cdot \rangle_b$ inner product of both sides of \eqref{part2:2} with $\psi_{b,j}$ for $1 \leq j \leq k$, and performing a similar computation as in \textbf{Step 1}, we deduce from \eqref{order of modulation parameter for k>1}, \eqref{bootstrap assumption for k>1}, and \eqref{Norm of psi}, along with the bound $|b_s| \lesssim b$ from \eqref{hyperparameter}, that
\begin{equation} \label{mod estimate for k>1}
\left| \langle \epsilon,\, -b_s \partial_b \psi_{b,j} + \frac{b_s}{2} |y|^2 \psi_{b,j} + (a - b)\left[ -\Lambda \psi_{b,j} + b y^2 \psi_{b,j} \right] \rangle_b \right| \lesssim b |b_k|.
\end{equation}
Combining \eqref{mod estimate for k>1} with \eqref{profile error estimate for k>1} and \eqref{def of mod k>1}, we obtain, for $1 \leq j \leq k$,
\[
\frac{|\langle \text{Mod}, \psi_{b,j} \rangle_b|}{\langle \psi_{b,j}, \psi_{b,j} \rangle_b}
= \left| (b_j)_s + \lambda_{b,j} b_j + a b_k \frac{\langle \Lambda \psi_{b,k}, \psi_{b,j} \rangle_b}{\langle \psi_{b,j}, \psi_{b,j} \rangle_b} \right| \lesssim b |b_k|.
\]
Then, using \eqref{asymptotic expansion of lambda_{b,k}} and \eqref{Norm of psi}, we obtain
\begin{align} \label{part3eq002}
\sum_{j=1}^{k} \left| (b_j)_s + \lambda_j b_j + a b_k \langle \Lambda \eta_k, \eta_j \rangle_0 \right| \lesssim b |b_k|.
\end{align}
Finally, substituting the expression for $a$ from \eqref{boundaryconditions} into \eqref{part3eq002}, and applying the same computation as in \eqref{prop3.4:eq1}, yields the desired result.
\end{proof}

\subsection{Energy bounds and closing the bootstrap argument}
\label{Energy bounds and closing bootstrap arguments}
In this section, we complete the continuity argument by deriving suitable energy bounds and proving Proposition~\ref{bootstrap assumptions}. Our main goal is to establish pointwise control of the energy functional $\mathcal{E}$ with respect to the rescaled time variable $s$, as stated in Proposition~\ref{pointwise energy control}. This estimate allows us to close the bootstrap argument and thereby conclude the proof.

\begin{proposition}\label{pointwise energy control}
(Energy bound) We have the pointwise control: 1. $k=1$ : 
\begin{equation}
\frac{1}{2}\frac{d}{ds}\left\{ \mathcal{E}+O(|b_1|^{7/2})\right\} +\lambda_{2}\mathcal{E}\lesssim|b_1|^{7/2}.\label{Energyestimatek=1}
\end{equation}
2. $k>1$ : 
\begin{equation}
\frac{1}{2}\frac{d}{ds}\left\{ \mathcal{E}+O(|b_{k}|^{3})\right\} +\lambda_{k+1}\mathcal{E}\lesssim b_{k}^{3}\label{Energyestimatek>1}
\end{equation}
\end{proposition}

\begin{remark}
Before proving the pointwise energy bound for $\mathcal{E}$, we briefly outline the main steps of the argument. Although the setting is different, our method is based on the approach developed in \cite{Had-Raph}.
Starting from the well-known commutator identity
\[
[\Delta, \Lambda] = 2 \Delta, \quad \text{where $[\cdot, \cdot]$ denotes the commutator},
\]
we rewrite equation \eqref{renormalised stefan} in terms of the higher-order quantity $\epsilon_2 \coloneqq \mathcal{H}_b \epsilon$. This allows us to derive an identity for the time derivative of the energy $\mathcal{E}(s)$. The spectral gap property then provides coercivity of the $L^2_b$-norm of $\epsilon_2$, which plays a key role in closing the bootstrap argument. A technical difficulty arises in estimating the boundary term involving $\partial_y \epsilon_2$, which does not vanish at $y = 1$ due to the highest-order derivative in the equation. This term involves the derivative $a_s$, but can be controlled using the boundary condition \eqref{boundaryconditions} and the modulation estimates \eqref{modulation estimates for k=1} and \eqref{modulation esitmates for k>1} as in \cite{Had-Raph}. See the estimate for $|e^{-b} \epsilon_2(1) \partial_y \epsilon_2(1)|$ in the proof of Proposition~\ref{pointwise energy control} for further details.
\end{remark}

\begin{proof}[Proof of Proposition~\ref{pointwise energy control}]
We present the proof in the case $k > 1$, since the argument for $k = 1$ follows analogously by replacing $b$ with $b_1$. The only distinction arises in handling the highest-order boundary term $\partial_y \epsilon_2(s,1)$, for which we will treat the cases $k = 1$ and $k > 1$ separately.

First, recall that $\epsilon$ satisfies \eqref{part2:2}, and define $\epsilon_2 \coloneqq \mathcal{H}_b \epsilon$, which then satisfies the equation
\[
\partial_s \epsilon_2 + \mathcal{H}_a \epsilon_2 = [\partial_s, \mathcal{H}_b] \epsilon + [\mathcal{H}_a, \mathcal{H}_b] \epsilon + \mathcal{H}_b \mathcal{F},
\]
where $\mathcal{F} \coloneqq -\text{Mod} - \Psi$. Using the algebraic identity $[\Delta, \Lambda] = 2\Delta$, we compute the commutators $[\partial_s, \mathcal{H}_b]$ and $[\mathcal{H}_a, \mathcal{H}_b]$ as in \cite{Had-Raph}.
We now compute the commutator terms. Using the commutator identity
\begin{align*}
[\partial_s, \mathcal{H}_b] \epsilon + [\mathcal{H}_a, \mathcal{H}_b] \epsilon 
&= b_s \Lambda \epsilon + [\mathcal{H}_b + (a - b) \Lambda, \mathcal{H}_b] \epsilon \\
&= b_s \Lambda \epsilon + (a - b) [\Lambda, -\Delta] \epsilon \\
&= b_s \Lambda \epsilon + 2(a - b) \Delta \epsilon \\
&= (b_s + 2b(a - b)) \Lambda \epsilon - 2(a - b)[ -\Delta \epsilon + b \Lambda \epsilon ].
\end{align*}
For convenience, we introduce the notation
\[
\Phi \coloneqq b_s + 2b(a - b).
\]
Hence, the equation satisfied by $\epsilon_2$ takes the form
\begin{equation} \label{epsilon_2eq}
\partial_s \epsilon_2 + \mathcal{H}_a \epsilon_2 = \Phi \Lambda \epsilon - 2(b - a) \epsilon_2 + \mathcal{H}_b \mathcal{F}.
\end{equation}
Now multiply $\epsilon_{2}$ on both sides of \eqref{epsilon_2eq}
and integrate with respect to the variable $y$. This gives the modified energy identities : 
\begin{equation}\label{differentiation of E}
\begin{aligned}\frac{1}{2}\frac{d}{ds}\mathcal{E} & =\frac{1}{2}\frac{d}{ds}\int_{0}^{1}\epsilon_{2}^{2}\rho_{b} y\ dy\ \\
 & =\frac{1}{2}\int_{0}^{1}\left\{ 2\epsilon_{2}\partial_{s}\epsilon_{2}\rho_{b}+\epsilon_{2}^{2}\left(-\frac{b_{s}}{2}y^{2}\right)\rho_{b}y\right\} \ dy\\
 & =-\frac{b_{s}}{4}\norm{y\epsilon_{2}}_{L_{b}^{2}}^{2}+\langle\epsilon_{2},\partial_{s}\epsilon_{2}\rangle_{b}\\
 & =-\frac{b_{s}}{4}\norm{y\epsilon_{2}}_{L_{b}^{2}}^{2}+\langle\Phi\Lambda\epsilon-2(b-a)\epsilon_{2}+\mathcal{H}_{b}\mathcal{F}-\mathcal{H}_{a}\epsilon_{2},\epsilon_{2}\rangle_{b}.
\end{aligned}
\end{equation}
On the other hand, an integration by parts yields
\begin{equation} \label{< H_a epsilon_2, epsilon_2 >_b}
\begin{aligned}
- \langle \mathcal{H}_a \epsilon_2, \epsilon_2 \rangle_b 
&= - \int_0^1 \epsilon_2 \, \mathcal{H}_a \epsilon_2 \, \rho_b y \, dy \\
&= - \int_0^1 \epsilon_2 \left( \mathcal{H}_b \epsilon_2 + (a - b) \Lambda \epsilon_2 \right) \rho_b y \, dy \\
&= \int_0^1 \partial_y (\rho_b y \partial_y \epsilon_2) \epsilon_2 \, dy 
+ (b - a) \int_0^1 \epsilon_2 \, y \partial_y \epsilon_2 \, \rho_b y \, dy \\
&= [\rho_b y (\partial_y \epsilon_2) \epsilon_2 ]_0^1 
- \int_0^1 |\partial_y \epsilon_2|^2 \, \rho_b y \, dy 
+ (b - a) \langle \epsilon_2, \Lambda \epsilon_2 \rangle_b \\
&= \rho_b(1) \, \partial_y \epsilon_2(1) \, \epsilon_2(1) 
- \| \partial_y \epsilon_2 \|_{L_b^2}^2 
+ (b - a) \langle \epsilon_2, \Lambda \epsilon_2 \rangle_b.
\end{aligned}
\end{equation}
Injecting \eqref{< H_a epsilon_2, epsilon_2 >_b} into \eqref{differentiation of E}, we obtain
\begin{equation} \label{derivative of E}
\begin{aligned}
\frac{1}{2} \frac{d}{ds} \mathcal{E} = 
& - \| \partial_y \epsilon_2 \|_{L_b^2}^2 
+ (b - a) \langle \epsilon_2, \Lambda \epsilon_2 \rangle_b 
+ \rho_b(1) \, \partial_y \epsilon_2(1) \, \epsilon_2(1) \\
& - \frac{b_s}{4} \| y \epsilon_2 \|_{L_b^2}^2 
+ \Phi \langle \Lambda \epsilon, \epsilon_2 \rangle_b 
- 2(b - a) \| \epsilon_2 \|_{L_b^2}^2 
+ \langle \mathcal{H}_b \mathcal{F}, \epsilon_2 \rangle_b.
\end{aligned}
\end{equation}
We now estimate each term on the right-hand side of \eqref{derivative of E}.

\textbf{1.} Estimates for $|(\mathcal{H}_{b}\epsilon)(1)|$: Since $v$ is a solution to \eqref{renormalised stefan}, we have
\[
(\partial_s + \mathcal{H}_a) v(s, y) = 0 \quad \text{for } |y| \leq 1.
\]
Moreover, using the boundary condition $v(s,1) = 0$, it follows that $\partial_s v(s,1) = 0$. Evaluating the equation at the boundary point $y = 1$, we obtain
\begin{equation}\label{H_b epsilon(s,1) = -a(a-b)}
0 = (\mathcal{H}_a v)(s,1) = (\mathcal{H}_b + (a - b)\Lambda)v(s,1) = (\mathcal{H}_b \epsilon)(s,1) + (a - b)a.
\end{equation}
Therefore, by the asymptotic relations in \eqref{asymptotic for a, k=1} and \eqref{order of modulation parameter for k>1}, we conclude that
\begin{equation} \label{estimates for H_b(1)}
\begin{aligned}
|(\mathcal{H}_{b_1} \epsilon)(1)| &\lesssim b_1^2 \quad \text{for } k = 1, \\
|(\mathcal{H}_{b} \epsilon)(1)| &\lesssim b_k^2 \quad \text{for } k > 1.
\end{aligned}
\end{equation}

\textbf{2.} Estimates for $\norm{\partial_{y}\epsilon_{2}}_{L_{b}^{2}}^{2}$
: Define $\xi, \ \zeta \in H_0^1(B_1(0))$ by
\[
\xi(y) \coloneqq (\mathcal{H}_b \epsilon)(y) - (\mathcal{H}_b \epsilon)(1), \quad \text{for } |y| \leq 1,
\]
\[
\zeta(y) \coloneqq \xi(y) - \sum_{j=1}^{k} \frac{\langle \xi, \eta_j \rangle_0}{\langle \eta_j, \eta_j \rangle_0} \eta_j.
\]
Since $\langle \zeta, \eta_j \rangle_0 = 0$ for $1 \leq j \leq k$, the spectral gap property for $\{ \eta_j \}$ in \eqref{spectralgapofeta} yields
\begin{equation} \label{part3spec1}
\int_0^1 |\partial_y \zeta|^2 \rho_b y \, dy \geq (\lambda_{k+1} + O(|b|)) \int_0^1 |\zeta|^2 \rho_b y \, dy.
\end{equation}
On the other hand, from the orthogonality conditions on $\epsilon$ in \eqref{profile set up k=1} and \eqref{profile set up k>1}, we have $\langle \psi_{b,j}, \epsilon \rangle_b = 0$. Hence,
\[
\langle \mathcal{H}_b \epsilon, \psi_{b,j} \rangle_b = 0,
\]
and from \eqref{decompose of psi_{b,k}}, \eqref{value of mu_{b,jk}}, and \eqref{estimates for tilde{phi}_{b,k}}, we know that
\[
\| \psi_{b,j} - \eta_j \|_{H_b^1} \lesssim |b|.
\]
Thus, for $1 \leq j \leq k$, and possibly after reducing $b^*(k) > 0$ while keeping the same notation, we obtain the estimate for all $|b| < b^*(k)$:
\[
|\langle \xi, \eta_j \rangle_0| 
\lesssim |b| \| \xi \|_{L_b^2} + |\langle \psi_{b,j}, \xi \rangle_b| 
\lesssim |b| \| \xi \|_{L_b^2} + |(\mathcal{H}_b \epsilon)(1)|.
\]
Consequently, we deduce
\begin{equation}\label{H^1 norm of xi-zeta estimate}
\| \xi - \zeta \|_{H_b^1} 
\leq \sum_{j=1}^{k} |\langle \xi, \eta_j \rangle_0| \cdot \| \eta_j \|_{H_0^1}
\lesssim_k |b| \| \xi \|_{L_b^2} + |(\mathcal{H}_b \epsilon)(1)|.
\end{equation}
Applying \eqref{H^1 norm of xi-zeta estimate} to \eqref{part3spec1} and possibly after reducing $b^*(k) > 0$ while keeping the same notation, we estimate
\begin{equation} \label{energy1}
\begin{aligned}
\| \partial_y\xi \|_{L_b^2} 
&\geq - \| \partial_y(\xi - \zeta) \|_{L_b^2} + \| \partial_y\zeta \|_{L_b^2} \\
&\geq (\lambda_{k+1} + O(|b|))^{1/2} \| \zeta \|_{L_b^2} 
- C_{k,1} \left( |b| \| \xi \|_{L_b^2} + |(\mathcal{H}_b \epsilon)(1)| \right) \\
&\geq (\lambda_{k+1} + O(|b|))^{1/2} \| \xi \|_{L_b^2} 
- C_{k,2} |b| \| \xi \|_{L_b^2} 
- C_{k,3} |(\mathcal{H}_b \epsilon)(1)| \\
&= (\lambda_{k+1} + O(|b|))^{1/2} \| \mathcal{H}_b \epsilon \|_{L_b^2} 
- C_{k,4} |(\mathcal{H}_b \epsilon)(1)|.
\end{aligned}
\end{equation}
for some constants $C_{k,1}, C_{k,2}, C_{k,3}, C_{k,4} > 0$.

Now, using the bootstrap bound on $\mathcal{E}$ in \eqref{bootstrap assumption for k=1} for $k=1$ and \eqref{bootstrap assumption for k>1} for $k>1$, the pointwise estimate in \eqref{estimates for H_b(1)}, and the decay estimate $b_k(s) = O(e^{-\lambda_k s})$ from Proposition~\ref{Modulation equation proposition}, we see that the right-hand side of \eqref{energy1} is non-negative. Squaring both sides of \eqref{energy1}, we deduce
\[
\| \partial_y \mathcal{H}_b \epsilon \|_{L_b^2}^2 
\geq (\lambda_{k+1} + O(|b|)) \| \mathcal{H}_b \epsilon \|_{L_b^2}^2 
- C |(\mathcal{H}_b \epsilon)(1)| \cdot \| \mathcal{H}_b \epsilon \|_{L_b^2}
\]
for some constant $C > 0$. Therefore, we conclude the following estimate:
\begin{equation} \label{estimates for d_y H_b e}
\begin{aligned}
\| \partial_y \mathcal{H}_{b_1} \epsilon \|_{L_{b_1}^2}^2 
&\geq (\lambda_2 + O(|b_1|)) \mathcal{E} + O(|b_1|^{7/2}), \quad &&\text{for } k = 1, \\
\| \partial_y \mathcal{H}_b \epsilon \|_{L_b^2}^2 
&\geq (\lambda_{k+1} + O(|b|)) \mathcal{E} + O(b b_k^2), \quad &&\text{for } k > 1.
\end{aligned}
\end{equation}

\textbf{3.} Estimates for $\langle\epsilon_{2},\Lambda\epsilon_{2}\rangle_{b}$
: Integration by parts yields
\begin{align*}
\langle \epsilon_2, \Lambda \epsilon_2 \rangle_b 
&= \int_0^1 \epsilon_2 \, y \partial_y \epsilon_2 \, \rho_b y \, dy \\
&= \left[ \frac{1}{2} |\epsilon_2(y)|^2 \rho_b(y) y^2 \right]_{0}^{1} 
    - \frac{1}{2} \int_0^1 |\epsilon_2(y)|^2 \left( 2y + y^2(-b y) \right) \rho_b(y) \, dy.
\end{align*}
Applying the Cauchy--Schwarz inequality, we obtain
\[
|\langle \epsilon_2, \Lambda \epsilon_2 \rangle_b| 
\lesssim |\epsilon_2(1)|^2 + \| \epsilon_2 \|_{L_b^2}^2.
\]
Now, using the bootstrap assumptions on $\mathcal{E}$ in \eqref{bootstrap assumption for k=1} for $k = 1$ and in \eqref{bootstrap assumption for k>1} for $k > 1$, together with the pointwise estimate \eqref{estimates for H_b(1)}, we deduce
\begin{equation} \label{estimates for <e, Lam e>}
\begin{aligned}
|\langle \epsilon_2, \Lambda \epsilon_2 \rangle_{b_1}| &\lesssim b_1^4 + b_1^3, \quad &&\text{for } k = 1, \\
|\langle \epsilon_2, \Lambda \epsilon_2 \rangle_b| &\lesssim b_k^4 + b^2, \quad &&\text{for } k > 1.
\end{aligned}
\end{equation}

\textbf{4.} Estimates for $\langle\mathcal{H}_{b}\mathcal{F},\epsilon_{2}\rangle_{b}$
: Recall that $\mathcal{F} = -\text{Mod} - \Psi$. Using the self-adjointness of $\mathcal{H}_b$ in $H_b^1$, the orthogonality conditions on $\epsilon$ from \eqref{profile set up k=1} and the profile error estimate \eqref{profile error estimate for k=1} in the case $k = 1$, and from \eqref{profile set up k>1} and \eqref{profile error estimate for k>1} in the case $k > 1$, we obtain the following estimates:
\begin{equation} \label{estimatesfor<H_{b}F,e_{2}>}
\begin{aligned}
|\langle \mathcal{H}_{b_1} \mathcal{F}, \epsilon_2 \rangle_{b_1}| 
&= |\langle \mathcal{H}_{b_1} \Psi, \epsilon_2 \rangle_{b_1}| 
\lesssim |b_1|^{7/2}, \quad &&\text{for } k = 1, \\
|\langle \mathcal{H}_b \mathcal{F}, \epsilon_2 \rangle_b| 
&\lesssim b b_k^2, \quad &&\text{for } k > 1.
\end{aligned}
\end{equation}

\textbf{5.} Estimates for $|e^{-b}\epsilon_{2}(1)\partial_{y}\epsilon_{2}(1)|$
: From the boundary condition $a = \partial_y v(s,1)$, we differentiate in time to obtain
\[
a_s = \partial_s \partial_y v(s,1).
\]
On the other hand, differentiating equation \eqref{renormalised stefan} with respect to $y$ yields
\begin{equation} \label{highest order eq}
0 = \partial_s \partial_y v + \partial_y\left( \mathcal{H}_b + (a - b)\Lambda \right)v 
= \partial_s \partial_y v + \partial_y \epsilon_2 + \partial_y \mathcal{H}_b Q_{\beta} + (a - b) y \Delta v.
\end{equation}
Evaluating \eqref{highest order eq} at $y = 1$, we use the boundary conditions: $\partial_s v(s,1) = 0$, $\partial_y v(s,1) = a$, and the fact that $v$ solves \eqref{renormalised stefan}, which implies $(\mathcal{H}_a v)(s,1) = 0$. Hence,
\[
\Delta v(s,1) = a \partial_y v(s,1) = a^2.
\]
Thus, we obtain the following expression for the boundary derivative:
\begin{equation} \label{a_s estimate}
\partial_y \epsilon_2(1) = -\begin{cases}
a_s + b_1 \lambda_{b_1,1} \partial_y \psi_{b_1,1}(1) + a^2(a - b_1), & \text{for } k = 1, \\
a_s + \sum_{j=1}^k b_j \lambda_{b,j} \partial_y \psi_{b,j}(1) + a^2(a - b), & \text{for } k > 1.
\end{cases}
\end{equation}
The term $a_s$ is delicate, as it involves a third-order differentiation at the boundary. However, it can be controlled using the modulation estimates from \eqref{modulation estimates for k=1} and \eqref{modulation esitmates for k>1}.

For $k=1$ : From \eqref{a_s estimate}, \eqref{modulation estimates for k=1}, \eqref{asymptotic for a, k=1}, and \eqref{differentiation of psi_{b,j} at y=1}, we obtain
\begin{align*}
\partial_y \epsilon_2(s,1) &= -a_s - \lambda_{b_1,1} b_1 \partial_y \psi_{b_1,1}(1) - a^2(a - b_1) \\
&= -\left(a + \sqrt{2\lambda_1} b_1 \right)_s + \sqrt{2\lambda_1} \left((b_1)_s + \lambda_{b_1,1} b_1 \right) - a^2(a - b_1) \\
&= -\left(a + \sqrt{2\lambda_1} b_1 \right)_s + O(b_1^2).
\end{align*}
Thus, using \eqref{H_b epsilon(s,1) = -a(a-b)} and \eqref{differentiation of psi_{b,j} at y=1}, we estimate
\begin{align*}
&e^{-b_1(s)} \epsilon_2(s,1) \partial_y \epsilon_2(s,1)\\
&= e^{-b_1} a(a - b_1) \left\{ \left(a + \sqrt{2\lambda_1} b_1 \right)_s + O(b_1^2) \right\} \\
&= e^{-b_1} (a + \sqrt{2\lambda_1} b_1 - b_1 - \sqrt{2\lambda_1} b_1)(a + \sqrt{2\lambda_1} b_1 - \sqrt{2\lambda_1} b_1) \\
&\quad \times \left\{ \left(a + \sqrt{2\lambda_1} b_1 \right)_s + O(b_1^2) \right\} \\
&= \frac{d}{ds} \left\{ e^{-b_1} \frac{(a + \sqrt{2\lambda_1} b_1)^3}{3} 
- e^{-b_1} (1 + 2\sqrt{2\lambda_1}) b_1 \frac{(a + \sqrt{2\lambda_1} b_1)^2}{2} \right. \\
&\quad \left. - e^{-b_1} (1 + \sqrt{2\lambda_1}) \sqrt{2\lambda_1} b_1^2 (a + \sqrt{2\lambda_1} b_1) \right\} + e^{-b_1} (b_1)_s \frac{(a + \sqrt{2\lambda_1} b_1)^3}{3} \\
&\quad + (1 + 2\sqrt{2\lambda_1}) e^{-b_1} \frac{(a + \sqrt{2\lambda_1} b_1)^2}{2} \left\{ (b_1)_s - b_1 (b_1)_s \right\} \\
&\quad + (1 + \sqrt{2\lambda_1}) \sqrt{2\lambda_1} e^{-b_1} \left\{ 2b_1 (b_1)_s - b_1^2 (b_1)_s \right\} (a + \sqrt{2\lambda_1} b_1) \\
&= \frac{d}{ds} \left\{ O(|b_1|^{7/2}) \right\} + O(|b_1|^{7/2}),
\end{align*}
after possibly reducing $b^*(1) > 0$ while maintaining the same notation. The last equality follows from the decay $|(b_1)_s| \lesssim |b_1|$ from \eqref{modulation estimates for k=1}, which implies $b_1(s) = O(e^{-\lambda_1 s})$, together with the boundary estimate \eqref{boundaryconditions}:
\[
|a + \sqrt{2\lambda_1} b_1| \lesssim |b_1|^{3/2}.
\]
For $k>1$: from \eqref{a_s estimate} and \eqref{order of modulation parameter for k>1}, we estimate
\begin{align*}
e^{-b} \, \partial_y \epsilon_2(1) \, \epsilon_2(1) 
&= e^{-b} a(a - b) \left\{ a_s + \sum_{j=1}^k \lambda_{b,j} b_j \, \partial_y \psi_{b,j}(1) + a^2(a - b) \right\} \\
&= e^{-b} a(a - b) \left\{ (a - b)_s + O(|b_k|) \right\} \\
&= e^{-b} \left\{ (a - b)^2 (a - b)_s + b(a - b)(a - b)_s \right\} + O(|b_k|^3) \\
&= \frac{d}{ds} \left\{ e^{-b} \frac{(a - b)^3}{3} + e^{-b} \frac{b(a - b)^2}{2} \right\} \\
&\quad + e^{-b} b_s \frac{(a - b)^3}{3} + e^{-b} \frac{(a - b)^2}{2}(b b_s - b_s) + O(|b_k|^3) \\
&= \frac{d}{ds} \left\{ O(|b_k|^3) \right\} + O(|b_k|^3),
\end{align*}
where in the last equality, we used the bounds $|b_s| \lesssim b$ and $|a - b| \lesssim |b_k|$ from \eqref{hyperparameter} and \eqref{order of modulation parameter for k>1}.
Therefore, we conclude
\begin{align} \label{estimates for e_2(1)d_ye(1)}
e^{-b} \epsilon_2(1) \, \partial_y \epsilon_2(1) = 
\begin{cases}
\displaystyle \frac{d}{ds} \left\{ O(|b_1|^{7/2}) \right\} + O(|b_1|^{7/2}), & \text{for } k = 1, \\
\displaystyle \frac{d}{ds} \left\{ O(|b_k|^3) \right\} + O(|b_k|^3), & \text{for } k > 1.
\end{cases}
\end{align}

\textbf{6.} Estimates for $\Phi\langle\Lambda\epsilon,\epsilon_{2}\rangle_{b}$: Since $|(b_1)_s| \lesssim |b_1|$ from \eqref{modulation estimates for k=1} and $|b_s| \lesssim b$ from \eqref{hyperparameter}, and using the bounds of \eqref{asymptotic for a, k=1} and \eqref{order of modulation parameter for k>1}, we deduce that for each $k \in \mathbb{N}$,
\[
|\Phi| \lesssim |b|,
\]
where we set $b = b_1$ when $k = 1$.

Furthermore, using the bootstrap bounds on $\mathcal{E}$ from \eqref{bootstrap assumption for k=1} for $k = 1$ and \eqref{bootstrap assumption for k>1} for $k > 1$, we estimate
\begin{equation} \label{estimates for innoflambdaep}
\begin{aligned}
|\Phi \langle \Lambda \epsilon, \epsilon_2 \rangle_{b_1}| 
&\lesssim |b_1| \| \Lambda\epsilon \|_{{L}^2_{b_1}} \sqrt{\mathcal{E}} \lesssim b_1^4, \quad &&\text{for } k = 1, \\
|\Phi \langle \Lambda \epsilon, \epsilon_2 \rangle_{b}| 
&\lesssim |b| \| \Lambda\epsilon \|_{L^2_b} \sqrt{\mathcal{E}} \lesssim b^3, \quad &&\text{for } k > 1.
\end{aligned}
\end{equation}

Therefore, substituting the estimates from \eqref{estimates for H_b(1)},
\eqref{estimates for d_y H_b e}, \eqref{estimates for <e, Lam e>},
\eqref{estimatesfor<H_{b}F,e_{2}>}, \eqref{estimates for e_2(1)d_ye(1)},
and \eqref{estimates for innoflambdaep} into the energy identity \eqref{derivative of E},
we obtain the desired pointwise bound on the energy $\mathcal{E}(s)$ for each $k \in \mathbb{N}$. This completes the proof of Proposition~\ref{pointwise energy control}.
\end{proof}

We are now in a position to close the bootstrap argument. To this end, we apply a continuity argument to complete the proof of Proposition~\ref{bootstrap assumptions}. The pointwise energy estimates established in Proposition~\ref{pointwise energy control} serve as a key input, allowing us to propagate the bootstrap bounds globally in the variable $s$.

\begin{proof}[Proof of Proposition~\ref{bootstrap assumptions}]
\textbf{1.} For $k=1$: From the modulation equation for $k = 1$ given in \eqref{modulation estimates for k=1}, we have
\[
\left| (b_1)_s + \lambda_1 b_1 + \sqrt{2\lambda_1} b_1^2 \right| \lesssim |b_1|^{5/2}.
\]
Possibly after reducing $b^*(1) > 0$, while keeping the same notation, this implies the exponential bound
\begin{equation}\label{b_1 bootstrap}
|b_1(s)| \leq 2b^*(1) e^{-\lambda_1 s}.
\end{equation}
Recall from the energy estimate \eqref{Energyestimatek=1} for $k = 1$ that
\[
\frac{1}{2} \frac{d}{ds} \left( \mathcal{E} + O(|b_1|^{7/2}) \right) + \lambda_2 \left( \mathcal{E} + O(|b_1|^{7/2}) \right) \lesssim |b_1|^{7/2}.
\]
Integrating this differential inequality in time from \( s = 0 \) to \( s \), we obtain
\begin{align*}
\mathcal{E}(s) + O(|b_1(s)|^{7/2}) 
&\leq \left( \mathcal{E}(0) + O(|b_1(0)|^{7/2}) \right) e^{-2\lambda_2 s} \\
&\quad + C e^{-2\lambda_2 s} \int_0^s (b^*)^{7/2} e^{\left(-\frac{7}{2}\lambda_1 + 2\lambda_2\right)\tau} \, d\tau,
\end{align*}
for some constant $C > 0$. Since $3\lambda_1 < 2\lambda_2$, the exponent in the integral is negative. Therefore, by possibly choosing $b^*(1) > 0$ even smaller and letting $C_1 > 0$ be sufficiently large, we deduce
\[
\mathcal{E}(s) \leq \left( (b^*(1))^{1/2}\tilde{C} + 1 \right) b_1^3(s) \leq \frac{C_1}{2} b_1^3(s),
\]
where $\tilde{C}>0$ is a large universal constant. This closes the bootstrap bound for $\mathcal{E}$ in the case $k = 1$, completing the continuity argument.

\textbf{2.} For $k>1$ : 
Similarly, we obtain the asymptotic behavior of $b_k(s)$ from the modulation estimates for $k > 1$ given in \eqref{modulation esitmates for k>1}:
\[
\left|(b_k)_s + \lambda_k b_k + (-1)^{k+1} \sqrt{2\lambda_k} \, b_k^2\right| \lesssim b b_k.
\]
Possibly after reducing $b^*(k) > 0$ (while keeping the same notation), this yields the exponential bound
\begin{equation}\label{b_k bootstrap}
|b_k(s)| \leq 2 b^*(k) e^{-\lambda_k s}.
\end{equation}
Now, from the energy bound \eqref{Energyestimatek>1}, we have
\[
\frac{1}{2} \frac{d}{ds} \left( \mathcal{E} + O(|b_k|^3) \right) + \lambda_{k+1} \mathcal{E} \lesssim b_k^3.
\]
Integrating in time from $0$ to $s$ yields
\begin{align*}
\mathcal{E}(s) &\leq \left\{ \mathcal{E}(0) + O(|b_k(0)|^3) \right\} e^{-2\lambda_{k+1}s} \\
&\quad + C (b^*(k))^3 e^{-2\lambda_{k+1}s} \int_0^s e^{(-3\lambda_k + 2\lambda_{k+1})\tau} \, d\tau + O(|b_k(s)|^3),
\end{align*}
for some constant $C > 0$. Then, possibly after further reducing $b^*(k) > 0$, and choosing $C_k$ sufficiently large, we obtain
\[
\mathcal{E}(s) \leq \left( b^*(k) \tilde{C} + 1 \right) b^2(s) \leq \frac{C_k}{2} b^2(s),
\]
where $\tilde{C} > 0$ is a large universal constant. This closes the bootstrap estimate for $\mathcal{E}$ in the case $k > 1$.

\begin{remark}\label{motivation for introduce b(s)}
We emphasize the necessity of introducing the auxiliary parameter $b(s)$ in \eqref{hyperparameter} for the case $k > 1$. For sufficiently large $k$, the eigenvalues of the radial Dirichlet Laplacian satisfy $3\lambda_k > 2\lambda_{k+1}$, which prevents us from closing the energy bootstrap argument using the same strategy as in the case $k=1$. To avoid this, we introduce $b(s)$ as a comparison parameter slightly smaller than $b_k(s)$, tailored to each $k>1$. This adjustment ensures the integrability of error terms in the energy inequality and allows the continuity argument to be implemented.
\end{remark}

We now prove the existence of initial conditions for $(b_j(0))_{1 \leq j \leq k-1}$ that ensure the trapping condition $|b_j(s)| \lesssim b(s)$ for all $1 \leq j \leq k-1$, via a topological argument.

From the modulation estimates for $k>1$ in \eqref{modulation esitmates for k>1} and the exponential decay \eqref{b_k bootstrap}, we obtain
\[
\sum_{j=1}^{k-1} \left| (b_j)_s + \lambda_j b_j \right| \leq C (b^*(k))^2 e^{-2\lambda_k s}
\]
for some universal constant $C > 0$. 

Injecting the definition of $V_j$ from \eqref{def of V_j} into the evolution equation for $b_j(s)$, we deduce
\begin{equation*}
|V_j'(s) + (\eta_k + \lambda_j - \lambda_k) V_j(s)| \leq C (b^*(k))^2 e^{-\lambda_k s}, \quad \text{for } 1 \leq j \leq k-1.
\end{equation*}

To close the bootstrap estimate for $V_j$, we proceed by contradiction using the Brouwer Fixed Point Theorem. Suppose that for every choice of initial data $\{V_j(0)\}_{1 \leq j \leq k-1}$ in the ball of radius $D_k$, the maximal time $s^*$ for which the bootstrap assumption holds is finite. That is, there exists $s^* < \infty$ such that
\begin{equation*}
\sum_{j=1}^{k-1} |V_j(s^*)|^2 = D_k^2.
\end{equation*}

Since we chose $\eta_k > 0$ so that $\lambda_k - \lambda_j - \eta_k > \frac{1}{2}$ for all $1 \leq j \leq k-1$ (which is possible by the definition of $\eta_k$ in \eqref{def of V_j}), we differentiate and estimate the growth of $\sum_{j=1}^{k-1} |V_j(s)|^2$ at $s = s^*$:
\begin{align*}
\frac{1}{2} \frac{d}{ds} \left( \sum_{j=1}^{k-1} |V_j|^2 \right)(s^*) 
&= \sum_{j=1}^{k-1} V_j'(s^*) V_j(s^*) \\
&= \sum_{j=1}^{k-1} \left\{ (\lambda_k - \lambda_j - \eta_k) V_j^2(s^*) + C V_j(s^*) (b^*(k))^2 e^{-\lambda_k s^*} \right\} \\
&\geq \left( \frac{1}{2} D_k^2 - C (b*(k))^2 D_k \right).
\end{align*}

Now, by choosing $D_k>0$ sufficiently large so that $C (b^*(k))^2 < D_k / 4$, the time derivative becomes strictly positive. This implies that $s^*$ is indeed an exit time, contradicting the assumption of maximality.

Next, we define a continuous map 
\[
T: B_{D_k}(0) \subset \mathbb{R}^{k-1} \to B_{D_k}(0), \quad \tilde{V}(0) \mapsto \tilde{V}(s^*(\tilde{V}(0))) := (V_1(s^*), \ldots, V_{k-1}(s^*)),
\]
which maps the initial data to the exit point in the boundary sphere $D_k \mathbb{S}^{k-1}$.

By construction, $T$ is continuous and equals the identity on the boundary $\partial B_{D_k}(0) = D_k \mathbb{S}^{k-1}$, contradicting the Brouwer Fixed Point Theorem. Therefore, there must exist initial data $\{V_j(0)\}_{1 \leq j \leq k-1}$ such that the bootstrap is maintained globally in the variable $s$, completing the proof.

\end{proof}

\subsection{Proof of Theorem\ref{main theorem}}\label{main theorem proof}
We are now in a position to prove Theorem~\ref{main theorem}. In this final step, we will choose suitable initial data and integrate the relation $-\lambda_s / \lambda$ in time to derive the precise asymptotic behavior of $\lambda(s)$, thereby completing the proof of the main result. 

\begin{proof}[Proof of Theorem~\ref{main theorem}]
We now consider initial data $v_0$ for the renormalized equation~\eqref{renormalised stefan} at $s=0$.
For $k=1$, we take $v_0$ of the form given by~\eqref{profile set up k=1} at $s=0$, with the initial condition satisfying~\eqref{initial conditions for k=1}. For $k > 1$, we consider $v_0$ of the form given by~\eqref{profile set up k>1} at $s=0$, with initial condition satisfying~\eqref{initial conditions for k>1}, and where the coefficients $(b_j(0))_{1 \leq j \leq k-1}$ are chosen to satisfy the trapping condition~\eqref{bootstrap assumption for k>1}.
These choices ensure that the solution satisfies the assumptions of Proposition~\ref{bootstrap assumptions}. Based on this, we derive the asymptotic behavior of $\lambda(t)$ as $t \to \infty$, and show how the solution exhibits either melting or freezing dynamics depending on the parity of $k \in \mathbb{N}$ and the sign of the initial value $b_k(0)$.

For each $k \in \mathbb{N}$, we define a function $g_k : [0, \infty) \to \mathbb{R}$ by
\begin{equation} \label{proof of thm: def of g_k'}
    g_k'(s) \coloneqq \frac{(b_k)_s}{(-1)^{k+1} \sqrt{2\lambda_k} b_k^2 + \lambda_k b_k} + 1,
\end{equation}
with initial condition
\[
g_k(0) = \frac{1}{\lambda_k} \log\left| \frac{\sqrt{2\lambda_k} b_k(0)}{(-1)^{k+1} \sqrt{2\lambda_k} b_k(0) + \lambda_k} \right|.
\]

We also introduce a unified notation for the error terms in the modulation estimates, $c_k : [0, \infty) \to \mathbb{R}$, defined by
\[
c_k(s) \coloneqq \begin{cases}
e^{-\frac{3}{2} \lambda_1 s}, & \text{for } k = 1, \\
\frac{e^{-\lambda_k s}}{s + 1}, & \text{for } k > 1.
\end{cases}
\]

Then, using the modulation estimates for $b_k$, given by \eqref{modulation estimates for k=1} when $k = 1$ and \eqref{modulation esitmates for k>1} when $k > 1$, we obtain the bound
\[
|g_k'(s)| \lesssim c_k(s), \quad \text{for all } s \in [0, \infty),
\]
for each $k \in \mathbb{N}$. Since $c_k \in L^1([0, \infty))$, this implies that $g_k'$ is integrable over $[0, \infty)$, and thus by the Fundamental Theorem of Calculus, the limit
\[
\lim_{s \to \infty} g_k(s) = L_0 \in \mathbb{R}
\]
exists. Consequently, we deduce that there exists a non-zero constant $c \in \mathbb{R}$ such that
\[
\lim_{s \to \infty} e^{\lambda_k s} b_k(s) = c.
\]
Indeed, by integrating both sides of \eqref{proof of thm: def of g_k'} from $0$ to $s$, we obtain
\[
g_k(s) = \frac{1}{\lambda_k} \log\left| \frac{\sqrt{2\lambda_k} e^{\lambda_k s} b_k(s)}{(-1)^{k+1} \sqrt{2\lambda_k} b_k(s) + \lambda_k} \right|.
\]
Since $\lim_{s \to \infty} g_k(s) \in \mathbb{R}$ and $\lim_{s \to \infty} b_k(s) = 0$ from \eqref{b_1 bootstrap} and \eqref{b_k bootstrap}, the limit above implies
\[
\lim_{s \to \infty} e^{\lambda_k s} b_k(s) = c \neq 0,
\]
as claimed. Therefore, we conclude that
\[
L_0 = \frac{1}{\lambda_k} \log\left| \frac{c\sqrt{2\lambda_k}}{\lambda_k} \right|,
\]
and hence
\begin{equation}\label{thm1.1 : eq1}
    \frac{|\sqrt{2\lambda_{k}}b_{k}(s)|}{(-1)^{k+1}\sqrt{2\lambda_{k}}b_{k}(s)+\lambda_{k}} 
    = \left|\frac{c}{\sqrt{\lambda_k/2}}\right| e^{-\lambda_k s + \lambda_k (g_k(s) - L_0)}.
\end{equation}
From the bootstrap bounds \eqref{b_1 bootstrap} for $k=1$ and \eqref{b_k bootstrap} for $k > 1$, we further have
\[
\left| \frac{c}{\sqrt{\lambda_k/2}} \right| \leq \frac{2b^*(k)}{\sqrt{\lambda_k/2}} \ll 1.
\]

On the other hand, we define a function $h_k : [0, \infty) \rightarrow \mathbb{R}$ by
\begin{equation}\label{proof of thm: def of h_k'}
    h_k'(s) \coloneqq a(s) - (-1)^{k+1}\sqrt{2\lambda_k}b_k(s),
\end{equation}
with initial condition $h_k(0) = 0$. From the boundary condition \eqref{boundaryconditions}, we deduce the bound
\[
    |h_k'(s)| \lesssim c_k(s), \quad \text{for all } s \in [0, \infty).
\]
Hence, $h_k'$ is integrable on $[0, \infty)$, and by the Fundamental Theorem of Calculus, there exists a constant $L_1 \in \mathbb{R}$ such that
\[
    \lim_{s \to \infty} h_k(s) = L_1.
\]
Since $a = -\lambda_s/\lambda$, by integrating both sides of \eqref{proof of thm: def of h_k'} from $0$ to $s$, we obtain
\[
    h_k(s) = -\log \lambda(s) + (-1)^{k+1}\sqrt{2\lambda_k} \int_0^s b_k(\tau) \, d\tau.
\]
Because $b_k$ is integrable on $[0, \infty)$ by \eqref{b_1 bootstrap} for $k = 1$ and \eqref{b_k bootstrap} for $k > 1$, there exists a limit $\lambda_{\infty} > 0$ such that $\lim_{s \to \infty} \lambda(s) = \lambda_{\infty}$. Consequently, we obtain
\[
    L_1 = -\log \lambda_{\infty} + (-1)^{k+1}\sqrt{2\lambda_k} \int_0^{\infty} b_k(\tau) \, d\tau,
\]
and therefore,
\begin{equation}\label{thm1.1 : eq2}
    \log \lambda(s) = \log \lambda_{\infty} + (-1)^k \sqrt{2\lambda_k} \int_s^{\infty} b_k(\tau) \, d\tau + (L_1 - h_k(s)).
\end{equation}

Note that both $|L_1 - h_k(s)|$ and $|g_k(s) - L_0|$ are of order $O\left(\int_s^{\infty} c_k(\tau)\, d\tau\right)$. 
For $k = 1$, since $c_1(s) = e^{-\frac{3}{2} \lambda_1 s}$, we directly compute
\[
\int_s^{\infty} c_1(\tau)\, d\tau = \int_s^{\infty} e^{-\frac{3}{2} \lambda_1 \tau}\, d\tau = O(c_1(s)).
\]
For $k > 1$, since $c_k(s) = \frac{e^{-\lambda_k s}}{s+1}$, an integration by parts gives
\[
\int_s^{\infty} \frac{e^{-\lambda_k \tau}}{\tau+1}\, d\tau = \frac{1}{\lambda_k} \cdot \frac{e^{-\lambda_k s}}{s+1} - \frac{1}{\lambda_k} \int_s^{\infty} \frac{e^{-\lambda_k \tau}}{(\tau+1)^2}\, d\tau = O(c_k(s)).
\]
Therefore, for each $k \in \mathbb{N}$, we obtain the estimate
\begin{equation}\label{integration of c_k}
    \int_s^{\infty} c_k(\tau)\, d\tau = O(c_k(s)).
\end{equation}

We now specify the initial value $b_k(0)$, thus determining the initial data for the renormalized equation \eqref{renormalised stefan}, for each $k \in \mathbb{N}$. As observed in \eqref{thm1.1 : eq1} and \eqref{thm1.1 : eq2}, the long-term behavior of the radius function $\lambda(s)$ is governed by the index $k$ and the sign of $b_k(0)$. More precisely, the solution lies in the \emph{melting regime} if either $k$ is odd and $b_k(0) > 0$, or $k$ is even and $b_k(0) < 0$. In all other cases, the solution lies in \emph{freezing regime}. This classification is determined by the sign of the exponential correction term in the asymptotic expansion of $\lambda(t)$. In what follows, we compute the asymptotic behavior of $\lambda(t)$ in each regime separately. We fix $b_{k}(0)$ to be $|b_{k}(0)| \leq b^*(k)$.

\textbf{1.} \emph{Melting regime}:
For $k \in 2\mathbb{N} - 1$, we choose $b_k(0) > 0$, and for $k \in 2\mathbb{N}$, we choose $b_k(0) < 0$. Then, from the asymptotic law for $b_k$ in \eqref{thm1.1 : eq1} and the estimate \eqref{integration of c_k}, we obtain
\[
b_k(s) = (-1)^{k+1} \frac{\sqrt{\lambda_k/2}}{-1 + C e^{\lambda_k s + O(c_k(s))}},
\]
where we denote $C \coloneqq \left|\frac{\sqrt{\lambda_k/2}}{c}\right| \gg 1$. Integrating $b_k$ from $s$ to $\infty$, we have
\begin{equation} \label{meltinglaw}
\begin{aligned}
    \int_{s}^{\infty} b_k(\tau) \, d\tau 
    &= (-1)^{k+1} \sqrt{\frac{\lambda_k}{2}} \int_{s}^{\infty} \left( \frac{1}{C e^{\lambda_k \tau} - 1} + O(c_k(\tau)) \right) d\tau \\
    &= (-1)^{k+1} \frac{1}{\sqrt{2\lambda_k}} \log\left( \frac{C e^{\lambda_k s}}{C e^{\lambda_k s} - 1} \right) + O\left( \int_s^{\infty} c_k(\tau) \, d\tau \right).
\end{aligned}
\end{equation}
Substituting \eqref{meltinglaw} into \eqref{thm1.1 : eq2}, and using the estimate \eqref{integration of c_k}, we deduce the asymptotic behavior of $\lambda(s)$ in the melting regime:
\begin{equation} \label{lambdalawformelt}
\lambda(s) = \lambda_\infty \left( 1 - \left| \frac{c}{\sqrt{\lambda_k/2}} \right| e^{-\lambda_k s} \right) e^{O(c_k(s))}.
\end{equation}

\textbf{2.} \emph{Freezing regime}:
For $k \in 2\mathbb{N} - 1$, we choose $b_k(0) < 0$, and for $k \in 2\mathbb{N}$, we choose $b_k(0) > 0$. From the asymptotic law for $b_k$ in \eqref{thm1.1 : eq1}, we obtain
\[
b_k(s) = (-1)^k \frac{\sqrt{\lambda_k/2}}{1 + C e^{\lambda_k s + O(c_k(s))}},
\]
where $C \coloneqq \left|\frac{\sqrt{\lambda_k/2}}{c}\right| \gg 1$. By integrating $b_k$ from $s$ to $\infty$ and using the estimate \eqref{integration of c_k}, we find
\begin{equation}\label{freezinglaw}
\int_s^{\infty} b_k(\tau) \, d\tau = (-1)^{k+1} \frac{1}{\sqrt{2\lambda_k}} \log\left( \frac{C e^{\lambda_k s}}{1 + C e^{\lambda_k s}} \right) + O(c_k(s)).
\end{equation}

Substituting \eqref{freezinglaw} into \eqref{thm1.1 : eq2} and using the estimate \eqref{integration of c_k}, we deduce the asymptotic behavior of $\lambda(s)$ in the freezing regime:
\begin{equation}\label{lambdalawforfreez}
\lambda(s) = \lambda_\infty \left(1 + \left|\frac{c}{\sqrt{\lambda_k/2}}\right| e^{-\lambda_k s} \right) e^{O(c_k(s))}.
\end{equation}

We now change the time variable from $s$ to $t$. From the asymptotic expansions \eqref{lambdalawformelt} and \eqref{lambdalawforfreez}, together with the estimate $|c_k(s)| = o(e^{-\lambda_k s})$, we obtain
\[
\lambda^2(s) = \lambda_{\infty}^2 \left(1 + 2(-1)^{\mu} C e^{-\lambda_k s} + O(c_k(s)) \right),
\]
where $\mu = 1$ in the melting regime and $\mu = 0$ in the freezing regime. Since the relation between the renormalized time $s$ and physical time $t$ is given by $ds/dt = 1/\lambda^2(s)$, we integrate to find
\[
t = \int_0^t d\tau = \int_0^s \lambda^2(\tau) \, d\tau = \lambda_{\infty}^2 s + \tilde{X}(u_0) + \Theta(e^{-\lambda_k s}),
\]
for some constant $\tilde{X}(u_0)$ depending on the choice of initial data for \eqref{stefanradial}.

As a result, we derive the asymptotic behavior of $\lambda(t)$ in physical time:
\[
\lambda(t) = \lambda_{\infty} + B_k(u_0) e^{-\frac{\lambda_k}{\lambda_{\infty}^2} t + f_k(t)},
\]
for some function $f_k : [0, \infty) \to \mathbb{R}$ satisfying $f_k(0) = 0$ and $f_k(t) = o_{t \to \infty}(1)$, and some constant $B_k(u_0) \in \mathbb{R}$ depending on the initial data. In particular, $B_k(u_0) > 0$ in the freezing regime and $B_k(u_0) < 0$ in the melting regime.
Finally, using the normalization $\lambda(0) = 1$, the asymptotic law for $\lambda(t)$ becomes
\[
\lambda(t) = \lambda_{\infty} + (1 - \lambda_{\infty}) e^{-\frac{\lambda_k}{\lambda_{\infty}^2} t + o_{t \to \infty}(1)}.
\]

We now determine $\lambda_{\infty}$ in terms of the initial data $u_0$ and hence prove \eqref{terminal value}. Since $u$ solves the Stefan problem \eqref{stefanradial}, applying a divergence theorem yields the following conservation law:
\begin{equation} \label{conservationlaw}
\frac{d}{dt} \left( \int_{|x| \leq \lambda(t)} u(t,x) \, dx + \pi \lambda^2(t) \right) = 0.
\end{equation}
We claim that
\[
\lim_{t \to \infty} \| u(t) \|_{L^1(\{ |x| \leq \lambda(t) \})} = 0.
\]
In fact, by renormalization $u(t,x) = v(s, y)$ with $y = x/\lambda(t)$ and $s = \int_0^t \lambda^2(\tau) \, d\tau$, we compute
\[
\| u(t) \|_{L^1(\{ |x| \leq \lambda(t) \})} = \lambda(t)^2 \| v(s) \|_{L^1(B_1(0))} \lesssim \| v(s) \|_{L^2_b} \lesssim |b_k(s)|.
\]
Since $\lim_{s \to \infty} b_k(s) = 0$ and $\lambda(t) \to \lambda_{\infty} > 0$ as $t \to \infty$, the claim follows.

Using this in \eqref{conservationlaw}, and the normalization $\lambda(0) = 1$, we conclude
\begin{equation*}
\lambda_{\infty} = \sqrt{1 + \frac{1}{\pi} \int_{\Omega_0} u_0(x) \, dx}.
\end{equation*}
We choose a small constant $\delta_k > 0$ such that $\delta_k \lesssim b^*(k)$, and take the initial data $u_0$ for \eqref{stefanradial} that satisfy
\[
\| u_0 \|_{L^1(B_1(0))} \leq \delta_k.
\]
This ensures that the assumptions of Proposition~\ref{bootstrap assumptions} are satisfied, allowing full analysis to take place.

Moreover, by the construction of the solution for $k = 1$ and the topological argument used to initialize $(b_j(0))_{1 \leq j \leq k-1}$ for $k > 1$, we observe that the resulting solution is stable in codimension $k - 1$ for each $k \in \mathbb{N}$.
\end{proof}

\appendix

\section{Near invertibility of $\mathcal{H}_b - \lambda_k$, Proof of Lemma~\ref{lemma:near invertibility}}\label{appendix}

\begin{proof}[Proof of Lemma~\ref{lemma:near invertibility}]
We fix $K \in \mathbb{N}$ and choose $b^*(K)>0$ sufficiently small, and assume that $|b| < b^*(K)$. We may further reduce $b^*(K)$, if necessary, and will indicate this explicitly each time, while continuing to use the same notation. We proceed with our proof by employing a Lax–Milgram type argument.

Let $\mathcal{C} \subset H_{b}^{1}$ denote the orthogonal complement of $\{\eta_{j}\}_{j=1}^{k}$ with respect to the $\langle \cdot, \cdot \rangle_b$-inner product:
\[
\mathcal{C} \coloneqq \left\{ w \in H_{b}^{1} \;\middle|\; \langle w, \eta_j \rangle_b = 0 \text{ for } 1 \leq j \leq k \right\}.
\]
Define a Dirichlet energy functional $\mathcal{F} : \mathcal{C} \to \mathbb{R}$ associated with the operator $\mathcal{H}_b - \lambda_k-f$ by
\[
\mathcal{F}(w) \coloneqq \frac{1}{2} \int_0^1 |\partial_y w|^2 \rho_b y\,dy - \frac{\lambda_k}{2} \int_0^1 w^2 \rho_b y\,dy - \langle f, w \rangle_b, \quad w \in \mathcal{C}.
\]
We first claim that
\begin{equation}\label{eq:finiteness of I_b}
I_b \coloneqq \inf_{w \in \mathcal{C}} \mathcal{F}(w) > -\infty.
\end{equation}
To establish \eqref{eq:finiteness of I_b}, we claim that for all $w \in \mathcal{C}$, the following spectral gap estimate holds:
\begin{equation}\label{eq:7}
    \| \partial_y w \|_{L^2_b}^2 \geq (\lambda_{k+1} + O(|b|)) \| w \|_{L^2_b}^2.
\end{equation}
To prove \eqref{eq:7}, we first observe that since $w \in H^1_b$ and $w(1) = 0$ in the trace sense, we have $w \in H_0^1(\mathbb{R}) \subset L^2_0$. Therefore, $w$ admits the Fourier--Bessel expansion
\begin{equation*}
w = \sum_{j=1}^{\infty} \langle w, \eta_j \rangle_0 \, \eta_j.
\end{equation*}

Possibly after reducing $b^*(K) > 0$ (while keeping the same notation), we may assume that $|1 - \rho_b(y)| \leq |b|$ for all $|y| \leq 1$ and $|b| < b^*(K)$. Then, we have
\begin{equation}\label{aa:2}
\langle w,\eta_{j}\rangle_{0} = \int_{0}^{1} w \eta_{j} \rho_{b} y\,dy + \int_{0}^{1} w \eta_{j} (1 - \rho_{b}) y\,dy 
\leq |\langle w, \eta_{j} \rangle_{b}| + |b| \|w\|_{L^{2}_0},
\end{equation}
where we used the fact that $\|\eta_j\|_{L^2_0} = 1$. 
Since $\langle w, \eta_j \rangle_b = 0$ for $1 \leq j \leq k$, it follows from \eqref{aa:2} that $|\langle w, \eta_j \rangle_0| \leq |b| \|w\|_{L^2_0}$ for $1 \leq j \leq k$. 
Combining this with Parseval's identity, we obtain
\begin{equation}\label{eq:5}
\|w\|_{L^2_0}^2 = \sum_{j=1}^{\infty} |\langle w, \eta_j \rangle_0|^2 
\leq k b^2 \|w\|_{L^2_0}^2 + \sum_{j = k+1}^{\infty} |\langle w, \eta_j \rangle_0|^2.
\end{equation}
Noting that $\eta_j$ is the $L^2_0$-normalized eigenfunction of the unperturbed problem \eqref{eq:1}, and combining \eqref{eq:5}, we obtain the following estimate:
\begin{equation*}
    \begin{aligned}
        \norm{\partial_{y}w}_{L^{2}_0}^{2} & =\sum_{j=1}^{\infty}|\langle w,\eta_{j}\rangle_{0}|^{2}\lambda_{j}\\
        & \geq\lambda_{k+1}\sum_{j=k+1}^{\infty}|\langle w,\eta_{j}\rangle_{0}|^{2}-\lambda_{k}\sum_{j=1}^{k}|\langle w,\eta_{j}\rangle_{0}|^{2}\\
        & \geq\{\lambda_{k+1}(1-kb^{2})-\lambda_{k}kb^{2}\}\norm{w}_{L^{2}}^{2}\\
        & =\{\lambda_{k+1}-(\lambda_{k+1}+\lambda_{k})kb^{2}\}\norm{w}_{L^{2}}^{2}.
    \end{aligned}
\end{equation*}
Using the relations $\norm{\partial_y w}_{L^2_b} = (1 + O(b))\norm{\partial_y w}_{L^2_0}$ and $\norm{w}_{L^2_b} = (1 + O(b))\norm{w}_{L^2_0}$, we conclude the desired spectral gap estimate \eqref{eq:7}.
Possibly reducing $b^*(K) > 0$, we obtain
\begin{equation*}
    \mathcal{F}(w) \geq \frac{1}{2}(\lambda_{k+1}-\lambda_k - \frac{1}{2}+O(|b|))\norm{w}_{L^2_b}^2 - \frac{1}{2}\norm{f}_{L^2_b}^2 \geq -\frac{1}{2}\norm{f}_{L^2_b}^2 > -\infty.
\end{equation*}
This proves \eqref{eq:finiteness of I_b}.
We now establish the existence of a function $u \in \mathcal{C}$ such that $\mathcal{F}(u) = \inf_{w \in \mathcal{C}} \mathcal{F}(w)$. This will be achieved by solving the minimization problem for $\mathcal{F}$ subject to the constraint $w \in \mathcal{C}$. 

Let $\{w_n\}_{n \in \mathbb{N}} \subset \mathcal{C}$ be a minimizing sequence for $\mathcal{F}$, i.e., 
\[
\lim_{n \to \infty} \mathcal{F}(w_n) = \inf_{w \in \mathcal{C}} \mathcal{F}(w).
\]
Then, from \eqref{eq:7} we see that $\norm{w_n}_{H^1_b}$ is uniformly bounded. We then extract a subsequence (still denoted by $\{w_n\}$) that converges weakly in $H_b^1$ to some function $u$. Note that a weak convergence gives $\langle u,\eta_j\rangle_b=0$ for $1\leq j\leq k$. Furthermore, $\mathcal{F}$ is lower semi-continuous in $\mathcal{C}$ due to $f\in L^{\infty}(B_{1}(0))$. Hence, $u \in \mathcal{C}$ and $I_b = \mathcal{F}(u)$. To prove the uniqueness of the minimizer of $\mathcal{F}$ in $\mathcal{C}$, we use the convexity of $\mathcal{F}$. Indeed, for $u_1, u_2 \in \mathcal{C}$, we have
\begin{align*}
\mathcal{F}(\tau u_1+(1-\tau)u_2)<\tau\mathcal{F}(u_1)+(1-\tau)\mathcal{F}(u_2),\quad \text{for }\tau\in(0,1).
\end{align*}

\begin{align*}
 \mathcal{F}&(\tau u_1+(1-\tau)u_2)\\
 & = \int_{0}^{1}\left\{ \frac{1}{2}|\tau\partial_{y}u_1+(1-\tau)\partial_{y}u_2|^{2}-\frac{\lambda_{k}}{2}|\tau u_1+(1-\tau)u_2|^{2}-(\tau u_1+(1-\tau)u_2)f(y)\right\} \rho_{b}y\ dy\\
 & = \frac{\tau^{2}}{2}\int_{0}^{1}|\partial_{y}u_1|^{2}\rho_{b}y\ dy+\frac{(1-\tau)^{2}}{2}\int_{0}^{1}|\partial_{y}u_2|^{2}\rho_{b}y\ dy+\tau(1-\tau)\int_{0}^{1}\partial_{y}u_1\partial_{y}u_2\rho_{b}y\ dy\\
 & \quad-\frac{\lambda_{k}}{2}\tau^{2}\int_{0}^{1}u_1^{2}\rho_{b}y\ dy-\frac{\lambda_{k}}{2}(1-\tau)^{2}\int_{0}^{1}u_2^{2}\rho_{b}y\ dy-\lambda_{k}\tau(1-\tau)\int_{0}^{1}u_1u_2\rho_{b}y\ dy\\
 & \quad -\int_{0}^{1}\ (\tau u_1+(1-\tau)u_2)f(y)\rho_{b}y\ dy\\
 &= \tau\mathcal{F}(u_1)+(1-\tau)\mathcal{F}(u_2)-\tau(1-\tau)\left\{ X+Y-\int_{0}^{1}\left(\partial_{y}u_1\partial_{y}u_2-\lambda_{k}u_1u_2\right)\rho_{b}y\ dy\right\} ,
\end{align*}
where we denoted by 
\begin{align*}
 & X \coloneqq \frac{1}{2}\int_{0}^{1}|\partial_{y}u_1|^{2}\rho_{b}y\ dy-\frac{\lambda_{k}}{2}\int_{0}^{1}u_1^{2}\rho_{b}y\ dy,\\
 & Y \coloneqq \frac{1}{2}\int_{0}^{1}|\partial_{y}u_2|^{2}\rho_{b}y\ dy-\frac{\lambda_{k}}{2}\int_{0}^{1}u_2^{2}\rho_{b}y\ dy.
\end{align*}
Since $u_1 - u_2 \in \mathcal{C}$ and the spectral gap estimate \eqref{eq:7} hold for all elements in $\mathcal{C}$, it follows that
\begin{align*}
\mathcal{F}(\tau u_1+(1-\tau)u_2)= & \tau\mathcal{F}(u_1)+(1-\tau)\mathcal{F}(_2)-\tau(1-\tau)\{\norm{\partial_{y}u_1-\partial_{y}u_2}_{L_{b}^{2}}^{2}-\lambda_{k}\norm{u_1-u_2}_{L_{b}^{2}}^{2}\}\\
< & \tau\mathcal{F}(u_1)+(1-\tau)\mathcal{F}(u_2).
\end{align*}
Therefore, we conclude that there exists a unique weak solution $u \in \mathcal{C}$ such that $(\mathcal{H}_b - \lambda_k)u - f = 0$ in the orthogonal complement of $\text{span}\{\eta_1, \cdots, \eta_k\}$ with respect to the inner product $\langle \cdot, \cdot \rangle_b$. By the Lagrange multiplier Theorem, there exist real numbers $\mu_j \in \mathbb{R}$ for $j = 1, \cdots, k$ such that
\begin{equation*}
(\mathcal{H}_b - \lambda_k)u - f = \sum_{j=1}^{k} \mu_j \eta_j.
\end{equation*}
Using the orthogonality conditions $\langle f, \eta_j \rangle_b = 0$ for $1 \leq j \leq k$, we obtain
\begin{equation}
\langle (\mathcal{H}_b - \lambda_k)u, \eta_j \rangle_b = \sum_{i=1}^{k} \mu_i \langle \eta_i, \eta_j \rangle_b. \label{eq:8}
\end{equation}
Since $u \in \mathcal{C}$, we have $\langle u, \eta_j \rangle_b = 0$ for $1 \leq j \leq k$ and $u(1) = 0$. 
Integrating by parts, we compute
\begin{equation}\label{eq:9}
\begin{aligned}
\langle \mathcal{H}_b u, \eta_j \rangle_b 
&= \left[ -\rho_b y \partial_y u \, \eta_j \right]_{0}^{1} + \int_{0}^{1} \partial_y u \, \partial_y \eta_j \, \rho_b y \, dy\\
&= \left[ u \, \partial_y \eta_j \, \rho_b y \right]_{0}^{1} + \int_0^1 (\mathcal{H}_b \eta_j) u \, \rho_b y \, dy  \\
&= \langle u, \mathcal{H}_b \eta_j \rangle_b = b \langle u, \Lambda \eta_j \rangle_b.
\end{aligned}
\end{equation}
Combining \eqref{eq:8} and \eqref{eq:9}, we determine the Lagrange multipliers $\mu_j$ as
\begin{equation*}
(\mu_j)_{1 \leq j \leq k} = b M_{b,k}^{-1} \left( \langle u, \Lambda \eta_i \rangle_b \right)_{1 \leq i \leq k},
\end{equation*}
which shows that $u$ is the unique solution to \eqref{eq:3}.

It remains to prove the estimate \eqref{estimatesofinvertibility}.
Taking the $\langle \cdot, \cdot \rangle_b$-inner product of both sides of equation \eqref{eq:3} with the solution $u$, and applying the spectral estimate \eqref{eq:7}, we obtain
\begin{align*}
\norm{f}_{L_b^2} \norm{u}_{L_b^2} 
& \geq \langle f, u \rangle_b \\
& = \langle (\mathcal{H}_b - \lambda_k) u, u \rangle_b \\
& = \norm{\partial_y u}_{L_b^2}^2 - \lambda_k \norm{u}_{L_b^2}^2 \\
& \gtrsim \norm{u}_{L_b^2}^2,
\end{align*}
which implies the desired estimate. 
\end{proof}

\providecommand{\bysame}{\leavevmode\hbox to3em{\hrulefill}\thinspace}
\providecommand{\MR}{\relax\ifhmode\unskip\space\fi MR } 
\providecommand{\MRhref}[2]{%
  \href{http://www.ams.org/mathscinet-getitem?mr=#1}{#2}
} 

\bibliographystyle{amsplain}
\bibliography{Stefan}

\providecommand{\bysame}{\leavevmode\hbox to3em{\hrulefill}\thinspace}
\providecommand{\MR}{\relax\ifhmode\unskip\space\fi MR }
\providecommand{\MRhref}[2]{%
  \href{http://www.ams.org/mathscinet-getitem?mr=#1}{#2}
}
\providecommand{\href}[2]{#2}
\begin{thebibliography}{10}

\bibitem{Athanasopoulos1}
Ioannis Athanasopoulos, Luis~A. Caffarelli, and Sandro Salsa, \emph{Regularity of the free boundary in parabolic phase-transition problems}, Acta Math. \textbf{176} (1996), no.~2, 245--282. \MR{1397563}

\bibitem{Athanasopoulos2}
\bysame, \emph{Phase transition problems of parabolic type: flat free boundaries are smooth}, Comm. Pure Appl. Math. \textbf{51} (1998), no.~1, 77--112. \MR{1486632}

\bibitem{BiernatHMF}
Pawe{\l} Biernat and Yukihiro Seki, \emph{Type {II} blow-up mechanism for supercritical harmonic map heat flow}, Int. Math. Res. Not. IMRN (2019), no.~2, 407--456. \MR{3903563}

\bibitem{Caffarelli1}
Luis~A. Caffarelli, \emph{The regularity of free boundaries in higher dimensions}, Acta Math. \textbf{139} (1977), no.~3-4, 155--184. \MR{454350}

\bibitem{Caffarelli2}
\bysame, \emph{Some aspects of the one-phase {S}tefan problem}, Indiana Univ. Math. J. \textbf{27} (1978), no.~1, 73--77. \MR{466965}

\bibitem{Caff-Evans}
Luis~A. Caffarelli and Lawrence~C. Evans, \emph{Continuity of the temperature in the two-phase {S}tefan problem}, Arch. Rational Mech. Anal. \textbf{81} (1983), no.~3, 199--220. \MR{683353}

\bibitem{Caffa-Friedman}
Luis~A. Caffarelli and Avner Friedman, \emph{Continuity of the temperature in the {S}tefan problem}, Indiana Univ. Math. J. \textbf{28} (1979), no.~1, 53--70. \MR{523623}

\bibitem{Caffalleli3}
Luis~A. Caffarelli and Sandro Salsa, \emph{A geometric approach to free boundary problems}, Graduate Studies in Mathematics, vol.~68, American Mathematical Society, Providence, RI, 2005. \MR{2145284}

\bibitem{L.Chayes.I.Kim1}
Lincoln Chayes and Inwon~C. Kim, \emph{A two-sided contracting {S}tefan problem}, Comm. Partial Differential Equations \textbf{33} (2008), no.~10-12, 2225--2256. \MR{2479285}

\bibitem{L.Chayes.I.Kim2}
\bysame, \emph{The supercooled {S}tefan problem in one dimension}, Commun. Pure Appl. Anal. \textbf{11} (2012), no.~2, 845--859. \MR{2861812}

\bibitem{L.Chayes}
Lincoln Chayes and Glen Swindle, \emph{Hydrodynamic limits for one-dimensional particle systems with moving boundaries}, Ann. Probab. \textbf{24} (1996), no.~2, 559--598. \MR{1404521}

\bibitem{Choikim2}
Sunhi Choi and Inwon~C. Kim, \emph{Regularity of one-phase {S}tefan problem near {L}ipschitz initial data}, Amer. J. Math. \textbf{132} (2010), no.~6, 1693--1727. \MR{2766502}

\bibitem{ChoiKim2024}
Sunhi Choi, Inwon~C. Kim, and Young-Heon Kim, \emph{Existence for the supercooled stefan problem in general dimensions}, 2024, arXiv:2402.17154 [math.AP].

\bibitem{ChuKim2025}
Raymond Chu, Inwon~C. Kim, Young‑Heon Kim, and Kyeongsik Nam, \emph{The nonlocal {S}tefan problem via a {M}artingale transport}, Probab. Theory Related Fields (2025), to appear.

\bibitem{CollotWave}
Charles Collot, \emph{Type {II} blow up manifolds for the energy supercritical semilinear wave equation}, Mem. Amer. Math. Soc. \textbf{252} (2018), no.~1205, v+163. \MR{3778126}

\bibitem{CollotKS}
Charles Collot, Tej-Eddine Ghoul, and Nader Masmoudi, \emph{Refined description and stability for singular solutions of the 2{D} {K}eller-{S}egel system}, Comm. Pure Appl. Math. \textbf{75} (2022), no.~7, 1419--1516. \MR{4438587}

\bibitem{CollotAniso}
Charles Collot, Frank Merle, and Pierre Rapha\"el, \emph{Strongly anisotropic type {II} blow up at an isolated point}, J. Amer. Math. Soc. \textbf{33} (2020), no.~2, 527--607. \MR{4073868}

\bibitem{Delaure}
Fran\c{c}ois Delarue, Sergey Nadtochiy, and Mykhaylo Shkolnikov, \emph{Global solutions to the supercooled {S}tefan problem with blow-ups: regularity and uniqueness}, Probab. Math. Phys. \textbf{3} (2022), no.~1, 171--213. \MR{4420299}

\bibitem{HouNguyenKS}
\bysame, \emph{Axisymmetric type ii blowup solutions to the three dimensional keller-segel system}, 2025.

\bibitem{Friedman1}
Avner Friedman, \emph{The {S}tefan problem in several space variables}, Trans. Amer. Math. Soc. \textbf{133} (1968), 51--87. \MR{227625}

\bibitem{Friedman3}
Avner Friedman and David~S. Kinderlehrer, \emph{A one phase {S}tefan problem}, Indiana Univ. Math. J. \textbf{24} (1974/75), no.~11, 1005--1035. \MR{385326}

\bibitem{GaoKrieger}
Can Gao and Joachim Krieger, \emph{Optimal polynomial blow up range for critical wave maps}, Commun. Pure Appl. Anal. \textbf{14} (2015), no.~5, 1705--1741. \MR{3359541}

\bibitem{GhoulWaveMap}
Tarek Ghoul, Slim Ibrahim, and Van~T. Nguyen, \emph{Construction of type {II} blowup solutions for the 1-corotational energy supercritical wave maps}, J. Differential Equations \textbf{265} (2018), no.~7, 2968--3047. \MR{3812220}

\bibitem{Had-Raph}
Mahir Had\v{z, }i\'c and Pierre Rapha\"el, \emph{On melting and freezing for the 2{D} radial {S}tefan problem}, J. Eur. Math. Soc. (JEMS) \textbf{21} (2019), no.~11, 3259--3341. \MR{4012340}

\bibitem{Had-Shk}
Mahir Had\v{z, }i\'c and Steve Shkoller, \emph{Well-posedness for the classical {S}tefan problem and the zero surface tension limit}, Arch. Ration. Mech. Anal. \textbf{223} (2017), no.~1, 213--264. \MR{3590373}

\bibitem{Hanzawa}
Ei-ichi Hanzawa, \emph{Classical solutions of the {S}tefan problem}, Tohoku Math. J. (2) \textbf{33} (1981), no.~3, 297--335. \MR{633045}

\bibitem{M.Herrero}
Miguel~A. Herrero and Juan J.~L. Vel\'azquez, \emph{Singularity formation in the one-dimensional supercooled {S}tefan problem}, European J. Appl. Math. \textbf{7} (1996), no.~2, 119--150. \MR{1388108}

\bibitem{Herrero-Velaz}
\bysame, \emph{On the melting of ice balls}, SIAM J. Math. Anal. \textbf{28} (1997), no.~1, 1--32. \MR{1427725}

\bibitem{KamenomostskayaS.L.}
Sofya~L. Kamenomostskaya, \emph{On {S}tefan's problem}, Mat. Sb. (N.S.) \textbf{53(95)} (1961), 489--514. \MR{141895}

\bibitem{Kim.I.C}
Inwon~C. Kim, \emph{Uniqueness and existence results on the {H}ele-{S}haw and the {S}tefan problems}, Arch. Ration. Mech. Anal. \textbf{168} (2003), no.~4, 299--328. \MR{1994745}

\bibitem{Kim-Kim2024}
Inwon~C. Kim and Young-Heon Kim, \emph{The {S}tefan problem and free targets of optimal {B}rownian martingale transport}, Ann. Appl. Probab. \textbf{34} (2024), no.~2, 2364--2414. \MR{4728172}

\bibitem{Kinderlehrer}
David~S. Kinderlehrer and Louis Nirenberg, \emph{Regularity in free boundary problems}, Ann. Scuola Norm. Sup. Pisa Cl. Sci. (4) \textbf{4} (1977), no.~2, 373--391. \MR{440187}

\bibitem{Kinderlehrer2}
\bysame, \emph{The smoothness of the free boundary in the one phase {S}tefan problem}, Comm. Pure Appl. Math. \textbf{31} (1978), no.~3, 257--282. \MR{480348}

\bibitem{KriegerBeyond}
Joachim Krieger, Shuang Miao, and Wilhelm Schlag, \emph{A stability theory beyond the co-rotational setting for critical wave maps blow up}, 2024.

\bibitem{Solonnikov}
Olga~A. Lady\v{z}enskaja, Vsevolod~A. Solonnikov, and Nina~N. Ural\cprime~ceva, \emph{Linear and quasilinear equations of parabolic type}, Translations of Mathematical Monographs, vol. Vol. 23, American Mathematical Society, Providence, RI, 1968, Translated from the Russian by S. Smith. \MR{241822}

\bibitem{Meirmanov}
Anvarbek~M. Meirmanov, \emph{The {S}tefan problem}, De Gruyter Expositions in Mathematics, vol.~3, Walter de Gruyter \& Co., Berlin, 1992, Translated from the Russian by Marek Niezg\'odka and Anna Crowley, With an appendix by the author and I. G. G\"otz. \MR{1154310}

\bibitem{Raphaelwave}
Pierre Rapha\"el and R\'emi Schweyer, \emph{Stable blowup dynamics for the 1-corotational energy critical harmonic heat flow}, Comm. Pure Appl. Math. \textbf{66} (2013), no.~3, 414--480. \MR{3008229}

\bibitem{Sherman}
Bernard Sherman, \emph{A general one-phase {S}tefan problem}, Quart. Appl. Math. \textbf{28} (1970), 377--382. \MR{282082}

\bibitem{WaldronHMF}
Alex Waldron, \emph{Strict type-{II} blowup in harmonic map flow}, Proc. Amer. Math. Soc. \textbf{151} (2023), no.~11, 4893--4907. \MR{4634891}

\bibitem{3d-stefan}
Chencheng Zhang, \emph{On melting for the 3{D} radial {S}tefan problem}, Calc. Var. Partial Differential Equations \textbf{64} (2025), no.~3, Paper No. 101, 67. \MR{4878964}

\end{thebibliography}

\end{document}